\documentclass[12pt]{article}
\usepackage{bm}  %bold math symbols%
\usepackage{amsmath,amscd,amssymb,amsthm,mathrsfs,amsfonts,layout,indentfirst,graphicx,caption,mathabx,  tikz, stmaryrd,appendix}

\usepackage{palatino}  %template
\usepackage{slashed} % Dirac operator
\usepackage{mathrsfs} % Enable using \mathscr
\usepackage{lipsum}
\let\OLDthebibliography\thebibliography
\renewcommand\thebibliography[1]{
	\OLDthebibliography{#1}
	\setlength{\parskip}{0pt}
	\setlength{\itemsep}{2pt} 
}

\allowdisplaybreaks  %allow aligns to break between pages
\usepackage{latexsym}
\usepackage{chngcntr}
\usepackage[colorlinks,linkcolor=blue,anchorcolor=blue,linktocpage]{hyperref}
\hypersetup{citecolor=[rgb]{0,0.5,0}}

\usepackage{fullpage}
\counterwithin{figure}{section}

\theoremstyle{definition}
\newtheorem{df}{Definition}[section]

\newtheorem{rem}[df]{Remark}

\newtheorem{cv}[df]{Convention}

\theoremstyle{plain}
\newtheorem{thm}[df]{Theorem}
\newtheorem{pp}[df]{Proposition}
\newtheorem{co}[df]{Corollary}
\newtheorem{lm}[df]{Lemma}

\newcommand{\tr}{\mathrm{t}} %transpose
\newcommand{\End}{\mathrm{End}} %endomorphism
\newcommand{\id}{\mathrm{id}}
\newcommand{\Hom}{\mathrm{Hom}}

\newcommand{\Span}{\mathrm{Span}}
\newcommand{\wtd}{\widetilde}

\newcommand{\Ker}{\mathrm{Ker}}
\newcommand{\Conj}{\mathrm{Conj}}
\newcommand{\Ad}{\mathrm{Ad}}
\newcommand{\gh}{{\widehat{\mathfrak g}}}
\newcommand{\gk}{\mathfrak g}
\newcommand{\ph}{{\widehat{\mathfrak p}}}
\newcommand{\pk}{\mathfrak p}
\newcommand{\hk}{\mathfrak h}
\newcommand{\hh}{\widehat{\mathfrak h}}
\newcommand{\Lo}{{\Lambda^\circ}}
\newcommand{\ih}{i\hk_{\mathbb R}}

\newcommand{\tk}{\mathfrak t}
\newcommand{\ek}{\mathfrak e}
\newcommand{\fk}{\mathfrak f}
\newcommand{\kk}{\mathfrak k}
\newcommand{\bk}[1]{\langle {#1}\rangle}
\newcommand{\wt}{\widetilde}
\newcommand{\so}{\mathfrak {so}}
\newcommand{\gl}{\mathfrak {gl}}
\newcommand{\mc}{\mathcal}
\newcommand{\smp}{\mathfrak {sp}}

\numberwithin{equation}{section}

\title{Energy bounds condition for intertwining operators of type $B$, $C$, and $G_2$ unitary affine vertex operator algebras}
\author{{\sc Bin Gui}\\
	{\small Department of Mathematics, Rutgers University}\\
	{\small bin.gui@rutgers.edu}
}
\date{}
\begin{document}\sloppy % avoid stretch into margins
	\pagenumbering{arabic}
	%\pagenumbering{gobble}
	%\newpage
	%\setcounter{page}{1}
	
	\setcounter{section}{0}
	\maketitle

\begin{abstract}
The energy bounds condition for intertwining operators of unitary rational vertex operator algebras (VOAs) was studied, first by A.Wassermann for type $A$ affine VOAs, and later by T.Loke for $c<1$ Virasoro VOAs, and by V.Toledano-Laredo for type $D$ affine VOAs. In this paper, we extend their results to affine VOAs of type $B$, $C$, and $G_2$. As a consequence, the modular tensor categories of these unitary vertex operator algebras are unitary.
\end{abstract}

	\tableofcontents
\section*{Introduction}

In algebraic quantum field theory, energy bounds condition is a very natural functional analytic property satisfied by many examples. This property  guarantees not only the preclosedness of smeared   field operators, but also the strong commutativity of causally disjoint ones when the energy bounds are small. In 2d conformal field theory (CFT), the energy bounds condition has been extensively studied for chiral fields (vertex operators), from which one can construct conformal nets, i.e., nets of von Neumann algebras on the unit circle $S^1$ satisfying important properties such as Haag duality. See for instance \cite{GW84,BS90,CKLW18}.

It is  quite  natural to think about extending the analysis of energy bounds condition to the \emph{intertwining operators} of a unitary rational vertex operator algebra (VOA), since intertwining operators are the building blocks (the chiral halves) of full CFT field operators \cite{MS88,HK07}, the energy bounds of which are clearly expected. Such analysis was initiated in the seminal work of A.Wassermann \cite{Was98},  followed by the works of T.Loke \cite{Loke94} and V.Toledano-Laredo \cite{TL04}. In these works, certain energy bounds were established  for a generating set of intertwining operators of unitary affine type $A$ and type $D$ VOAs, as well as $c<1$ unitary Virasoro VOAs. These results are crucial for constructing finite index subfactors \cite{Jon83,Lon89}, unitary fusion categories, unitary modular tensor categories \cite{Xu00,KLM01,KL04}, and hence unitary 3d topological quantum field theories  \cite{Tur94} under the settings of algebraic quantum field theories and operator algebras. Moreover, these functional analytic properties also have important applications in VOAs, as can be seen in our recent works  \cite{Gui19,Gui17} which used these properties to prove the unitarity of the representation tensor categories of many unitary rational VOAs.

Our goal in this article is to prove the energy bounds condition for the intertwining operators of unitary affine VOAs of type $B$, $C$, and the exceptional type $G_2$. As a consequence, we show the unitarity of the modular tensor categories for these VOAs (theorem \ref{lb49}). Further applications to conformal nets and operator algebras will be given in future work \cite{Gui18}. Roughly speaking, if $W_i,W_j,W_k$ are unitary representations of a unitary VOA $V$,  $\mathcal Y$ is a type ${W_k\choose W_i~W_j}={k\choose i~j}$ intertwining operator $V$, and $w^{(i)}$ is a homogeneous vector of $W_i$, we say that $\mathcal Y$ satisfies $r$-th order energy bounds ($r\geq0$) if the smeared intertwining opeartor $$\mathcal Y(w^{(i)},f)=\oint_{S^1}\mathcal Y(w^{(i)},z)f(z)\frac{dz}{2i\pi}$$ is bounded by $L_0^r$, and is also continuous with respect to the smooth function $f$ under some Sobolev norm on $C^\infty(S^1)$. $1$-st order energy bounds are especially important, as they imply the strong commutativity of two unbounded operators commuting adjointly on a common invariant core. So, when possible, we also determine those intertwining operators satisfying $1$-st order energy bounds. Although the concrete analysis of each type is different from another, the main idea, which is called compression principle in our paper, is the same as those in \cite{Was98} and also in \cite{Loke94,TL04}, which we now describe. 

\subsubsection*{Compression principle}

In general, energy bounds are easier to establish for vertex operators than for intertwining operators. For example, for free fermion vertex superalgebras, affine VOAs, and Virasoro VOAs, one can use the  Lie (super)algebra relations to deduce the $0$-th or $1$-st order energy bounds for primary fields (cf. \cite{BS90} section 2), and use normal ordering (or more generally, the Jacobi identity in the form \eqref{eq24}) to prove the energy bounds condition for any vertex operators. The energy bounds condition is also preserved when passing to vertex subalgebras. So one can indeed establish energy-bounds  for many examples of vertex operators. Unfortunately, these methods  seem to not work for intertwining operators, as intertwining operators don't in general form an  algebra.\footnote{Perhaps the only exceptions are the intertwining operators of a Heisenberg VOA, which form a so called generalized vertex algebra \cite{DL93}.} 

However, since we already know that many vertex operators are energy-bounded, we can try to embed  our VOA $V$ in a larger unitary vertex (super)algebra $\widetilde V$ whose vertex operators are energy bounded, choose  irreducible (unitary) $V$-submodules $W_i,W_j,W_k$ of the action $V\curvearrowright \widetilde V$, and restrict the vertex operator $\widetilde Y$ of $\widetilde V$ onto $W_i,W_j,W_k$. Then, up to a multiplication by a monomial of the formal variable $x$, this restricted vertex operator, which is clearly energy-bounded, is an intertwining operator $V$. If one can show that any irreducible intertwining operator $\mathcal Y$ of $V$ can be realized as such a compression of an energy-bounded vertex operator, or more generally, as a compression of an intertwining operator $\mathfrak Y$ of a larger $\widetilde V$, the energy-bounds condition of which has already been established, then we can prove the energy bounds condition for $\mathcal Y$. The precise statement of this compression principle, at least for affine VOAs, can be found in theorem \ref{lb10}.

\subsubsection*{Examples}

So far, compression principle is the only essential way to obtain energy bounds condition for intertwining operators. Let us now explain how this principle can be applied to all known examples. First, for the intertwining operators of  unitary $c<1$ Virasoro VOAs, one can deduce the energy bounds condition from that of the intertwining operators of affine $\mathfrak {sl}_2$-VOAs, since the coset construction of Goddard-Kent-Olive \cite{GKO86} realizes any $c<1$ Virasoro VOA as a unitary vertex subalgebras of a tensor product of two affine $\mathfrak {sl}_2$ VOAs \cite{Loke94}.

Unitary Affine VOAs are the examples that we are especially interested in in this paper. Let $\gk$ be a unitary complex simple Lie algebra, and $l=1,2,\dots$. Then it is well-known that the level $l$ affine $\gk$-VOA $V^l_\gk$ can be embedded as a unitary vertex subalgebra of $(V^1_\gk)^{\otimes l}$ (proposition \ref{lb33}). Moreover, when $\gk$ is a classical Lie algebra or of type $G_2$, it can be shown that any irreducible intertwining operator of $V^l_\gk$ is a compression of one of $(V^1_\gk)^{\otimes l}$. So for such $\gk$, it suffices to prove the energy bounds condition for the intertwining operators of the level $1$ affine VOA $V^1_\gk$.\footnote{In fact, we expect that this statement holds for all simple Lie algebras except type $E_8$. See our discussion in chapter \ref{lb38}.} And when the level is $1$, all classical Lie types can be reduced to type $D$ due to the diagonal embeddings $\mathfrak{sl}_n\subset\mathfrak{so}_{2n}$ and $\mathfrak{sp}_{2n}\subset\mathfrak{s0}_{4n}$, and the obvious embedding $\mathfrak{so}_{2n+1}\subset\mathfrak{so}_{2n+2}$, whose levels (i.e, Dynkin indexes, see section \ref{lb39}) are all $1$. For the type $G_2$ simple Lie algebra $\gk_2$, one should consider the embedding $\gk_2\subset\ek_8$\footnote{The author would like to thank Marcel Bischoff for suggesting thinking about the embedding $\gk_2\subset\ek_8$. All the embeddings mentioned in this paragraph were also suggested in \cite{Was90}.} whose Dynkin index is also $1$, and realize the intertwining operators of $V^1_{\gk_2}$ as the compressions of the vertex operator (but not more generally intertwining operators) of $V^1_{\ek_8}$, the energy bounds condition of which is clearly known as mentioned previously.

It seems that the only possible way of solving the type $D$ level $1$ case is through Heisenberg and lattice VOAs, as was carried out in \cite{TL04} chapters V, VI. As every classical Lie type can (or should) be reduced to this case, we feel it necessary to explain the key ideas of the proof. This method indeed works for any simply laced Lie algebra. So let us assume that $\gk$ is of type $A$, $D$, or $E$. Then by Frenkel-Kac construction \cite{FK80} (see also \cite{Kac98} section 5.6), the corresponding level $1$ affine VOA is unitarily equivalent to the even lattice VOA whose lattice is the root lattice of $\gk$. So the problem is turned to show the energy bounds condition for (the vertex operators and) the intertwining operators of even lattice VOAs. 

It is impossible to discuss lattice VOAs without a full understanding of their ``Cartan subalgebras", the Heisenberg VOAs. Let $\hk$ be a unitary (finite-dimensional) abelian Lie algebra. One can define, for any $l>0$, a level $l$ Heisenberg VOA $V^l_\hk$ in a similar way as affine VOAs are defined. The most striking differences between Heisenberg and affine VOAs are: (1) $V^l_\hk$ is irrelevant to the level $l$. So we can always choose $l=1$. (2) The level adds no constraints on the irreducible representations of $V^1_\hk$. So the equivalence classes of the unitary irreducible representations of $V^1_\hk$ are in one to one correspondence with the elements of the dual space  $(\ih)^*$ of the self-adjoint (real) subspace $\ih$ of $\hk$. So the Heisenberg VOA $V^1_\hk$ is not a rational VOA. However, we can still study the ``tensor category" of $V^1_\hk$ in a broader sense, as the fusion rules are always finite, and the intertwining operators of $V^1_\hk$ actually satisfy some very simple fusion and (abelian) braid relations (cf. \cite{DL93}, see also chapter \ref{lb34}). With the help of this tensor categorical structure, one can ``rationalize" $V^1_\hk$ by extending it to a larger VOA using (irreducible) intertwining operators of $V^1_\hk$ whose charge spaces, source spaces, and target spaces all correspond to elements inside a non-degenerate even lattice $\Lambda$ in $\ih\simeq(\ih)^*$. This VOA $V_\Lambda$ is the lattice VOA for $\Lambda$. The explicit form of an intertwining operator $\mathfrak Y$ of $V_\Lambda$ can be determined, which shows that $\mathfrak Y$ must be a (globally infinite but locally finite) ``homogeneous" sum of intertwining operators of the Heisenberg VOA $V^1_\hk$, so that the analytic behaviors of $\mathfrak Y$ depends homogeneously on those of the intertwining operators of $V^1_\hk$ (theorem \ref{lb40}). Therefore, to study the energy bounds condition for the intertwining operators of lattice VOAs, it suffices to do this for any irreducible intertwining operator $\mathcal Y$ of $V^1_\hk$. We assume that $\mathcal Y$ is of type $\nu\choose\lambda~\mu$, where $\lambda,\mu,\nu\in(\ih)^*$. Then, when $(\lambda|\lambda)=1$, $\mathcal Y$ can be regarded as a compression of the vertex operator of the lattice vertex superalgebra $V_{\mathbb Z^n}$ for the integral lattice $\mathbb Z^n$. Since, by Boson-Fermion correspondence, $V_{\mathbb Z^n}$ is equivalent to the vertex superalgebra for the $n$-dimensional free fermion, one can show show that $\mathcal Y$ satisfies $0$-th order energy bounds. When $(\lambda|\lambda)$ takes general values, a tensor product argument due to \cite{TL04} can be applied to prove the energy-bounds condition for $\mathcal Y$. Thus the energy bounds conditions for Heisenberg VOAs and lattice VOAs are established. See chapter \ref{lb34} for more details.

\subsubsection*{Outline of this paper}

Chapter 1 is aimed to  fix the notations and provide the necessary backgrounds for establishing and understanding the compression principle,  and for analyzing the energy bounds condition for intertwining operators. In particular, we state the basic theories of unitary VOAs, affine Lie algebras, and finite-dimensional simple Lie algebras with a unifying language emphasizing their unitary properties. So, for example, the invariant \emph{inner products} of complex simple Lie algebras are introduced at the very beginning of our discussion of Lie algebras, and the real structures of root systems and weight spaces become very natural from this point of view. On the other hand, the notion of invariant \emph{bilinear forms} is completely avoided in our paper, as the author believes that this concept is, at some point, inconvenient for one to understand the unitary structures. In some sense our philosophy is similar to that of \cite{Was10}. 

Chapter 2 gives a detailed account of the compression principle. In section 2.1 we discuss how a given unitary affine VOA $V^l_\gk$ can be realized as a unitary vertex subalgebra of a larger unitary VOA $V$. In section 2.2 we state the compression principle, and show how this principle can be used to reduce  higher level problems to  level $1$ ones. 

After the presentation of general theories, in chapters 3, 4, 5 we apply the compression principle to establish the energy bounds condition for the intertwining operators of type $B$, $C$, and $G_2$ affine Lie algebras. Conclusions and future perspectives are given in chapter 6. 

As mentioned earlier, the results on all classical Lie type affine VOAs depend on those of type $D$, which in turn depend on those of the intertwining operators of Heisenberg and lattice VOAs. Although the latter is implicitly contained in \cite{TL04}, we feel it necessary to give a brief account of this theory in our paper, since, after all, the results in \cite{TL04} are explicitly stated only for the lattice VOAs $V_\Lambda$'s where $\Lambda$ is the root lattice of a simply laces Lie algebra. Moreover, \cite{TL04} uses the notion of ``primary fields" rather than intertwining operators, and the conditions that primary fields need to satisfy are, at first glance, quite weaker than those on intertwining operators.\footnote{They are indeed equivalent at least for affine VOAs, but the proofs are non-trivial. See \cite{Gui17} section 8.2 for an explanation of this issue, and the reference therein.} The reason that we avoid primary fields in our paper is clear: it is hard to generalize them onto VOAs beyond affine and Virasoro ones. Therefore, as a general theory of even lattice VOAs, the account in \cite{TL04} chapters V, VI is incomplete, and should be accompanied by the results in \cite{DL93}. We will discuss this in the appendix chapter A.

\subsubsection*{Acknowledgment}
This project was initiated when the author was  in Vanderbilt university. The author would like to thank Vaughan Jones for introducing him to this topic,  and for the constant support during his research.

\section{Background}

\subsection{Unitary VOAs}

\subsubsection*{Unitary VOAs and unitary representations}
We assume the reader is familiar with the basic concepts and computations in VOA, such as those in \cite{FHL93}, \cite{Kac98}, or \cite{LL04}.

Let $(V,Y,\Omega,\nu)$ (or simply $V$) be a VOA, where $\Omega$ is the vacuum vector, and  $\nu$ is the conformal vector. For any vector $v\in V$ and a formal variable $x$, we write $Y(v,x)=\sum_{n\in\mathbb Z}Y(v,n)x^{-n-1}$, where each $Y(v,n)\in\End(V)$ is a mode of the field operator $Y(v,x)$. Let $\{L_n=Y(v,n+1)\}$ be the Virasoro operators of $V$. The vector space $V$ has  grading $V=\bigoplus_{n\in\mathbb Z}V(n)$, where $L_0|_{V(n)}=n\cdot\id_{V(n)}$.   We assume that $V$ is of CFT type, i.e., $V(0)=\mathbb C\Omega$ and $V(n)=0$ when $n<0$.

Let $\Theta$ be an antilinear automorphism of $V$. This means that $\Theta:V\rightarrow V$ is antilinear, that $\Omega$ and $\nu$ are fixed by $\Theta$, and that $\Theta Y(v,x)=Y(\Theta v,x)\Theta$ for any $v\in V$. We say that  $(V,Y,\Omega,\nu,\Theta)$ (or simply $V$) is a \textbf{unitary VOA}, if there exists an inner product $\langle\cdot|\cdot\rangle$ on $V$, antilinear on the second variable, such that for any $v,v_1,v_2\in V$,
\begin{align}
\langle Y(v,x)v_1|v_2\rangle=\langle v_1|Y(e^{xL_1}(-x^{-2})^{L_0}\Theta v,x^{-1})v_2\rangle.\label{eq18}
\end{align}
The above relation can be simply written as
\begin{align*}
Y(v,x)^\dagger=Y(e^{xL_1}(-x^{-2})^{L_0}\Theta v,x^{-1}),
\end{align*}
where $Y(v,x)^\dagger$ means the formal adjoint of $Y(v,x)$. In particular, if we let $v=\nu$ and note that $L_1\nu=0$, we have
\begin{align*}
L_n^\dagger=L_{-n}\qquad(n\in\mathbb Z).
\end{align*}

Note that $\Omega\in V$ is a cyclic vector under the action of vertex operators, and that $V(0)$ is spanned by $\Omega$. Therefore, the inner product $\langle \cdot|\cdot\rangle$ is uniquely determined by the positive value $\langle \Omega|\Omega\rangle$. If we \textbf{normalize} $\langle \cdot|\cdot\rangle$ so  that $\langle \Omega|\Omega\rangle=1$, then $\langle \cdot|\cdot\rangle$ is unique, or more precisely, $\langle \cdot|\cdot\rangle$ is uniquely determined by $\Theta$ and $\nu$. We call $\Theta$ the \textbf{PCT operator} of $V$. It is not hard to show that $\Theta$ is  an anti-unitary map (cf. \cite{CKLW18} proposition 5.1). One can also check easily that any unitary CFT type VOA is simple. 

The invariant inner product on $V$ also determines the conformal vector $\nu$ and the PCT operator $\Theta$:
\begin{pp}\label{lb3}
Suppose that $(V,Y,\Omega,\nu,\Theta)$ and $(V,Y,\Omega,\widetilde\nu,\widetilde\Theta)$ are unitary VOAs of CFT type having the same normalized invariant inner product $\bk{\cdot|\cdot}$. Then $\nu=\widetilde\nu$ and $\Theta=\widetilde\Theta$.
\end{pp}
\begin{proof}
(cf. \cite{CKLW18} proposition 4.8.) Let $Y(\widetilde\nu,x)=\sum_{n\in\mathbb Z}\widetilde L_nx^{-n-2}$. If we can show  $L_{-2}=\widetilde L_{-2}$, then $\nu=Y(\nu,-1)\Omega=L_{-2}\Omega=\widetilde L_{-2}\Omega=Y(\widetilde\nu,-1)\Omega=\widetilde\nu$. So $L_n=\widetilde L_n$ for any $n\in\mathbb Z$. Now for any $v,v_1,v_2\in V$,
\begin{align*}
&\langle Y(\Theta^{-1} v,x)v_1|v_2\rangle=\langle v_1|Y(e^{xL_1}(-x^{-2})^{L_0}v,x^{-1})v_2\rangle\\
=&\langle v_1|Y(e^{x \widetilde L_1}(-x^{-2})^{\widetilde  L_0}v,x^{-1})v_2\rangle=\langle Y(\widetilde\Theta^{-1} v,x)v_1|v_2\rangle,
\end{align*}
showing that $\Theta=\widetilde\Theta$.

We now show  $L_{-2}=\widetilde L_{-2}$. By translation property, for any $v\in V$ we have
\begin{align*}
Y(L_{-1}v,x)=\frac d{dx}Y(v,x)=Y(\widetilde L_{-1}v,x).
\end{align*} 
The state-field correspondence now gives $L_1v=\widetilde L_{-1}v$. Hence $L_{-1}=\widetilde L_{-1}$. Taking the adjoint of both sides under the same inner product $\bk{\cdot|\cdot}$ corresponding to  $(V,Y,\Omega,\nu,\Theta)$ and $(V,Y,\Omega,\widetilde\nu,\widetilde\Theta)$ gives us $L_1=\widetilde L_1$. Hence, by Virasoro relation, $2L_0=[L_1,L_{-1}]=[\widetilde L_1,\widetilde L_{-1}]=2\widetilde L_0$. We conclude that $L_n=\widetilde L_n$ when $n=-1,0,1$.

Using the Jacobi identity (e.g. \eqref{eq14}), we have
\begin{align}
[L_{-2},\widetilde L_0]=[Y(\nu,-1),Y(\widetilde\nu,1)]=\sum_{l\geq0}{-1\choose l}Y\big(Y(\nu,l)\widetilde\nu,-l\big)=\sum_{n\geq0}(-1)^lY(L_{l-1}\widetilde\nu,-l).\label{eq16}
\end{align}
We now compute each summand on the right hand side. By translation property, $Y(L_{-1}\widetilde\nu,0)=0$. Also, $Y(L_0\widetilde\nu,-1)=Y(\widetilde L_0\widetilde\nu,-1)=2Y(\widetilde\nu,-1)=2\widetilde L_{-2}$. Since $L_1\widetilde\nu=\widetilde L_1\widetilde\nu=0$, we have $Y(L_1\widetilde\nu,-2)=0$. For any $l\in\mathbb Z$, $L_0L_l\widetilde \nu=[L_0,L_l]\widetilde\nu+L_lL_0\widetilde\nu=-lL_l\widetilde\nu+L_l\widetilde L_0\widetilde\nu=(2-l)L_l\widetilde\nu$. So  $L_l\widetilde\nu\in V(2-l)$. In particular, $L_2\widetilde\nu\in V(0)=\Span_{\mathbb C}(\Omega)$, and hence  $Y(L_2\widetilde\nu,3)$ is proportional to $Y(\Omega,3)$, which is $0$. When $l>2$, $V(2-l)=0$. So $L_l\widetilde\nu=0$. Hence we conclude that the right hand side of equation \eqref{eq16} equals $-2\widetilde L_{-2}$. One the other hand, the Virasoro relation implies that $[L_{-2},\widetilde L_0]=[L_{-2},L_0]=-2L_{-2}$. Therefore $L_{-2}=\widetilde L_{-2}$, which finishes our proof.
\end{proof}
\begin{co}
If $(V,Y,\Omega,\nu,\Theta)$ is a unitary VOA of CFT type, then $\Theta$ is an involution, i.e., $\Theta^2=\id_V$.
\end{co}
\begin{proof}
Using the commuting relation for $L_0$ and $L_1$, it is not hard to show that \eqref{eq18} also holds when $\Theta$ is replaced by $\Theta^{-1}$. So $(V,Y,\Omega,\nu,\Theta^{-1})$ is also a unitary VOA of CFT type, and it has the same normalized invariant inner product as that of $(V,Y,\Omega,\nu,\Theta)$. So by our last proposition, $\Theta=\Theta^{-1}$.
\end{proof}

Let $V$ and $\widetilde V$ be VOAs with vacuum vectors $\Omega,\widetilde\Omega$ and conformal vectors $\nu,\widetilde\nu$ respectively. We also let $Y,\widetilde Y$ be their vertex operators. We say that $V$ is a \textbf{vertex subalgebra} of $\widetilde V$, if there exists an injective linear map $\varphi:V\rightarrow \widetilde V$, such that: \\
(a) $\varphi\Omega=\widetilde\Omega$. \\
(b) $\varphi Y(v,x)=\widetilde Y(\varphi v,x)\varphi$ for any $v\in V$. \\
(c) $\varphi(V)$ is a \emph{graded} subspace of $\widetilde V$.\\ In this case we may identify $V$ with $\varphi(V)$, and regard $V$ itself as a graded subspace of $\widetilde V$. Note that we do \emph{not} require $\varphi\nu=\widetilde \nu$. If $V$ and $\widetilde V$ are unitary VOAs with PCT operators $\Theta,\widetilde\Theta$, we say that $V$ is a \textbf{unitary vertex subalgebra} of $\widetilde V$, if besides conditions (a) (b) (c) we also have: \\
(d) $\varphi\Theta=\widetilde\Theta\varphi$.\\
(e) $\varphi$ is an isometry, i.e., $\langle u|v\rangle=\langle\varphi u|\varphi v\rangle$ for any $u,v\in V$.
\begin{rem}
Assume that $V$ is a unitary vertex subalgebra of $\wt V$. We regard $V$ as a subspace of $\wt V$. So $Y(v,x)=\widetilde Y(v,x)$ for any $v\in V$. If we let $\{L_n \}$ and $\{\wt L_n \}$ be the Virasoro operators of $V$ and $\wt V$ respectively, then, as an easy consequence of translation property and creation property, we must have
\begin{align}
L_{-1}v=\wt L_{-1}v\qquad(\forall v\in V).\label{eq35}
\end{align}
\end{rem}

\begin{cv}
In this article, we always assume, unless otherwise stated, that $V$ is a unitary VOA of CFT type.
\end{cv}
Representations (modules) of $V$ will be denoted by $W_i,W_j,W_k$, etc. If $W_i$ is a $V$ module, we let $Y_i$ be the corresponding vertex operator, i.e., $Y_i(v,x)=\sum_{n\in\mathbb Z}Y_i(v,n)x^{-n-1}$ describes the action of $V$ on $W_i$. Assume that $W_i$ is equipped with an inner product $\langle\cdot|\cdot\rangle$. We say that $(W_i,\langle\cdot|\cdot\rangle)$ (or simply $W_i$) is a \textbf{unitary representation} of $V$, if for any $v\in V$,
\begin{align}
Y_i(v,x)^\dagger=Y_i(e^{xL_1}(-x^{-2})^{L_0}\Theta v,x^{-1}).\label{eq19}
\end{align}
A vector $w^{(i)}\in W_i$  is called \textbf{homogeneous} if it is an eigenvector of $L_0$, and we call its eigenvalue $\Delta_{w^{(i)}}$  the \textbf{conformal weight} of $w^{(i)}$. A homogeneous vector $w^{(i)}\in W_i$ is called \textbf{quasi-primary} if $L_1w^{(i)}=0$. If $v\in V$ is quasi-primary, then equation \eqref{eq19} is simplified to
\begin{align}
Y_i(v,x)^\dagger=(-1)^{\Delta_v}x^{-2\Delta_v}Y_i(\Theta v,x^{-1}).\label{eq29}
\end{align}
Take $v=\nu$, we obtain
\begin{align}
L_n^\dagger=L_{-n}
\end{align}
when acting on $W_i$. So the action of $V$ on $W_i$ restricts to a unitary representation of Virasoro algebra. In particular, \emph{the eigenvalues of $L_0\curvearrowright W_i$ are non-negative} (cf., for example, \cite{Gui19} proposition 1.7). So $L_0$ induces grading $W_i=\bigoplus_{s\geq0}W_i(s)$.

If $W_i$ is a unitary $V$-module, we let the vector space $W_{\overline i}$ be the   complex conjugate of $W_i$, and let $C_i:W_i\rightarrow W_{\overline i}$ be the antilinear isomorphism. We equip $W_{\overline i}$ with an inner product under which $C_i$ becomes anti-unitary. The action of $V$ on $W_i$ is described by the vertex operator
\begin{align*}
Y_{\overline i}(v,x)=C_iY_i(\Theta v,x)C_i^{-1}.
\end{align*}
Clearly $W_i$ is also a unitary representation of $V$. We call $W_{\overline i}$ the \textbf{contragredient representation (module)} of $W_i$.\footnote{Our definition  of contragredient modules is equivalent to the one given in \cite{FHL93}. See \cite{Gui19} equation (1.19).}

Note that the action of $V$ on $V$ is also a unitary representation, called the vacuum representation of $V$, which is also denoted  by $W_0$.  We also identify $W_{\overline 0}$ with $V$ through the isomorphism $C_0\Theta:V=W_0\rightarrow W_{\overline 0}$. So the vacuum representation is self-dual.

\begin{rem}\label{lb6}
Suppose that $V,\widetilde V$ are unitary VOAs of CFT type, $V$ is a unitary vertex subalgebra of $\widetilde V$, and $(W_{\widetilde i},Y_{\wt i})$ is a unitary representation of $\wt V$. Then $(W_{\widetilde i},Y_{\wt i})$ is also a unitary representation of $V$, for the Jacobi identity clearly holds, and the translation property follows from relation \eqref{eq35}.
\end{rem}

\subsubsection*{Intertwining operators}
For any complex vector space $U$ we set
\begin{gather}
U((x))=\bigg\{\sum_{n\in\mathbb Z}u_nx^n:u_n\in U,u_n=0\text{ for sufficiently small }n\bigg\},\\
U\{x \}=\bigg\{\sum_{s\in\mathbb R}u_sx^s:u_s\in U\bigg\}.
\end{gather}
Now let $W_i,W_j,W_k$ be unitary $V$-modules. A type $W_k\choose W_iW_j$ (or type $k\choose i~j$) \textbf{intertwining operator} $\mathcal Y_\alpha$
is a linear map 
\begin{gather*}
W_i\rightarrow(\Hom(W_j,W_k))\{x\},\\
w^{(i)}\mapsto \mathcal Y_\alpha(w^{(i)},x)=\sum_{s\in\mathbb R}\mathcal Y_\alpha(w^{(i)},s) x^{-s-1}\\
\text{ (where $\mathcal Y_\alpha(w^{(i)},s)\in\Hom(W_j,W_k)$)},
\end{gather*}
such that:\\
(a) (Lower truncation) For any $w^{(j)}\in W_j$, $\mathcal Y_\alpha(w^{(i)},s)w^{(j)}=0$ for $s$ sufficiently large.\\
(b) (Jacobi identity) For any $u\in V,w^{(i)}\in W_i,m,n\in\mathbb Z,s\in\mathbb R$, we have
\begin{align}
&\sum_{l\in\mathbb Z_{\geq0}}{m\choose l}\mathcal Y_\alpha\big(Y_i(u,n+l)w^{(i)},m+s-l\big)\nonumber\\
=&\sum_{l\in\mathbb Z_{\geq0}}(-1)^l{n\choose l}Y_k(u,m+n-l)\mathcal Y_\alpha(w^{(i)},s+l)\nonumber\\
&-\sum_{l\in\mathbb Z_{\geq0}}(-1)^{l+n}{n\choose l}\mathcal Y_\alpha(w^{(i)},n+s-l)Y_j(u,m+l).\label{eq5}
\end{align}
(c)	(Translation property) $	\frac d{dx} \mathcal Y_\alpha(w^{(i)},x)=\mathcal Y_\alpha(L_{-1}w^{(i)},x)$.\\
$W_i,W_j,W_k$ are called the \textbf{charge space}, the \textbf{source space}, and the \textbf{target space} of $\mathcal Y_\alpha$ respectively. If all these three $V$-modules are irreducible, then $\mathcal Y_\alpha$ is called \textbf{irreducible}.

By setting $m=0$ in the Jacobi identity, we obtain
\begin{align}
\mathcal Y_\alpha\big(Y_i(u,n)w^{(i)},s\big)
=&\sum_{l\in\mathbb Z_{\geq0}}(-1)^l{n\choose l}Y_k(u,n-l)\mathcal Y_\alpha(w^{(i)},s+l)\nonumber\\
&-\sum_{l\in\mathbb Z_{\geq0}}(-1)^{l+n}{n\choose l}\mathcal Y_\alpha(w^{(i)},n+s-l)Y_j(u,l).\label{eq24}
\end{align}
If we let $n=0$ instead, then the Jacobi identity becomes
\begin{align}
Y_k(u,m)\mathcal Y_\alpha(w^{(i)},s)-\mathcal Y_\alpha(w^{(i)},s)Y_j(u,m)=\sum_{l\in\mathbb Z_{\geq0}}{m\choose l}\mathcal Y_\alpha\big(Y_i(u,l)w^{(i)},m+s-l\big),\label{eq14}
\end{align}
or equivalently,
\begin{align}
Y_k(u,m)\mathcal Y_\alpha(w^{(i)},x)-\mathcal Y_\alpha(w^{(i)},x)Y_j(u,m)=\sum_{l\in\mathbb Z_{\geq0}}{m\choose l}\mathcal Y_\alpha\big(Y_i(u,l)w^{(i)},x\big)x^{m-l}.\label{eq15}
\end{align}
Choose $u=\nu$ (the conformal vector of $V$) and $m=1$. Then we have $[L_0,\mathcal Y_\alpha(w^{(i)},x)]=\mathcal Y_\alpha(L_0w^{(i)},x)+x \mathcal Y_\alpha(L_{-1}w^{(i)},x)$. Therefore, the translation property is equivalent to
\begin{align}
[L_0,\mathcal Y_\alpha(w^{(i)},x)]=\mathcal Y_\alpha(L_0w^{(i)},x)+x\frac d{dx} \mathcal Y_\alpha(w^{(i)},x),\label{eq6}
\end{align}
or, written in terms of modes,
\begin{align}
[L_0,\mathcal Y_\alpha(w^{(i)},s)]=\mathcal Y_\alpha(L_0w^{(i)},s)-(s+1) \mathcal Y_\alpha(w^{(i)},s),\qquad(s\in\mathbb R).\label{eq4}
\end{align}
So, when $w^{(i)}\in W_i$ is homogeneous with conformal weight $\Delta_{w^{(i)}}$,  the linear operator $\mathcal Y_\alpha(w^{(i)},s):W_j\rightarrow W_k$ raises the conformal weights by $\Delta_{w^{(i)}}-s-1$.  From this we conclude that \eqref{eq4} itself implies the lower truncation property. Therefore:

\begin{pp}
A linear map $\mathcal Y_\alpha:W_i\rightarrow(\Hom(W_j,W_k))\{x\}$ is an intertwining operator of $V$ if and only if  relation \eqref{eq6} and the Jacobi identity \eqref{eq5} (both sides of which automatically converge) hold.
\end{pp}
Relation \eqref{eq6} will also be called the translation property.\\

Let $\mathcal V{k\choose i~j}$ be the vector space of type $k\choose i~j$ intertwining operators. Then, clearly, the vertex operator $Y_i$ for $W_i$ should be inside $\mathcal V{i\choose 0~i}$. In particular, $Y\in\mathcal V{0\choose 0~0}$, where $Y$ is the vertex operator describing the action of $V$ on $V$. Define the \textbf{fusion rule}
\begin{align*}
N^k_{ij}=\dim\mathcal V{k\choose i~j}.
\end{align*}
Then clearly $N^i_{0i}=1$ when $W_i$ is irreducible. In particular, $N^0_{00}=1$.

In this article, we are interested in those VOAs having finite fusion rules and finitely many equivalence classes of irreducible representations. For any such VOA $V$, it will also be useful to find a finite set $\mc F$ of irreducible representations which tensor-generates the tensor category of $V$. More precisely:

\begin{df}\label{lb42}
Let  $\{W_i:i\in\mathcal F \}$ be a finite set of irreducible $V$-modules with $\mathcal F$ the set of indexes. With abuse of notation we also let $\mathcal F$ denote this set. Let $\overline{\mathcal F}=\{W_{\overline i}:i\in\mathcal F \}$ be the set of irreducible $V$-modules contragredient to those of $\mathcal F$. We say that $\mathcal F$ is \textbf{generating}, if for any irreducible $V$-module $W_k$, there exist $j_1,i_1,i_2,\dots, i_n\in\mathcal F\cup\overline{\mathcal F}$, and irreducible $V$-modules $W_{j_2},\dots,W_{j_n}$,	such that the vector spaces $\mathcal V{j_2\choose i_1~j_1},\mathcal V{j_3\choose i_2~j_2},\dots,\mathcal V{j_n\choose i_{n-1}~j_{n-1}},\mathcal V{k\choose i_n~j_n}$ are non-trivial.
\end{df}	

For any $\mathcal Y_\alpha\in\mathcal V{k\choose i~j}$, we define its \textbf{adjoint intertwining operator} $(\mathcal Y_\alpha)^\dagger\equiv\mathcal Y_{\alpha^*}\in\mathcal V{j\choose \overline i~k}$, such  that for any $w^{(i)}\in W_i$, (recall that $C_iw^{(i)}\in W_{\overline i}$)
\begin{gather}
\mathcal Y_{\alpha^*}(C_i{w^{(i)}},x)=\mathcal Y_\alpha(e^{xL_1}(e^{-i\pi}x^{-2})^{L_0}w^{(i)},x^{-1})^\dagger.\label{eq10}
\end{gather}
More precisely, for any $w^{(i)}\in W_i,w^{(j)}\in W_j,w^{(k)}\in W_k$,
\begin{align}
\langle \mathcal Y_{\alpha^*}(C_i{w^{(i)}},x)w^{(k)}|w^{(j)}  \rangle=\langle w^{(k)}|\mathcal Y_\alpha(e^{xL_1}(e^{-i\pi}x^{-2})^{L_0}w^{(i)},x^{-1})w^{(j)} \rangle.
\end{align}
The notation $(e^{-i\pi}x^{-2})^{L_0}$ is understood as follows. If $w^{(i)}$ is homogeneous with conformal weight $\Delta_{w^{(i)}}$, then we set $(e^{-i\pi}x^{-2})^{L_0}w^{(i)}=e^{-i\pi\Delta_{w^{(i)}}}x^{-2\Delta_{w^{(i)}}}w^{(i)}$. The general case is defined using linearity. To prove that \eqref{eq10} actually defines an intertwining operator of $V$, one can use the argument in \cite{FHL93} section 5.2.

Note that the unitarity of $W_i$ implies that $Y_i$ equals its adjoint intertwining operator.

One can also define the \textbf{braided intertwining operators} (\cite{FHL93} chapter 5) $\mathcal Y_{B_+\alpha},\mathcal Y_{B_-\alpha}\in\mathcal V{k\choose j~i}$ of $\mathcal Y_\alpha$ using the formula
\begin{align}
\mathcal Y_{B_\pm\alpha}(w^{(j)},x)w^{(i)}=e^{xL_{-1}}\mathcal Y_\alpha(w^{(i)},e^{\pm i\pi}x)w^{(j)}\qquad(\forall w^{(i)}\in W_i,w^{(j)}\in W_j).
\end{align}
The two braided intertwining operators of the vertex operator $Y_i\in\mathcal V{i\choose0~i}$ are equal. We denote it by $\mathcal Y_{\kappa(i)}$, and call it the \textbf{creation operator} of $W_i$. Clearly $\mathcal Y_{\kappa(i)}$ is of type $i\choose i~0$.

It is obvious that the map $\mathcal Y_{\overline\alpha}:W_{\overline i}\otimes W_{\overline j}\rightarrow W_{\overline k}\{x\}$ defined by
\begin{align*}
\mathcal Y_{\overline\alpha}(C_iw^{(i)},x)C_jw^{(j)}=C_k\mathcal Y_\alpha(w^{(i)},x)w^{(j)}\qquad(\forall w^{(i)}\in W_i,w^{(j)}\in W_j)
\end{align*}
is a type $\overline k\choose\overline i~\overline j$ intertwining operator of $V$, called the \textbf{conjugate intertwining operator} of $\mathcal Y_\alpha$.

Now we have linear maps
\begin{align*}
B_\pm:\mathcal V{k\choose i~j}\mapsto\mathcal V{k\choose j~i},\qquad\mathcal Y_\alpha\mapsto\mathcal Y_{B_\pm\alpha },
\end{align*}
and antilinear maps
\begin{gather*}
\Ad:\mathcal V{k\choose i~j}\rightarrow\mathcal V{j\choose\overline i~k},\qquad\mathcal Y_\alpha\mapsto\mathcal Y_{\alpha^*},\\
\Conj:\mathcal V{k\choose i~j}\rightarrow\mathcal V{\overline k\choose\overline i~\overline j},\qquad \mathcal Y_\alpha\mapsto\mathcal Y_{\overline\alpha}.
\end{gather*}
One can show that $B_+$ is the inverse of $B_-$, and $\alpha^{**}=\alpha,\overline{\overline\alpha}=\alpha$. (See also \cite{FZ92}). So these maps are bijective. We conclude
\begin{align}
N^k_{ij}=N^k_{ji}=N^j_{\overline ik}=N^{\overline k}_{\overline i~\overline j}.
\end{align}

\subsubsection*{Tensor products of unitary VOAs}

Let $V^1$ and $V^2$ be two unitary VOAs of CFT type. For $i=1,2$, we let $Y^i,\Omega^i,\nu^i,\Theta^i$ be the vertex operator, the vacuum vector, the conformal vector, and the PCT operator of $V^i$. Consider $V^1\otimes V^2$ as the tensor product of the two pre-Hilbert spaces $V^1$ and $V^2$. We can define a vertex operator $Y$ for $V^1\otimes V^2$ satisfying that  
\begin{align}
Y(v^1\otimes v^2,x)=Y^1(v^1,x)\otimes Y^2(v^2,x)\label{eq1}
\end{align}
for any $v^1\in V^1,v^2\in V^2$. Let
\begin{gather}
\Omega=\Omega^1\otimes\Omega^2,\\
 \nu=\nu^1\otimes \Omega^2+\Omega^1\otimes\nu^2,\label{eq2}\\ \Theta=\Theta^1\otimes\Theta^2
\end{gather}
be the vacuum vector, the conformal vector, and the PCT operator of $V^1\otimes V^2$ respectively. Then $V^1\otimes V^2$ becomes a unitary VOA of CFT type. Note that \eqref{eq1} and \eqref{eq2} imply
\begin{align*}
L_0=L_0\otimes\id_{V^2}+\id_{V^1}\otimes L_0,
\end{align*}
which determines the grading of $V^1\otimes V^2$.

Now let $W_{i_1},W_{i_2}$ be unitary representations of $V^1,V^2$ with vertex operators $Y_{i_1},Y_{i_2}$ respectively, we define a vertex operator $Y_{i_1}\otimes Y_{i_2}\equiv Y_{i_1\otimes i_2}$ of $W_i\otimes W_{i_2}$ to satisfy
\begin{align}
(Y_{i_1}\otimes Y_{i_2})(w^{(i_1)}\otimes w^{(i_2)},x)=Y_{i_1}(w^{(i_1)},x)\otimes Y_{i_2}(w^{(i_2)},x).\label{eq3}
\end{align}
Then $Y_{i_1}\otimes Y_{i_2}$ satisfies Jacobi identity and translation property, making $W_{i_1}\otimes W_{i_2}$ a representation of $V^1\otimes V^2$. (The translation property is easy to check. The  Jacobi identity can be proved by the method mentioned in \cite{FHL93} section 4.6.) It is easy to check that this representation is unitary.

Let $W_{i_1},W_{j_1},W_{j_1}$ (resp. $W_{i_2},W_{j_2},W_{k_2}$) be unitary semisimple representations of $V^1$ (resp. $V^2$). If $\mathcal Y_{\alpha_1}\in\mathcal V{k_1\choose i_1~j_1}$ is an intertwining operator of $V^1$, and $\mathcal Y_{\alpha_2}\in\mathcal V{k_2\choose i_2~j_2}$ is an intertwining operator of $V^2$, we can define a type $k_1\otimes k_2\choose i_1\otimes i_2~j_1\otimes j_2$ intertwining operator of $V^1\otimes V^2$, such that for any $w^{(i_1)}\in W_{i_1},w^{(i_2)}\in W_{i_2}$,
\begin{align}
(\mathcal Y_{\alpha_1}\otimes\mathcal Y_{\alpha_2})(w^{(i_1)}\otimes w^{(i_2)},x)=\mathcal Y_{\alpha_1}(w^{(i_1)},x)\otimes \mathcal Y_{\alpha_2}(w^{(i_2)},x).\label{eq8}
\end{align}
The right hand side of the above relation  clearly converges because this is true when  all the representations involved are irreducible.  

The above discussion can be easily generalized to tensor products of more than two VOAs.

\subsubsection*{Energy bounds condition}

Let $W_i,W_j,W_k$ be unitary $V$-modules, $\mathcal Y_\alpha\in\mathcal V{k\choose i~j}$. Given a homogeneous $w^{(i)}\in W_i,r\geq0$, we say that $\mathcal Y_\alpha(w^{(i)},x)$ satisfies \textbf{$r$-th order energy bounds}, if there exist $M,t\geq0$, such that for any $s\in\mathbb R,w^{(j)}\in W_j$, the mode $\mathcal Y_\alpha(w^{(i)},s)$ satisfies the inequality
\begin{align}
\lVert\mathcal Y_\alpha(w^{(i)},s)w^{(j)}\lVert\leq M(1+|s|)^t\lVert (1+L_0)^rw^{(j)} \lVert.
\end{align}
In the case when the exact value of $r$ is not very important, we just say that $\mathcal Y_\alpha(w^{(i)},x)$ is \textbf{energy-bounded}. 

We say that an intertwining operator $\mathcal Y_\alpha\in\mathcal V{k\choose i~j}$ is \textbf{energy-bounded} if for any $w^{(i)}\in W_i$, $\mathcal Y_\alpha(w^{(i)},x)$ is energy bounded. A unitary $V$-module $W_i$ is \textbf{energy-bounded}, if $Y_i\in\mathcal V{i\choose 0~i}$ is energy-bounded.

We now collect  some useful criteria for  the energy bounds conditions of intertwining operators. Since they are already proved in \cite{Gui19} section 3.1,  we will not prove them again here.

\begin{pp}\label{lb14}
If $\mathcal Y_\alpha(w^{(i)},x)$ satisfies $r$-th order energy bounds, so does $\mathcal Y_{\alpha^*}(C_iw^{(i)},x)$ and $\mathcal Y_{\overline\alpha}(C_iw^{(i)},x)$.
\end{pp}
\begin{pp}
Let $\mathcal Y_{\alpha_1}$ be a type $k_1\choose i_1~j_1$ intertwining operator of $V^1$, and let $\mathcal Y_{\alpha_2}$ be a type $k_2\choose i_2~j_2$ intertwining operator of $V^2$. If $w^{(i_1)}\in W_{i_1},w^{(i_2)}\in W_{i_2}$ are homogeneous, and $\mathcal Y_{\alpha_1}(w^{(i_1)},x)$ and $\mathcal Y_{\alpha_2}(w^{(i_2)},x)$ are energy-bounded, then $(\mathcal Y_{\alpha_1}\otimes\mathcal Y_{\alpha_2})(w^{(i_1)}\otimes w^{(i_2)},x)$ is also energy-bounded.
\end{pp}
\begin{proof}
Although this proposition is not given in \cite{Gui19}, its proof is not hard. One can use, for instance, the same argument as in the proof of \cite{Gui19} proposition 3.5.
\end{proof}

Given a set $E$ of vectors of $V$, we say that $V$ is \textbf{generated by} $E$ (or $V$ is \textbf{generating}), if the vector space $V$ is spanned by vectors of the form $Y(v_1,n_1)\cdots Y(v_k,n_k)v_{k+1}$ ($k\in\mathbb Z_{\geq0},n_1,\dots,n_k\in\mathbb Z,v_1,\dots,v_{k+1}\in E$). 
\begin{pp}\label{lb15}
If $W_i$ is a unitary $V$-module, $E$ is a generating set of homogeneous vectors in $V$, and for any $v\in E$, $Y_i(v,x)$ is energy-bounded, then $W_i$ is energy-bounded.
\end{pp}

\begin{pp}\label{lb17}
If $\mathcal Y_\alpha\in\mathcal V{k\choose i~j}$, $W_i$ is irreducible, $W_j$ and $W_k$ are energy-bounded, and there exists a non-zero homogeneous vector $w^{(i)}\in W_i$, such that $\mathcal Y_\alpha(w^{(i)},x)$ is energy-bounded, then $\mathcal Y_\alpha$ is energy-bounded. 
\end{pp}

\subsection{Unitary Lie algebras}

Let $\mathfrak g$ be a finite-dimensional complex simple (non-abelian) Lie algebra. We give a $*$-structure on $\mathfrak g$. Let $\mathfrak g_{\mathbb R}$ be a compact real form of $\mathfrak g$. So $\mathfrak g_{\mathbb R}$ is a real simple Lie algebra with a negative-definite (symmetric) bilinear form $(\cdot|\cdot)$ (unique up to scalar multiplication), such that
\begin{align*}
([X,Y]|Z)=-(Y|[X,Z])\quad(X,Y,Z\in\mathfrak g_{\mathbb R}),
\end{align*}
and $\mathfrak g=\mathfrak g_{\mathbb R}\oplus_{\mathbb R}i\mathfrak g_{\mathbb R}$. Define an antilinear isomorphism $*:\mathfrak g\rightarrow\mathfrak g$, such that
\begin{align*}
(X+iY)^*=-X+iY\qquad(X,Y\in\mathfrak g_{\mathbb R}).
\end{align*}
Then $*$ is an involution, i.e., $**=\id_{\mathfrak g}$. We also have the relation
\begin{align*}
[X,Y]^*=[Y^*,X^*]\qquad(X,Y\in\mathfrak g).
\end{align*}
The $(\cdot|\cdot)$ on $\mathfrak g_{\mathbb R}$ extends to an \textbf{invariant inner product} on $\mathfrak g$, also denoted by $(\cdot|\cdot)$. The word ``invariant" means that 
\begin{align}
([X,Y]|Z)=(Y|[X^*,Z])\quad(X,Y,Z\in\mathfrak g).\label{eq13}
\end{align}
In other words, under this inner product the adjoint representation of $\mathfrak g$ is unitary. Due to the simplicity of $\mathfrak g$, the invariant inner products on $\mathfrak g$ are unique up to multiplication by scalars. If an invariant inner product is chosen, one can easily check that $*:\mathfrak g\rightarrow\mathfrak g$ is an anti-unitary map under this inner product, i.e.,
\begin{align}
(X|Y)=(Y^*|X^*)\qquad(X,Y\in\gk).
\end{align}

Let $U$ be a (complex) representation of $\mathfrak g$. If $U$ has an inner product $\langle\cdot|\cdot\rangle$, we say that the representation is \textbf{unitary}, if for any $X\in\mathfrak g,u_1,u_2\in U$, 
\begin{align*}
\langle Xu_1|u_2 \rangle=\bk{u_1|X^* u_2}.
\end{align*}

\begin{rem}
	In general, let $\pk$ be a (not necessarily simple) finite dimensional complex Lie algebra with a $*$-structure. So, by definition, $*$ satisfies $**=\id_{\pk}$ and $[X,Y]^*=[Y^*,X^*]$ for any $X,Y\in\pk$. We say that $\pk$ is \textbf{unitary}, if there exists an invariant inner product $(\cdot|\cdot)$ on $\mathfrak p$, under which $*$ is anti-unitary, i.e., $*:\pk\rightarrow\pk$ is an antilinear isomorphism, and $(X^*|Y^*)=(Y|X)$ for any $X,Y\in\pk$. The anti-unitary condition automatically holds when $\mathfrak p$ is semi-simple, i.e., when the center $\mathfrak z$ of $\pk$ is trivial.
	
	A finite dimensional complex simple Lie algebra $\gk$, with the $*$-structure defined above, is a unitary Lie algebra. An abelian finite dimensional Lie algebra $\mathfrak h_0$ with an arbitrary involution $*$ and an arbitrary inner product $(\cdot|\cdot)$ is also a unitary Lie algebra. In general, it is not hard to show that any finite dimensional unitary Lie algebra $\pk$ can be written as an orthogonal direct sum of unitary Lie subalgebras
	\begin{align}
	\pk=\mathfrak z\oplus^\perp\gk_1\oplus^\perp\cdots\oplus^\perp\gk_n,\label{eq36}
	\end{align} 
	where $\mathfrak z$, being  the center of $\pk$, is an abelian unitary Lie algebra, and $\gk_1,\dots,\gk_n$ are simple unitary Lie algebras. Such decomposition can be obtained by   considering the irreducible decomposition of the adjoint representation of $\pk\ominus^\perp\mathfrak z$ on itself. See \cite{Was10} section II.1
\end{rem}

Let $\mathfrak h$ be a maximal abelian unitary Lie subalgebra of $\mathfrak g$, i.e., a Cartan subalgebra of $\mathfrak g$. (Unitarity of the Lie subalgebra $\hk$ requires, by definition, that $H^*\in\mathfrak h$ whenever $H\in\mathfrak h$.) Let $\mathfrak h^*$ be the dual vector space of $\mathfrak h$. (The symbol $*$ in $\mathfrak h^*$ has nothing to do with the $*$-structure of $\mathfrak g$.) Then the root decomposition
\begin{align*}
\mathfrak g=\mathfrak h\oplus^\perp\bigoplus^\perp_{\alpha\in \Phi}\mathfrak g_\alpha
\end{align*}
gives the irreducible decomposition of the adjoint representation of $\mathfrak h$ on $\mathfrak g$. Here $\Phi\subset\mathfrak h^*$  is the root system of $\mathfrak g$, (In fact $\Phi$ is in the (real) dual space $(\ih)^*$ of $\ih$, where $\mathfrak{\hk}_{\mathbb R}=\{X\in\mathfrak{\hk}:X^*=-X \}$.) and $\mathfrak g_\alpha$ is the (1-dimensional) root space of the root $\alpha$, i.e. for any $H\in\mathfrak h,X\in\mathfrak g_\alpha$, we have $[H,X_\alpha]=\langle\alpha,H\rangle X$.

If we choose an invariant inner product $(\cdot|\cdot)$ on $\mathfrak g$, then the restricted inner product on $\mathfrak h$ induces naturally an anti-linear isomorphism $\mathfrak h^*\rightarrow \mathfrak h,\lambda\mapsto h_\lambda$, such that $\lambda(H)\equiv\langle H, \lambda\rangle=(H|h_\lambda)$. We can therefore define an inner product on $\mathfrak h^*$, also denoted by $(\cdot|\cdot)$, such that $(\lambda|\mu)=(h_\mu|h_\lambda)$ for any $\lambda,\mu\in\mathfrak h^*$. Now, for any $\alpha\in\Phi$, we choose a non-zero $X_\alpha\in\mathfrak g_\alpha$. Set
\begin{gather}
E_\alpha=\frac{\sqrt2}{\lVert X_\alpha\lVert\lVert\alpha\lVert}X_\alpha,\qquad 	F_\alpha=\frac{\sqrt2}{\lVert X_\alpha\lVert\lVert\alpha\lVert}X_\alpha^*,\qquad H_\alpha=\frac{2}{\lVert\alpha\lVert^2}h_\alpha. \label{eq11}
\end{gather}
Then $E_\alpha\in\mathfrak g_\alpha,F_\alpha\in\mathfrak g_{-\alpha},H_\alpha\in\mathfrak h$, and it is not hard to check the  following \textbf{unitary $\mathfrak {sl}_2$ relations}: 
\begin{gather}
[E_\alpha,F_\alpha]=H_\alpha,\qquad [H_\alpha,E_\alpha]=2E_\alpha,\qquad [H_\alpha,F_\alpha]=-2F_\alpha;\\
E_\alpha^*=F_\alpha,\qquad H_\alpha^*=H_\alpha.
\end{gather}
If there also exist non-zero $\widetilde E_\alpha\in\mathfrak g_\alpha$ and $\widetilde F_\alpha,\widetilde H_\alpha\in\mathfrak g$ satisfying the unitary $\mathfrak{sl}_2$ relation, then there exists $c\in\mathbb C$ with $|c|=1$, such that
\begin{align}
\widetilde E_\alpha=cE_\alpha,\qquad \widetilde F_\alpha=\overline cF_\alpha,\qquad \widetilde H_\alpha=H_\alpha.\label{eq65}
\end{align}
In particular, the element $H_\alpha$ can be defined independent of the invariant inner product of $\mathfrak g$.
\begin{cv}\label{lb2}
	Unless otherwise stated, in this article, we always let $(\cdot|\cdot)$ be the \textbf{normalized} invariant inner product  of $\mathfrak g$, i.e., the one under which the longest roots of $\mathfrak g$ have length $\sqrt 2$. The identification $\mathfrak h^*\rightarrow\mathfrak h,\lambda\mapsto h_\lambda$ are also defined with respect to this $(\cdot|\cdot)$.
\end{cv}

\subsubsection*{Unitary representations of $\gk$}
It is well known that irreducible finite-dimensional representations of $\mathfrak g$ are highest-weight representations, and their equivalence classes are characterized by their highest weights $\lambda\in\mathfrak h^*$. Let $\alpha_1,\dots,\alpha_n$ be the simple roots of $\mathfrak g$. Then the fundamental weights $\lambda_1,\dots,\lambda_n\in\mathfrak h^*$ are defined to satisfy that $\frac{2(\lambda_i|\alpha_j)}{\lVert\alpha_j\lVert^2}=\delta_{ij}$. Let
\begin{align*}
P_+(\mathfrak g)=\{k_1\lambda_1+\cdots+k_n\lambda_m:k_1,\dots,k_n\in\mathbb Z_{\geq0} \},
\end{align*}
the set of dominant integral weights of $\mathfrak g$. Then a weight $\lambda\in\mathfrak h^*$ is the highest weight of a finite-dimensional irreducible representation of $\mathfrak g$, if and only if $\lambda\in P_+(\mathfrak g)$.

Finite-dimensional irreducible representations of $\mathfrak g$ are unitarizable (see, for example, \cite{Was10} section II.14). It is easy to show that unitary highest-weight representations are automatically irreducible, and the highest-weights  are also inside $P_+(\mathfrak g)$ (since this is true for $\mathfrak {sl}_2$, see \cite{Was10} section II.4). Therefore, \emph{irreducible finite-dimensional $\mathfrak g$-modules and unitarizable highest-weight $\mathfrak g$-modules  are the same things}. 

For any $\lambda\in P_+(\gk)$, we choose a standard unitary highest-weight $\mathfrak g$-module $L_\gk(\lambda)$ whose highest weight is $\lambda$. Note that $L_\gk(0)=\mathbb C$, the $1$-dimensional trivial representation of $\gk$.

%\subsubsection*{Finite dimensional unitary Lie algebras}

%Sometimes it will be more convenient to consider non-simple finite dimensional Lie algebras with unitary structure. 

%

\subsubsection*{Dual representations}
We now discuss dual representations of $\gk$ from the unitary point of view. Dual representations of Lie algebras are similar to contragredient representations of VOAs. Let $U$ be a representation of $\gk$. We let $U^*$ be the dual vector space of $U$. Define an action of $\gk$ on $U^*$ using the formula
\begin{align*}
\bk{Xv',u}=-\bk{v',Xu}\qquad(\forall u\in U,v'\in U^*).
\end{align*}
Then this action makes $U^*$ a representation of $\gk$, called the \textbf{dual representation} of $U$. 

Now we assume that $U$ is a unitary representation. Then the inner product on $U$ induces  an antilinear isomorphism $C_U:U\rightarrow U^*$ satisfying
\begin{align*}
\bk{u,C_Uv}=\bk{u|v}\qquad(\forall u,v\in U).
\end{align*}
We thus define an inner product $\bk{\cdot|\cdot}$ on $U^*$:
\begin{align*}
\bk{C_Uu|C_Uv}=\bk{v|u}\qquad(\forall u,v\in U).
\end{align*}
Then $C_u$ becomes an anti-unitary map. It is easy to check that the action of $\gk$ on $U^*$ can be written as
\begin{align}
X\cdot v'=-C_UX^*C_U^{-1}\cdot v'\qquad(\forall v'\in U^*)
\end{align}
Thus $U^*$ is also a unitary $\gk$-module.

Now, for any $\lambda\in P_+(\gk)$, the dual representation of $L_\gk(\lambda)$ is also a finite-dimensional $\gk$-module, which is irreducible. Thus it is a unitary highest weight $\gk$-module.
\begin{cv}
For any $\lambda\in P_+(\gk)$, we let $\overline\lambda$ be the highest weight of the dual representation of $L_\gk(\lambda)$, and let $C_\lambda:L_\gk(\lambda)\rightarrow L_\gk(\overline\lambda)$ be the canonical anti-unitary map.
\end{cv}

\subsubsection*{Tensor products of representations}
Recall that if $U_1,U_2$ are finite-dimensional unitary $\gk$-modules, then $U_1\otimes U_2$ becomes a unitary $\gk$-module if the action of $\gk$ on which is defined by
\begin{align*}
X(u_1\otimes u_2)=Xu_1\otimes u_2+u_1\otimes Xu_2\qquad(X\in\gk,u_1\in U_1,u_2\in U_2).
\end{align*}
Now given two unitary highest weight modules $L_\gk(\lambda),L_\gk(\mu)$, we are interested in the irreducible decomposition of $L_\gk(\lambda)\otimes L_\gk(\mu)$. So we define, for any $\nu\in P_+(\gk)$,
\begin{align*}
\Hom_\gk(\lambda\otimes\mu,\nu)=\Hom_\gk(L_\gk(\lambda)\otimes L_\gk(\mu),L_\gk(\nu)).
\end{align*}
The number $\dim\Hom_\gk(\lambda\otimes\mu,\nu)$, being the multiplicity of $L_\gk(\nu)$ in $L_\gk(\lambda)\otimes L_\gk(\mu)$, is called a \textbf{tensor product rule} of $\gk$.

Tensor product rules share some similar properties with the fusion rules of VOAs. 
%First, we have an obvious unitary $\gk$-module isomorphism
%\begin{gather*}
%B:L_\gk(\lambda)\otimes L_\gk(\mu)\rightarrow L_\gk(\mu)\otimes L_\gk(\lambda),\\
%u^{(\lambda)}\otimes u^{(\mu)}\mapsto u^{(\mu)}\otimes u^{(\lambda)}.
%\end{gather*}
%The  $B$ induces an isomorphism of vector spaces, also denoted by $B$:
%\begin{align}
%B:\Hom_\gk(\lambda\otimes\mu,\nu)\rightarrow\Hom_\gk(\mu\otimes\lambda,\nu).
%\end{align}
%Next, the complex conjugations $C_\lambda\otimes C_\mu:L_\gk(\lambda)\otimes L_\gk(\mu)\rightarrow L_\gk(\overline \lambda)\otimes L_\gk(\overline\mu)$ and $C_\nu:L_\gk(\nu)\rightarrow L_\gk(\overline\nu)$ also induce an antilinear isomorphism
%\begin{align}
%\Conj:\Hom_\gk(\lambda\otimes\mu,\nu)\rightarrow \Hom_\gk(\overline\lambda\otimes\overline\mu,\overline\nu).
%\end{align}
For example, for any $T\in \Hom_\gk(\lambda\otimes\mu,\nu)$, we define a linear map $\Ad T:L_\gk(\overline\lambda)\otimes L_\gk(\nu)\rightarrow L_\gk(\mu)$, such that for any $u^{(\lambda)}\in L_\gk(\lambda),u^{(\mu)}\in L_\gk(\mu),u^{(\nu)}\in L_\gk(\nu)$,
\begin{align}
\bk{\Ad T(C_\lambda u^{(\lambda)}\otimes u^{(\nu)})|u^{(\mu)}}=\bk{u^{(\nu)}|T(u^{(\lambda)}\otimes u^{(\mu)} ) }.
\end{align}
It is easy to check that $\Ad T$ intertwines the action of $\gk$. So we have an antilinear map
\begin{align}
\Ad:\Hom_\gk(\lambda\otimes\mu,\nu)\rightarrow \Hom_\gk(\overline\lambda\otimes\nu,\mu).
\end{align}
which implies
\begin{align}
\dim\Hom_\gk(\lambda\otimes\mu,\nu)=\dim\Hom_\gk(\overline\lambda\otimes\nu,\mu).
\end{align}
In particular, we have $$\dim\Hom_\gk(\lambda\otimes\mu,0)=\dim\Hom_\gk(\overline\lambda\otimes 0,\mu)=\dim\Hom_\gk(L_\gk(\overline\lambda),L_\gk(\mu))=\delta_{\overline\lambda,\mu}.$$
We thus conclude:
\begin{pp}\label{lb23}
$\overline\lambda=\mu$ if and only if $L_\gk(\lambda)\otimes L_\gk(\mu)$ contains an irreducible submodule equivalent to the trivial module $\mathbb C$. In that case the multiplicity of $\mathbb C$ in $L_\gk(\lambda)\otimes L_\gk(\mu)$ is $1$.
\end{pp}

\subsubsection*{Computing tensor product rules}

It will be very helpful to compute some tensor product rules of a simple Lie algebra. For this purpose, we fix highest weight vectors $v_\mu\in L_\gk(\mu),v_\nu\in L_\gk(\nu)$. Let $L_\gk(\lambda)[\nu-\mu]$ be the weight subspace of $L_\gk(\lambda)$ with weight $\nu-\mu$, i.e., $L_\gk(\lambda)[\nu-\mu]$ is the subspace of vectors $u$ satisfying that $Hu=\bk{\nu-\mu,H}u$ for any $H\in\mathfrak h$.  Define a linear map
\begin{gather}
\Gamma:\Hom_\gk(\lambda\otimes\mu,\nu)\rightarrow L_\gk(\lambda)[\nu-\mu]^*,
\end{gather}
such that for any $T\in\Hom_\gk(\lambda\otimes\mu,\nu)$, $\Gamma T$ is defined by
\begin{align}
(\Gamma T)(u^{(\lambda)})=\bk{T(u^{(\lambda)}\otimes v_\mu)|v_\nu}\qquad(\forall u^{(\lambda)}\in L_\gk(\lambda)[\nu-\mu]).\label{eq40}
\end{align}
We first give an upper bound for a tensor product rule. Let $\gk_+$ and $\gk_-$ be the Lie subalgebras of raising operators and lowering operators respectively.
\begin{pp}\label{lb47}
The map $\Gamma$ is injective. Consequently,
\begin{align}
\dim\Hom_\gk(\lambda\otimes\mu,\nu)\leq\dim L_\gk(\lambda)[\nu-\mu].
\end{align}
\end{pp}
\begin{proof}
Assume that \eqref{eq40} equals $0$ for any $u^{(\lambda)}\in L_\gk(\lambda)[\nu-\mu]$. Then \eqref{eq40} obviously equals $0$ for any $u^{(\lambda)}\in L_\gk(\lambda)$. Since $v_\mu$ is a cyclic vector of $L_\gk(\mu)$  under the action of $\gk_-$, and since $v_\nu$ is annihilated by $\gk_+$, one can easily show that $\bk{T(u^{(\lambda)}\otimes u^{(\mu)})|v_\nu}=0$ for any $u^{(\lambda)}\in L_\gk(\lambda),u^{(\mu)}\in L_\gk(\mu)$. Using the fact that $v_\nu\in L_\gk(\nu)$ is also $\gk_-$-cyclic, one proves that $T=0$.
\end{proof}
We now describe the range of $\Gamma$. Recall that for any root $\alpha$, its dual root is defined by $\widecheck\alpha=2\alpha/\lVert\alpha \lVert^2$. For any $\rho\in\mathfrak h^*$, let $n_{\rho,\alpha}=(\rho|\widecheck\alpha)$.

\begin{pp}\label{lb11}
Let $K_\gk^\mu(\lambda)$ be the subspace of $L_\gk(\lambda)$ spanned by vectors of the form $F_\alpha^{n_{\mu,\alpha}+1}u^{(\lambda)}$, where $u^{(\lambda)}\in L_\gk(\lambda)$, $\alpha$ is a simple root of $\gk$, and $F_\alpha\in\gk_{-\alpha}$. Let $K_\gk^\mu(\lambda)[\nu-\mu]$ be the weight subspace of $K_\gk^\mu(\lambda)$ with weight $\nu-\mu$. Then an element $\varphi\in L_\gk(\lambda)[\nu-\mu]^*$ is in the range of $\Gamma$ if and only if $\varphi\perp K^\mu_\gk(\lambda)[\nu-\mu]$. As a consequence, 
\begin{align}
\dim\Hom_\gk(\lambda\otimes\mu,\nu)=\dim L_\gk(\lambda)[\nu-\mu]-\dim K^\mu_\gk(\lambda)[\nu-\mu].
\end{align}
\end{pp}

\begin{proof}
	
The "only if" part is easy to prove. Indeed, Assume that $\varphi=\Gamma T$. Then, since $F_\alpha^{n_{\mu,\alpha}+1}v_\mu=0$ and $F_\alpha^*v_\nu=0$, for any $u^{(\lambda)}\in L_\gk(\lambda)$ such that $F_\alpha^{n_{\mu,\alpha}+1}u^{(\lambda)}\in L_\gk(\lambda)[\nu-\mu]$, we have
	\begin{align*}
	\varphi(F_\alpha^{n_{\mu,\alpha}+1}u^{(\lambda)})=\bk{T(F_\alpha^{n_{\mu,\alpha}+1}u^{(\lambda)}\otimes v_\mu)|v_\nu }=-\bk{T(u^{(\lambda)}\otimes F_\alpha^{n_{\mu,\alpha}+1}v_\mu)|v_\nu }=0.
	\end{align*}
	Thus $\varphi\perp K^\mu_\gk(\lambda)[\nu-\mu]$. 

The proof of "if" part is postponed to section \ref{lb12}.
\end{proof}

\begin{co}\label{lb20}
Let $\lambda,\mu,\nu\in P_+(\gk)$. Assume $\dim L_\gk(\lambda)[\nu-\mu]=1$. Then $\dim\Hom_\gk(\lambda\otimes\mu,\nu)=1$ if and only if for any simple root $\alpha$,
\begin{align}
\dim L_\gk(\lambda)[\nu-\mu+(n_{\mu,\alpha}+1)\alpha]=0.
\end{align}
Otherwise, $\dim\Hom_\gk(\lambda\otimes\mu,\nu)=0$.
\end{co}

\begin{proof}
Assume firstly that $\dim\Hom_\gk(\lambda\otimes\mu,\nu)=0$. Then by proposition \ref{lb11}, $\dim K^\mu_\gk(\lambda)[\nu-\mu]=1$.  Choose a simple root $\alpha$ and a weight vector $u^{(\lambda)}$ of $L_\gk(\lambda)$ such that $F^{n_{\mu,\alpha}+1}_\alpha u^{(\lambda)}$ is a non-zero weight $\nu-\mu$ vector of $L_\gk(\lambda)$. Then $u^{(\lambda)}$ is a non-zero vector of $L_\gk(\lambda)$ of weight $\nu-\mu+(n_{\mu,\alpha}+1)\alpha$. Therefore $\dim L_\gk(\lambda)[\nu-\mu+(n_{\mu,\alpha}+1)\alpha]>0$.

Conversely, assume that there exists a simple root $\alpha$ such that $\dim L_\gk(\lambda)[\nu-\mu+(n_{\mu,\alpha}+1)\alpha]>0$. Consider the unitary representation of $\mathfrak{sl}_2$ on
\begin{align*}
U=\bigoplus_{n\in\mathbb Z}L_\gk(\lambda)[\nu-\mu+n\alpha],
\end{align*}
where the $\mathfrak{sl}_2$ is generated by the elements $E_\alpha,F_\alpha,H_\alpha$ defined by \eqref{eq11}. Note that both $L_\gk(\lambda)[\nu-\mu]$ and $L_\gk(\lambda)[\nu-\mu+(n_{\mu,\alpha}+1)\alpha]$ are non-trivial eigensubspaces of $H_\alpha$, and  their eigenvalues  differ by an even number. Therefore, from representation theory of $\mathfrak{sl}_2$, one can easily show the existence of $u^{(\lambda)}\in L_\gk(\lambda)[\nu-\mu+(n_{\mu,\alpha}+1)\alpha]$ such that $F^{n_{\mu,\alpha}+1}_\alpha u^{(\lambda)}$ is non-zero. Hence $\dim K^\mu_\gk(\lambda)[\nu-\mu]>0$. By proposition \ref{lb11}, $\dim\Hom_\gk(\lambda\otimes\mu,\nu)=0$.
\end{proof}

\subsection{Unitary affine VOAs}

\subsubsection*{The affine Lie algebra $\gh$ and its positive energy representations}

Given a finite-dimensional unitary complex simple Lie algebra $\mathfrak g$, we let $\gh=\mathfrak g\otimes\mathbb C[t]\oplus \mathbb CK$, where $t$ is a formal variable, and $K$ is a formal vector. Set $X(n)=X\otimes t^n$ for any $X\in\gk$ and $n\in\mathbb Z$. The Lie algebra structure on $\gh$ is given by
\begin{gather}
[X(m),Y(n)]=[X,Y](m+n)+m(X|Y^*)\delta_{m,-n}K\qquad(X,Y\in\mathfrak g,m,n\in\mathbb Z),\label{eq26} \\
[K,X(n)]=0\qquad(X\in\mathfrak g,n\in\mathbb Z).
\end{gather}
(Recall that by the anti-unitarity of $*$, $(X|Y^*)=(Y|X^*)$.) $\widehat{\mathfrak g}$ is called the \textbf{affine Lie algebra} of $\mathfrak g$. We define an involution $*:\widehat{\mathfrak g}\rightarrow\widehat{\mathfrak g}$ satisfying
\begin{gather}
X(n)^*=X^*(-n)\qquad(X\in\mathfrak g,n\in\mathbb Z),\label{eq27}\\
K^*=K.
\end{gather}
Again we have
\begin{align*}
[\mathfrak X,\mathfrak Y]^*=[\mathfrak Y^*,\mathfrak X^*]\qquad(\mathfrak X,\mathfrak Y\in\gh).
\end{align*}

Unitary representations (unitary modules) of $\widehat{\mathfrak g}$ are defined in the same way as those of $\mathfrak g$.  A unitary representation $W$ of $\widehat{\mathfrak g}$ is called a  \textbf{positive energy representation} (PER), if the vector space $W$ has an orthogonal grading $W=\bigoplus_{n\in\mathbb N}^\perp W(n)$, and for any $n\in\mathbb N$,
\begin{gather*}
\dim W(n)<\infty,\\
X(m)W(n)\subset W(n-m)\qquad(X\in\mathfrak g,m\in\mathbb Z),\\
KW(n)\subset W(n).
\end{gather*}

We now discuss briefly irreducible PERs of $\widehat{\mathfrak g}$. Details can be found in \cite{Was10}. Let $W$ be an irreducible  PER of $\widehat{\mathfrak g}$. Choose $n_0\in\mathbb N$ such that $\dim W(n_0)>0$, and $\dim W(n)=0$ for any $n<n_0$. $W(n_0)$ is called the \textbf{lowest energy subspace} of $W$. Identify $\mathfrak g$ with $\widehat{\mathfrak g}(0)$ by setting
\begin{gather*}
X=X(0)\qquad(X\in\mathfrak g).
\end{gather*}
Then $\mathfrak g$ is a unitary Lie subalgebra of $\widehat{\mathfrak g}$, and $W(n_0)$ is a finite-dimensional unitary representation of $\mathfrak g$, called the \textbf{lowest energy $\mathfrak g$-module}. It is easy to see that \emph{the lowest energy $\mathfrak g$-module of an irreducible PER of $\gh$ is irreducible.}

The action of $K$ on $W(n_0)$ commutes with the irreducible action  $\mathfrak g\curvearrowright W(n_0)$. Therefore, by Schur's lemma, 
\begin{align*}
K|_{W(n_0)}=l\cdot\id_{W(n_0)}
\end{align*}
for some $l\in\mathbb R$. That $l$ is real follows from the self-adjointness of $K$. We say that $l$ is the \textbf{level} of the irreducible PER $W$. (In general,  a general PER  of $\gh$ is said to have  \textbf{level} $l$, if the action of $K$ on the pre-Hilbert space is the identity operator multiplied by $l$.) Since $\gh\cdot W(n_0)=W$, $K$ must be the constant $l$ on the entire vector space $W$.

It is not hard to show that \emph{$\mathfrak g\curvearrowright W(n_0)$ and $l$ completely characterizes the PER $W$}, i.e., if two irreducible PERs reduce to equivalent lowest energy $\mathfrak g$-modules, and if their levels are equal, then they are (unitary) equivalent. We also know that finite-dimensional irreducible representations of $\mathfrak g$ must be highest-weight representations, and are determined, up to equivalence, by their highest-weights. So the equivalence class of $W$ is determined by the pair $(\lambda,l)$, where $\lambda\in\mathfrak h^*$ is the highest-weight of $\mathfrak g\curvearrowright W(n_0)$.

There is also a notion of  highest-weightness for PERs of $\gh$. Let $\mathfrak g_+$ (resp. $\mathfrak g_-$) be the Lie subalgebra of raising operators (resp. lowering operators) in $\mathfrak g$. Let $\widehat{\mathfrak g}_{>0}$ (resp. $\widehat{\mathfrak g}_{<0}$) be the Lie subalgebra of $\widehat{\mathfrak g}$ spanned by $X(n)$, where $X\in\mathfrak g$ and $n\in\mathbb Z_{>0}$ (resp. $n\in\mathbb Z_{<0}$). Let $W$ be a PER of $\gh$. A non-zero vector $v\in W$ is called a  \textbf{highest-weight vector} of $W$, if the following conditions are satisfied:\\
(a) There exists a grading of $W$, such that $v\in W(n)$ for some $n\in\mathbb N$.\\
(b) There exists $\lambda\in\mathfrak h^*$ such that $Hv=\bk{\lambda,H}v$ for any $H\in\mathfrak h$.\\
(c) There exists $l\in\mathbb R$ such that $Kv=lv$.\\
(d) $\mathfrak g_+v=0,\gh_{>0}v=0$.\\
In this case, $\lambda$ is called the \textbf{highest weight} of $v$, and $l$ is called the \textbf{level} of $v$. If $v$ is a cyclic vector, i.e. $W=U(\gh)\cdot v$ where $U(\gh)$ is the universal enveloping algebra of $\gh$, we say that $W$ is a \textbf{highest-weight PER} of $\gh$.

If $W$ is an irreducible PER, then the highest-weight vector $v$ of its lowest energy $\mathfrak g$-module is a highest-weight vector of $\gh\curvearrowright W$. So $W$ is a highest-weight PER of $\gh$. Conversely, if $W$ is a highest-weight PER, and  $v$ is a highest-weight vector with highest weight $\lambda$ and level $l$, then it is easy to see that $W$ is irreducible. So \emph{irreducible PERs and highest-weight PERs are the same for $\gh$}. Moreover, $v$ is a highest-weight vector of the lowest energy $\mathfrak g$-module of $W$. So $W$ has a unique up to scalar multiplication highest-weight vector, a unique highest weight, and a unique level. 

The remaining question is to know the possible $\lambda$ and $l$ for irreducible PERs. Clearly $\lambda\in P_+(\mathfrak g)$. But extra conditions are required. Let $\theta\in\Phi$ be the highest root of $\mathfrak g$, i.e., the highest weight of the adjoint representation $\mathfrak g\curvearrowright\mathfrak g$. Since $\theta$ is a longest root, we have $\lVert\theta\lVert^2=2$. Choose $E_{-\theta}\in\mathfrak g_{-\theta},F_{-\theta}\in\mathfrak g_{\theta},H_{-\theta}\in\mathfrak h$ as in \eqref{eq11}. Then we have $\lVert E_{-\theta} \lVert^2=\lVert F_{-\theta} \lVert^2=1$, and $H_{-\theta}=h_{-\theta}$. Using these facts, one can easily check that $E_{-\theta}(1),F_{-\theta}(-1),H_{-\theta}(0)+K$ satisfy the unitary $\mathfrak{sl}_2$ relation. Since the highest-weight vector $v$ of $W$ is also a highest-weight vector of this $\mathfrak{sl}_2$ with highest weight $-(\lambda|\theta)+l$ (which is also easy to check), we must have
\begin{gather*}
-(\lambda|\theta)+l\in\mathbb Z_{\geq0}.
\end{gather*}
Since $\lambda,\theta\in P_+(\mathfrak g)$, we must have $(\lambda|\theta)\in\mathbb Z_{\geq0}$. Therefore, $l$ is a non-negative integer no less than $(\lambda|\theta)$. We conclude that if $\lambda,l$ are the highest weight and the level of an irreducible PER of $\gh$, then condition \eqref{eq12} holds. The converse is also true: if condition \eqref{eq12} is satisfied, then there exists an irreducible PER with highest weight $\lambda$ and level $l$.  See \cite{Kac94}, or \cite{Was10} chapter III.  We conclude the following:

\begin{thm}\label{lb4}
Let $W$ be a PER of $\gh$. Then $W$ is irreducible if and only if it is a highest-weight PER. In this case, $W$ has a unique up to scalar multiplication highest-weight vector $v_\lambda$ with highest weight $\lambda$ and level $l$.   $\lambda$ and $l$ satisfy
\begin{gather}
\lambda\in P_+(\mathfrak g),\qquad l\in\mathbb Z_{\geq0},\qquad (\lambda|\theta)\leq l,\label{eq12}
\end{gather}
and they completely determine the PER $W$ up to unitary equivalence. The lowest energy $\mathfrak g$-module $W(n_0)$ of $W$ is irreducible. We have $K=l\cdot\id_W$, $v_\lambda\in W(n_0)$, and $v_\lambda$ is also a highest-weight vector of $\mathfrak g\curvearrowright W(n_0)$ with highest weight $\lambda$.

Conversely, if $\lambda$ and $l$ satisfy \eqref{eq12}, then there exists an irreducible PER of $\gh$ with highest weight $\lambda$ and level $l$.
\end{thm}

Due to this theorem, for any $l\in\mathbb Z_{\geq0}$, we say that a dominant integral weight $\lambda$ of $\gk$ is \textbf{admissible at level $l$}, if $(\lambda|\theta)\leq l$. For such $\lambda$, we choose $L_\gk(\lambda,l)$ to be a standard PER of $\gh$ with highest weight $\lambda$ and level $l$.

\subsubsection*{Unitary affine VOAs}
The PER $L_\gk(0,l)$ of $\gh$ is called the level $l$ \textbf{vacuum representation} of $\gh$. We now extend such representation of $\gh$ to a unitary VOA structure.

\begin{thm}
Let $\mathfrak g$ be a finite-dimensional complex simple unitary Lie algebra. Fix $l\in\mathbb Z_{\geq0}$, and fix a highest weight vector $\Omega\in L_\gk(0,l)$ satisfying $\lVert\Omega\lVert=1$. Then there exists a unique unitary VOA structure of CFT type on the pre-Hilbert space $V^l_\gk=L_\gk(0,l)$, such that $\Omega$ is the vacuum vector, and 
\begin{align}
Y\big(X(-1)\Omega,x\big)=\sum_{n\in\mathbb Z}X(n)x^{-n-1}\qquad(X\in\gk).\label{eq20}
\end{align}
$V^l_\gk$ is called the \textbf{level $l$ unitary affine VOA of $\gk$}. The conformal vector $\nu$ of $V^l_\gk$ is given by the  formula
\begin{align}
\nu=\frac 1{2(l+h^\vee)}\sum_i X_i^*(-1)X_i(-1)\Omega,\label{eq25}
\end{align}
where $h^\vee$ is the dual Coxeter number of $\gk$, and $\{X_i \}$ is a set of orthonormal basis of $\mathfrak g$ under the normalized invariant inner product $(\cdot|\cdot)$ (see convention \ref{lb2}). The PCT operator $\Theta$ is determined by 
\begin{align}
\Theta X(-1)\Omega=-X^*(-1)\Omega\qquad(X\in\gk).\label{eq21}
\end{align} 

\end{thm}

\begin{proof}
The existence of a VOA structure on $V^l_\gk$ such that relation \eqref{eq20} holds, that $\Omega$ is the vacuum vector, and that the $\nu$ defined by \eqref{eq21} is the conformal vector, is given in \cite{FZ92} theorem 2.4.1. In this case, the set 
\begin{align}
E=\{X(-1)\Omega:X\in\gk \}\label{eq23}
\end{align}
generates $V^l_\gk$. So any antilinear automorphism on the VOA $V^l_\gk$ is determined by its values on $E$. That $V^l_\gk$ is unitary with PCT operator $\Theta$ determined by \eqref{eq21} is proved in \cite{DL14} section 4.2.   One can also prove the unitarity of $V^l_\gk$ using \cite{CKLW18} proposition 5.17, which is a  general criterion on the unitarity of a VOA. To use this proposition, one needs to show  that $L_n^\dagger=L_{-n}$, and that there exists a generating set $F$ of quasi-primary vectors in $V^l_\gk$, such that
\begin{align}
Y(v,x)^\dagger=x^{-2\Delta_v}Y(v,x^{-1})\qquad(\forall v\in F)\label{eq28}.
\end{align}
Indeed, using \eqref{eq25} and \eqref{eq24}, one easily obtains  Sugawara's formulae:
\begin{gather}
L_0=\frac 1{l+h^\vee}\bigg(\frac 12\sum_{i}X^*_i(0)X_i(0)+\sum_{n>0}\sum_i X^*_i(-n)X_i(n)\bigg),\label{eq46}\\
L_m=\frac 1{2(l+h^\vee)}\sum_{n\in\mathbb Z}\sum_i X^*_i(m-n)X_i(n)\qquad(m\neq0).
\end{gather}
Direct computation now gives $L_m^\dagger=L_{-m}$. Using \eqref{eq26} and \eqref{eq27}, it is not hard to check that the set $E$ is quasi-primary, and that its subset $F=\{X(-1)\Omega:X\in\gk,X=X^* \}$, which is also generating, satisfies condition \eqref{eq28}.

Since $E$ is generating, \eqref{eq20} uniquely determines the vertex-algebraic structure of $V^l_\gk$. Uniqueness of $\Theta$ and $\nu$ follows from proposition \ref{lb3}.
\end{proof}

\begin{thm}\label{lb5}
Fix $l\in\mathbb Z_{\geq0}$. 

(1) For any $\lambda\in P_+(\gk)$ satisfying $(\lambda|\theta)\leq l$, there exists a unique unitary representation of $V^l_\gk$ on the pre-Hilbert space $L_\gk(\lambda,l)$, the vertex operator $Y_\lambda$ of which satisfies
\begin{align}
Y_\lambda\big(X(-1)\Omega,x\big)w^{(\lambda)}=\sum_{n\in\mathbb Z}X(n)w^{(\lambda)}\cdot x^{-n-1}\qquad(X\in\gk,w^{(\lambda)}\in L_\gk(\lambda,l)).\label{eq22}
\end{align}
The $V^l_\gk$-module $L_\gk(\lambda,l)$ is irreducible.  $L_\gk(\lambda_1,l)$ is unitarily equivalent to $L_\gk(\lambda_2,l)$ if and only if $\lambda_1=\lambda_2$.

(2) Any $V^l_\gk$-module is a (finite) direct sum of irreducible representations of the form $L_\gk(\lambda,l)$, where $\lambda\in P_+(\mathfrak g)$ is admissible at level $l$. In particular, any $V^l_\gk$-module is unitarizable.
\end{thm}
\begin{proof}
(1) (cf. \cite{DL14} Proposition 4.9.) Existence of a $V^l_\gk$-module $L_\gk(\lambda,l)$ satisfying \eqref{eq22} is proved in \cite{FZ92} section 3.1. Since $L_\gk(\lambda)$ is unitary as a $\gh$ module,  by \eqref{eq27}, \eqref{eq20}, and \eqref{eq21}, the vertex operator $Y_\lambda$ satisfies relation \eqref{eq29} when  $v\in E=\{X(-1)\Omega:X\in\gk \}$. Since any $v\in E$ is quasi-primary, relation \eqref{eq19} also holds for $Y_\lambda$ and any $v\in E$. Therefore, by \cite{Gui19} proposition 1.10, $L_\gk(\lambda,l)$ is a unitary $V^l_\gk$-module. The rest of part (1) is obvious.

(2) This was also proved in \cite{FZ92} section 3.1.
\end{proof}

\begin{rem}
Let $L_\gk(\lambda,l)$ be a unitary irreducible $V^l_\gk$-module, where $\lambda\in P_+(\gk)$ is admissible at level $l$. Recall that $\gk$ is regarded as a Lie subalgebra of $\widehat\gk$ by identifying $X\in\gk$ with $X(0)$, and the action of $\gk$ on the lowest energy subspace of $L_\gk(\lambda,l)$ is  equivalent to $L_\gk(\lambda)$, the unitary highest weight representation of $\gk$ with highest weight $\lambda$. So we identify the vector space $L_\gk(\lambda)$ with the lowest energy subspace of $L_\gk(\lambda,l)$.

Choose $\Delta_\lambda\in\mathbb R$ such that $L_0|_{L_\gk(\lambda)}=\Delta_\lambda\cdot\id_{L_\gk(\lambda)}$, i.e. $\Delta_\lambda$ is the lowest conformal weight  of $L_\gk(\lambda,l)$. By Sugawara's construction, $\Delta_\lambda>0$. We say that $\Delta_\lambda$ is the \textbf{conformal weight} of $L_\gk(\lambda,l)$. Then the conformal weight $\Delta_w$ of any non-zero homogeneous vector $w\in L_\gk(\lambda,l)$ must be in $\Delta_\lambda+\mathbb Z_{\geq0}$.
\end{rem}

\begin{rem}
It is easy to see that the contragredient $V^l_\gk$-module of $L_\gk(\lambda,l)$ is $L_\gk(\overline\lambda,l)$. As an immediate consequence, we have
\begin{align}
(\lambda|\theta)=(\overline\lambda|\theta)\qquad(\forall\lambda\in P_+(\gk)).
\end{align}
Indeed, we can choose $l=(\lambda|\theta)$. Then since $L_\gk(\overline\lambda,l)$ is also an irreducible $V$-module, we must have $(\overline\lambda|\theta)\leq l=(\lambda|\theta)$. Replace $\lambda$ with $\overline\lambda$, we get $(\lambda|\theta)=(\overline{\overline\lambda}|\theta)\leq(\overline\lambda|\theta)$.
\end{rem}

\begin{thm}\label{lb18}
Let $\lambda\in P_+(\gk)$ be admissible at level $l$. Then $L_\gk(\lambda,l)$ is energy-bounded. Moreover, for any $X\in\gk$, $Y_\lambda(X(-1)\Omega,x)$ satisfies $1$-st order energy-bounds.
\end{thm}

\begin{proof}
The argument in \cite{BS90} section 2 shows that $Y_\lambda(X(-1)\Omega,x)$ satisfies $1$-st order energy-bounds. Indeed, one can prove the $\frac 12$-th order energy bounds condition for $Y_\lambda(X(-1)\Omega,x)$ using the Sugawara's formula \eqref{eq46}, as shown in \cite{TL04} proposition II.1.2.1. Now the energy-boundedness of $Y_\lambda$ follows from proposition \ref{lb15}.
\end{proof}

\begin{subappendices}

\subsection{Proof of proposition \ref{lb11}}\label{lb12}

In this appendix section, we prove the "if" part of proposition \ref{lb11} using a method similar to \cite{TK88} proposition 2.3.

We first recall some basic facts about Verma modules. Let $A\gk$ be the free associative algebra generated by the \emph{vector space} $\gk$, i.e., $A\gk=\bigoplus_{n\in\mathbb Z_{\geq0}}\gk^{\otimes n}$. Let $I_\gk$ be the ideal of $A\gk$ generated by $XY-YX-[X,Y]$ ($X,Y\in\gk$). The universal enveloping algebra of $\gk$ is $U\gk=A\gk/I_\gk$. Now for any $\eta\in \mathfrak h^*$, we chose an arbitrary non-zero vector $v_\eta$. Then $U\gk\otimes v_\eta$ is a (left) $U\gk$-module. Identify $1\otimes v_\eta$ with $v_\eta$. Let $J_\eta$ be the $U\gk$-submodule of $U\gk\otimes v_\eta$ generated by the vectors $Hv_\eta-\bk{\eta,H}v_\eta$ ($H\in\mathfrak h$) and $Xv_\eta$ ($X\in\gk_+$). The quotient module $M_\eta=U\gk\otimes v_\eta/J_\eta$ is a Verma module of $\gk$ with highest weight $\eta$. $M_\eta$, as a $U\gk_-$ module, is canonically isomorphic to $U\gk_-\otimes v_\eta$, since, by Poincare-Birkhoff-Witt theorem, $U\gk= U\gk_-\otimes U\mathfrak h\otimes U\gk_+$. Last, let $\overline{M_\eta}$ be the complex conjugate of $M_\eta$, and let $C_\eta:M_\eta\rightarrow \overline{M_\eta},u^{(\eta)}\mapsto C_\eta u^{(\eta)}=\overline{u^{(\eta)}}$ be the antilinear isomorphism. We let $\gk$ act on $\overline {M_\eta}$ by setting $X\overline {u^{(\eta)}}=-\overline{ X^* u^{(\eta)}}$ ($u^{(\eta)}\in M_\eta$). Then $\overline {v_\eta}$ is a lowest weight vector of $\overline{M_\eta}$.

Now, for any  $\varphi\in L_\gk(\lambda)[\nu-\mu]^*$, we extend it to a $\gk$-invariant tri-linear form on $L_\gk(\lambda)\otimes M_\mu\otimes\overline{M_\nu}$. First, we extend $\varphi$ to an element in $L_\gk(\lambda)^*$ by setting $\varphi(u^{(\lambda)})=0$ when $u^{(\lambda)}\in L_\gk(\lambda)[\rho]$ and $\rho\neq\nu-\mu$. We now regard $\varphi$ as a linear map $L_\gk(\lambda)\otimes v_\mu\otimes \overline{v_\nu}\rightarrow \mathbb C$. Next, identify $v_\mu$ with $1\otimes v_\mu\in A\gk_-\otimes v_\mu$, and extend $\varphi$ to a linear map $L_\gk(\lambda)\otimes (A\gk_-\otimes v_\mu)\otimes\overline{v_\nu}\rightarrow\mathbb C$, satisfying
\begin{align}
\varphi(u^{(\lambda)}\otimes X_1\cdots X_n v_\mu\otimes \overline{v_\nu})=(-1)^n\varphi(X_n\cdots X_1u^{(\lambda)}\otimes v_\mu\otimes\overline{v_\nu} )\label{eq41}
\end{align}
for any $n\in\mathbb Z_{\geq0},X_1,\dots,X_n\in\gk_-$. One easily checks that $\varphi$ vanishes on $L_\gk(\lambda)\otimes (I_{\gk_-}\otimes v_\mu)\otimes\overline{v_\nu}$, where $I_{\gk_-}$ is the ideal of $\gk_-$ generated by $XY-YX-[X,Y]$ ($X,Y\in\gk_-$). So $\varphi$ factors through a linear map, also denoted by $\varphi$, mapping $L_\gk(\lambda)\otimes (U\gk_-\otimes v_\mu)\otimes\overline{v_\nu}\rightarrow \mathbb C$. Identify $U\gk_-\otimes v_\mu$ with $M_\mu$. Thus $\varphi$ is a linear functional on $L_\gk(\lambda)\otimes M_\mu\otimes\overline{v_\nu}$. 

Extend $\varphi$ to a linear functional on $L_\gk(\lambda)\otimes M_\mu\otimes (\overline{A\gk\otimes v_\nu})$ by setting, inductively,
\begin{align}
\varphi(u^{(\lambda)}\otimes u^{(\mu)}\otimes \overline{X_1 X_2\cdots X_n v_\nu})=&\varphi(X_1^*u^{(\lambda)}\otimes u^{(\mu)}\otimes \overline{ X_2\cdots X_n v_\nu})\nonumber\\
&+\varphi(u^{(\lambda)}\otimes X_1^*u^{(\mu)}\otimes \overline{ X_2\cdots X_n v_\nu}),\label{eq42}
\end{align}
where $X_1,\dots,X_n\in\gk$. Clearly $\varphi$ vanishes on $L_\gk(\lambda)\otimes M_\mu\otimes \overline{I_\gk\otimes v_\nu}$. So $\varphi$ factors through a linear functional on $L_\gk(\lambda)\otimes M_\mu\otimes\overline{U\gk\otimes v_\nu}$, also denoted by $\varphi$. Using \eqref{eq41} and the fact that $\varphi$ vanishes on $L_\gk(\lambda)[\rho]\otimes v_\mu\otimes\overline{v_\nu}$ when $\rho+\mu\neq\nu$, it is not hard to see that $\varphi$ also vanishes on $L_\gk(\lambda)\otimes M_\mu\otimes\overline{J_\nu\otimes v_\nu}$. So it factors through a linear functional on $L_\gk(\lambda)\otimes M_\mu\otimes\overline{M_\nu}$,  denoted again by $\varphi$. By \eqref{eq42}, $\varphi$ is clearly $\gk$-invariant.

Let $N_\mu$ be the $\gk_-$-submodule of $M_\mu$ generated by  $F_\alpha^{n_{\mu,\alpha}+1}v_\mu$ ($\alpha$ is a simple root,  $F_\alpha\in\gk_{-\alpha}$). Then it is well known that  each $F_\alpha^{n_{\mu,\alpha}+1}v_\mu$ is a highest weight vector, implying that $N_\mu$ is $\gk$-invariant. We also know that $L_\gk(\mu)=M_\mu/N_\mu$. Hence, if we assume that $\varphi\perp K^\mu_\gk(\lambda)[\nu-\mu]$, then by \eqref{eq41}, $\varphi$ vanishes on $L_\gk(\lambda)\otimes N_\mu\otimes \overline{v_\nu}$, and hence on $L_\gk(\lambda)\otimes N_\mu\otimes\overline{M_\nu}$ by the $\gk$-invariance of $\varphi$. So $\varphi$ can be regarded as a $\gk$-invariant linear functional on $L_\gk(\lambda)\otimes L_\gk(\mu)\otimes\overline{M_\nu}$. Let
\begin{align*}
O_\nu=\{v^{(\nu)}\in M_\nu:\varphi\perp L_\gk(\lambda)\otimes L_\gk(\mu)\otimes \overline{v^{(\nu)}} \}.
\end{align*}
Clearly $O_\nu$ is a $\gk$-submodule of $M_\nu$. For any $\rho\in\mathfrak h^*$, $O_\nu[\rho]=M_\nu[\rho]$ when $\rho$ is not a weight of $L_\gk(\lambda)\otimes L_\gk(\mu)$. So $M_\nu/O_\nu$ is a finite-dimensional $\gk$-module. Assume, without loss of generality, that $\varphi\neq0$. Then $M_\nu/O_\nu$ is non-trivial. So $M_\nu/O_\nu\simeq L_\gk(\nu)$. Thus $\varphi$ is a $\gk$-invariant linear functional on $L_\gk(\lambda)\otimes L_\gk(\mu)\otimes L_\gk(\overline\nu)$, which is equivalent to saying that $\varphi\in\Hom_\gk(\lambda\otimes\mu,\nu)$. It is obvious that $\Gamma\varphi$ is our original $\varphi$. Thus the "if" part of proposition \ref{lb11} is proved.

\end{subappendices}

\section{Compression principle}

\subsection{Unitary affine vertex subalgebras}\label{lb39}
It is well known that for a VOA $V$, its weight-$1$ subspace $V(1)$ has a natural Lie algebra structure. In this chapter we let $V$ be a unitary VOA of CFT type with PCT operator $\Theta$. We shall see that $V(1)$ is a unitary Lie algebra.
\begin{lm}
	Vectors in $V(1)$ are quasi-primary.
\end{lm}
\begin{proof}
	Choose any $v\in V(1)$. Then $L_1v\in V(0)=\mathbb C\Omega$. But $\bk{L_1v|\Omega}=\bk{v|L_{-1}\Omega}=0$. So $L_1v=0$.
\end{proof}

\begin{pp}
	The vector space $V(1)$, equipped with the bracket relation 
	\begin{align*}
	[u,v]=Y(u,0)v\qquad(u,v\in V(1))
	\end{align*}
	and the $*$-structure
	\begin{align}
	v^*=-\Theta v\qquad(v\in V(1)),\label{eq30}
	\end{align}
	is a finite dimensional unitary complex Lie algebra. The inner product of $V$ restricts to an invariant inner product of $V(1)$.
\end{pp}

\begin{proof}
	Equation \eqref{eq14} shows that
	\begin{align}
	[Y(u,0),Y(v,0)]=Y\big(Y(u,0)v,0\big).
	\end{align}
	Therefore, for any $u,v,w\in V(1)$,
	\begin{align*}
	[u,[v,w]]-[v,[u,w]]=&[u,Y(v,0)w]+[v,Y(u,0)w]=Y(u,0)Y(v,0)w-Y(v,0)Y(u,0)w\\
	=&Y(Y(u,0)v,0)w=[Y(u,0)v,w]=[[u,v],w].
	\end{align*}
	Thus the Jacobi identity is proved. 
	
	Now \eqref{eq30} and equation \eqref{eq29} imply
	\begin{align}
	Y(u,x)^\dagger=x^{-2}Y(u^*,x^{-1}),
	\end{align}
	or equivalently
	\begin{align}
	Y(u,n)^\dagger=Y(u^*,-n)\qquad(n\in\mathbb Z).\label{eq31}
	\end{align}
	Hence
	\begin{align*}
	\bk{[u,v]|w}=\bk{Y(u,0)v|w}=\bk{v|Y(u^*,0)w}=\bk{v|[u^*,w]}.
	\end{align*}
	Therefore the inner product on $V(1)$ inherited from $V$ is invariant. So the Lie algebra $V(1)$ equipped with the involution $*$ defined by \eqref{eq30} is unitary.
\end{proof}

Now we assume that $\gk$ is a unitary simple Lie subalgebra of $V(1)$. By definition, $\gk$ is preserved by the involution $*$ of $V(1)$  The invariant inner product $\bk{\cdot|\cdot}$ on $V(1)$ is also an invariant one of $\gk$. Recall that our normalized invariant inner product $(\cdot|\cdot)$ of $\mathfrak g$ is chosen under which the length of longest roots of $\gk$ is $\sqrt2$. Since $\gk$ is simple, there exists $l>0$ such that $\bk{X_1|X_2}=l(X_1|X_2)$ for any $X_1,X_2\in\gk$. We call $l$ \textbf{the  level of $\gk$ in $V$}.
\begin{lm}
	For any $X_1,X_2\in\gk\subset V(1)$,
	\begin{align}
	Y(X_1,1)X_2=l(X_1|X_2^*)\Omega.\label{eq32}
	\end{align}
\end{lm}

\begin{proof}
	Clearly $Y(X_1,1)X_2\in V(0)$. So there exists $c\in\mathbb C$ such that $Y(X_1,1)X_2=c\Omega$. We now compute, using relation \eqref{eq31}, that
	\begin{align*}
	c=\bk{c\Omega|\Omega}=\bk{Y(X_1,1)X_2|\Omega}=\bk{X_2|Y(X_1^*,-1)\Omega}=\bk{X_2|X_1^*}=l(X_2|X_1^*)=l(X_1|X_2^*).
	\end{align*}
\end{proof}
In the following theorem we show that the unitary simple Lie subalgebra $\gk\subset V(1)$ generates a unitary affine vertex subalgebra of $V$. This theorem can be regarded as the unitary analogue of \cite{DM06} theorem 3.1.
\begin{thm}\label{lb25}
	We have $l\in\mathbb Z_{\geq0}$. The map $\pi:\gh\rightarrow\End(V)$ defined by
	\begin{gather*}
	X(n)\mapsto Y(X,n)\qquad(X\in\gk),\\
	K\mapsto l\cdot\id_V
	\end{gather*}
	is a PER of $\gh$. Its unitary submodule  $\gh\curvearrowright\gh\Omega$ is unitarily equivalent to the vacuum representation $V^l_\gk\equiv L_\gk(0,l)$ of $\gh$, and $\Omega$ is a highest weight vector of $\gh\Omega$. Let $\varphi:V^l_\gk\rightarrow \gh\Omega$ be the unitary $\gh$-module isomorphism mapping $\Omega\in V^l_\gk$ to $\Omega\in \gh\Omega\subset V$. Then $\varphi$ embeds $V^l_\gk$ into $V$ as a unitary vertex subalgebra.
\end{thm}

\begin{proof}
	For any $X_1,X_2\in\gk, k=2,3,4,\dots$, $Y(X_1,k)X_2\in V(1-k)=0$. Therefore, by equations \eqref{eq14} and \eqref{eq32},
	\begin{align*}
	[\pi(X_1,m),\pi(X_2,n)]=&[Y(X_1,m),Y(X_2,n)]=\sum_{k\geq0}{m\choose k}Y\big(Y(X_1,k)X_2,m+n-k \big)\\
	=&Y\big(Y(X_1,0)X_2,m+n \big)+mY\big(Y(X_1,1)X_2,m+n-1 \big)\\
	=&Y([X_1,X_2],m+n)+ml(X_1|X_2^*)Y(\Omega,m+n-1)\\
	=&Y([X_1,X_2],m+n)+ml(X_1|X_2^*)\delta_{m,-n}\id_V\\
	=&\pi\big([X_1,X_2](m+n)+m(X_1|X_2^*)\delta_{m,-n}K\big)=\pi\big([X_1(m),X_2(n)]\big).
	\end{align*}
	Clearly $\pi(K)$ commutes with any $Y(X,n)$. By \eqref{eq31}, $\pi$ is a unitary representation of $\gh$. So $\pi$ is a level $l$ PER of $\gh$. In particular, $l\in\mathbb Z_{\geq0}$. By state-field correspondence, $Y(X,n)\Omega=0$ when $n\geq0$. So $\Omega$ is a highest weight vector with highest weight $0$. By theorem \ref{lb4}, there exists an isometry $\varphi:V^l_\gk\rightarrow V$, such that $\varphi(\Omega)=\Omega$, and $\varphi X(n)=\pi\big(X(n)\big)\varphi$ for any $X\in\gk,n\in\mathbb Z$. It is clear that $\varphi$ preserves the gradings of the two vector spaces.
	
	Write $U=V^l_\gk$ for simplicity. For clarity, we use different symbols $Y_U$ and $Y_V$ to denote vertex operators of $U$ and $V$. We also let $\Theta_U$ and $\Theta_V$ be the PCT operators of $U$ and $V$ respectively, and let $\Omega_U,\Omega_V$ be the vacuum vectors of $U$ and $V$ respectively. To show that $\varphi$ embeds $U$ into $V$ as a vertex subalgebra, we still need to check that
	\begin{gather}
	\varphi Y_U(u,x)=Y_V(\varphi u,x)\varphi,\label{eq33}\\
	\varphi\Theta_U\cdot u=\Theta_V\varphi\cdot u\label{eq34}
	\end{gather}
	for any $u\in U$.

	First, note that
	\begin{align*}
	\varphi Y_U(X,n)=\varphi X(n)=\pi(X(n))\varphi=Y_V(X,n)\varphi,
	\end{align*}
	and also
	\begin{align*}
	\varphi\Theta_U\cdot X(-1)\Omega_U=&-\varphi X^*(-1)\Omega_U=-\pi(X^*(-1))\varphi\Omega_U\\
	=&-Y_V(X^*,-1)\Omega_V=Y_V(\Theta_V X,-1)\Theta_V\Omega_V\\
	=&\Theta_V Y_V(X,-1)\Omega_V=\Theta_V\pi(X(-1))\varphi\Omega_U=\Theta_V\varphi\cdot X(-1)\Omega_U.
	\end{align*}
	So \eqref{eq33} and \eqref{eq34} hold for any $u\in U(1)$. Since $U(1)$ generates $U$, using Jacobi identity \eqref{eq24}, one can easily show that \eqref{eq33} and \eqref{eq34} hold for any $u\in U$.
\end{proof}
In the special case when $\gk=V(1)$ and $V(1)$ generates $V$, we have $V=\gh\Omega$. So $\varphi$ is a unitary equivalence between $V^l_\gk$ and $V$. The above theorem can be regarded as a uniqueness theorem for unitary affine VOA: If $V$ is a CFT type unitary VOA generated by the set $V(1)$ of weight-$1$ vectors, and if the Lie algebra $\gk=V(1)$ is simple, then the unitary VOA $V$ is equivalent to $V^l_\gk$ for some $l\in\mathbb Z_{\geq0}$. More generally, if $V(1)$ generats $V$ and  has  decomposition \eqref{eq36}, it is not hard to show the following equivalence of unitary VOAs:
\begin{align}
V\simeq V^1_{\mathfrak z}\otimes V_{\gk_1}^{l_1}\otimes\cdots\otimes  V_{\gk_n}^{l_n},
\end{align}
where $l_1,\dots,l_n\in\mathbb Z_{\geq0}$, and $V^1_{\mathfrak z}$ is the unitary (level $1$) Heisenberg VOA associated to the   unitary abelian Lie algebra $\mathfrak z$ whose unitary structure inherits from that of $V(1)$. (A finite dimensional complex vector space $\mathfrak h$ becomes a unitary abelian Lie algebra when $\mathfrak h$ is equipped with an inner product and an anti-unitary map $\Theta:\mathfrak h\rightarrow\mathfrak h$ satisfying $\Theta^2=\id_{\mathfrak h}$.)

\subsubsection*{Examples}

Let  $\pk$ be a finite dimensional unitary simple Lie algebra,  and assume that $\gk$ is a unitary simple Lie subalgebra of $\pk$. As our first example, we study the unitary affine vertex subalgebra arising from $\gk\subset V^l_\pk$. Let $(\cdot|\cdot)_\gk$ and $(\cdot|\cdot)_\pk$ be the normalized invariant inner product of $\gk$ and $\pk$ respectively. Then $(\cdot|\cdot)_\pk$ also restricts to an invariant inner product of $\gk$. So there exists $k>0$ such that
\begin{align}
(X_1|X_2)_\pk=k(X_1|X_2)_\gk\qquad(\forall X_1,X_2\in\gk).
\end{align}
$k$ is called the \textbf{Dynkin index} of the embedding $\gk\subset\pk$, and is denoted by $[\pk:\gk]$.

\begin{pp}
	For any $l\in\mathbb Z_{\geq0}$, the level of $\gk$ in $V^l_{\pk}$ is $kl$, where $k=[\pk:\gk]$. Hence we have a unitary VOA embedding $V^{kl}_\gk\subset V^l_\pk$.
\end{pp}
\begin{proof}
	Since the level of $\pk$ in $V^l_{\pk}$ is $l$, the invariant inner product $\bk{\cdot|\cdot}$ on $\pk$ inherited from the inner product of $V^l_\pk$ is $l(\cdot|\cdot)_\pk$. When further restricted to $\gk$, it becomes $kl(\cdot|\cdot)_\gk$. Therefore the level of $\gk$ in $V^l_\pk$ is $kl$.
\end{proof}

\begin{co}
	The only possible values of Dynkin index $[\pk:\gk]$ are $1,2,3,\dots$.
\end{co}

Now let us turn to the second example. Choose $l_1,\dots,l_n\in\mathbb Z_{\geq0}$, and take $V=V^{l_1}_\gk\otimes\cdots\otimes V^{l_n}_\gk$. Then $V(1)=\gk\oplus\cdots\oplus\gk$. Regard $\gk$ as a unitary Lie subalgebra of $V(1)$ by embedding $\gk$ diagonally into $V(1)$: 
\begin{gather}
\gk\subset \gk\oplus\gk\oplus\cdots\oplus\gk,\nonumber\\
X\mapsto (X,X,\cdots,X).\label{eq39}
\end{gather}
The following result is easy to check.
\begin{pp}\label{lb33}
	The level of $\gk$ in $V^{l_1}_\gk\otimes\cdots\otimes V^{l_n}_\gk$ corresponding to the diagonal embedding \eqref{eq39} is $l_1+\cdots+l_n$. So we have $V^{l_1+\cdots+l_n}_\gk\subset V^{l_1}_\gk\otimes\cdots\otimes V^{l_n}_\gk$.
\end{pp}

\subsubsection*{Unitary representations}

We now go back to general theory, and regard $V^l_\gk$ as a unitary vertex subalgebra \emph{inside} $V$. The following proposition is obvious.
\begin{pp}
	Let $W_i$ be a unitary representation of $V$. The map $\pi_i:\gh\rightarrow\End(W_i)$ defined by
	\begin{gather*}
	X(n)\mapsto Y_i(X,n)\qquad(X\in\gk),\\
	K\mapsto l\cdot\id_{W_i}
	\end{gather*}
	is a PER of $\gh$. We say that $V\curvearrowright W_i$ \textbf{restricts to} the PER $\gh\curvearrowright W_i$.
\end{pp}

Clearly, the $V^l_\gk$-module $L_\gk(\lambda,l)$ described in theorem \ref{lb5} restricts to the PER $\gh\curvearrowright L_\gk(\lambda,l)$.

\begin{thm}\label{lb8}
	Let $W_i$ be a unitary $V$-module, $\lambda\in P_+(\gk)$, and $w_\lambda\in W_i$  a non-zero homogeneous vector. If $w_\lambda$ is a $\lambda$-highest weight vector of the restricted PER $\gh\curvearrowright W_i$, then $\gh w_\lambda$ is $V^l_\gk$-invariant, and the action $V^l_\gk\curvearrowright \gh w_\lambda$ is a unitary $V^l_\gk$-module equivalent to $L_\gk(\lambda,l)$.
\end{thm}

\begin{proof}
	Clearly $V^l_\gk$ leaves $\gh w_\lambda$ invariant. By remark \ref{lb6}, $V^l_\gk\curvearrowright W_i$ is a unitary $V^l_\gk$-module. Therefore  $V^l_\gk\curvearrowright \gh w_\lambda$ is also a unitary $V^l_\gk$-module.
	
	Choose a highest weight vector $v_\lambda$ of $L_\gk(\lambda,l)$ satisfying $\lVert v_\lambda\lVert=\lVert w_\lambda\lVert$. Since $V^l_\gk\curvearrowright \gh w_\lambda$ restricts to a highest weight PER of $\gh$ with highest weight $\lambda$ and level $l$, there exists an (unitary) equivalence $\varphi:L_\gk(\lambda,l)\rightarrow \gh w_\lambda$ of PERs of $\gh$ satisfying $\varphi v_\lambda=w_\lambda$. Clearly $\varphi$ intertwines the actions of $V(1)$ on $L_\gk(\lambda,l)$ and on $\gh w_\lambda$. Since $V(1)$ generates $V^l_\gk$, $\varphi$ also intertwines the actions of $V^l_\gk$. Therefore, $\varphi$ is an equivalence between the unitary $V^l_\gk$-modules $L_\gk(\lambda,l)$ and $\gh w_\lambda$.
\end{proof}

The following lemma will be used later.
\begin{lm}\label{lb9}
	In theorem \ref{lb8}, regard $L_\gk(\lambda,l)$ as a $V$-submodule of $W_i$ by identifying $L_\gk(\lambda,l)$ with $\gh w_\lambda$. Let $\{L_n \}$ and $\{L^\gk_n \}$ be the Virasoro operators of $V$ and $V^l_\gk$ respectively. Then there exists $a\in\mathbb R$ such that
	\begin{align}
	L_0w^{(\lambda)}=(L^\gk_0+a)w^{(\lambda)}\qquad(\forall w^{(\lambda)}\in\gh w_\lambda)\label{eq38}
	\end{align}
\end{lm}
\begin{proof}
	Regard $V^l_\gk$ as inside $V$ as usual. It is obvious that $L_0u=L^\gk_0u$ for any $u\in V^l_\gk$. Therefore, by translation property, for any $w^{(\lambda)}\in\gh w_\lambda,u\in V^l_\gk$,
	\begin{align*}
	[L_0,Y_i(u,x)]w^{(\lambda)}=&(Y_i(L_0u,x)+x\frac d{dx}Y_i(u,x))w^{(\lambda)}=(Y_i(L_0^\gk u,x)+x\frac d{dx}Y_i(u,x))w^{(\lambda)}\\
	=&[L^\gk_0,Y_i(u,x)]w^{(\lambda)}.
	\end{align*}
	So $T:=(L_0-L^\gk_0)|_{\gh w_\lambda}$ commutes with the action of $V^l_\gk$ on $\gh w_\lambda$. In particular, $T$ commutes with $L_0|_{\gh w_\lambda}$. So $T$ leaves the lowest energy subspace $\gk w_\lambda$ of $\gh w_\lambda$ invariant. Since $\gk w_\lambda$ is finite-dimensional, we can choose an eigenvalue $a\in\mathbb C$ of $T|_{\gk w_\lambda}$. So $\Ker(T-a)$ is a non-trivial subspace of $\gh w_\lambda$ which is clearly $V^l_\gk$-invariant. Since $\gh w_\lambda$ is an irreducible $V^l_\gk$-module, $\Ker(T-a)$ must be $\gh w_\lambda$. Thus we've proved \eqref{eq38}. Since the eigenvalues of $L_0$ and $L^\gk_0$ are non-negative numbers, $a$ must be real.
\end{proof}

\subsection{Intertwining operators of unitary affine VOAs}\label{lb7}

By theorem \ref{lb5}, to study intertwining operators of $V^l_\gk$, it is sufficient to study the irreducible ones. So we let $L_\gk(\lambda,l),L_\gk(\mu,l),L_\gk(\nu,l)$ be three irreducible representations of $V^l_\gk$, where $\lambda,\mu,\nu\in P_+(\gk)$ are admissible at level $l$. $\mathcal V^l_\gk{\nu\choose\lambda~\mu}$ denotes the vector space of type ${\nu\choose\lambda~\mu}\equiv{L_\gk(\lambda,l)\choose L_\gk(\mu,l)~L_\gk(\nu,l)}$ intertwining operators of $V^l_\gk$. If $\mathcal Y_\alpha\in \mathcal V^l_\gk{\nu\choose\lambda~\mu}$, we call 
\begin{align*}
\Delta_\alpha:=\Delta_\lambda+\Delta_\mu-\Delta_\nu
\end{align*}
the \textbf{conformal weight} of $\mathcal Y_\alpha$.

Consider the lowest energy subspaces $L_\gk(\lambda),L_\gk(\mu),L_\gk(\nu)$ of $L_\gk(\lambda,l),L_\gk(\mu,l),L_\gk(\nu,l)$ respectively. By \eqref{eq4}, for any $u^{(\lambda)}\in L_\gk(\lambda), u^{(\mu)}\in L_\gk(\mu)$, the conformal weight of $\mathcal Y_\alpha(u^{(\lambda)},\Delta_\alpha-1)u^{(\mu)}$ is $\Delta_\nu$. So $\mathcal Y_\alpha(u^{(\lambda)},\Delta_\alpha-1)u^{(\mu)}\in L_\gk(\nu)$. We thus define
\begin{align*}
\Psi:\mathcal V^l_\gk{\nu\choose\lambda~\mu}\rightarrow\Hom(L_\gk(\lambda)\otimes L_\gk(\mu),L_\gk(\nu) ),
\end{align*}
by sending each $\mathcal Y_\alpha\in\mathcal V^l_\gk{\nu\choose\lambda~\mu}$ to the element $\Psi\mathcal Y_\alpha$ satisfying
\begin{gather*}
\Psi\mathcal Y_\alpha:L_\gk(\lambda)\otimes L_\gk(\mu)\rightarrow L_\gk(\nu),\\
\Psi\mathcal Y_\alpha (u^{(\lambda)}\otimes u^{(\mu)})= \mathcal Y_\alpha(u^{(\lambda)},\Delta_\alpha-1)u^{(\mu)}.
\end{gather*}
By Jacobi identity \eqref{eq15}, for any $X\in\gk$ we have
\begin{align*}
X\cdot\mathcal Y_\alpha(u^{(\lambda)},\Delta_\alpha-1)u^{(\mu)}-\mathcal Y_\alpha(u^{(\lambda)},\Delta_\alpha-1)Xu^{(\mu)}=\mathcal Y_\alpha(Xu^{(\lambda)},\Delta_\alpha-1)u^{(\mu)}.
\end{align*}
So $\Psi\mathcal Y_\alpha$ intertwines the actions of $\gk$ on the tensor product $\gk$-module $L_\gk(\lambda)\otimes L_\gk(\mu)$ and on $L_\gk(\nu)$. Thus the image of $\Psi$ is inside $\Hom_\gk(\lambda\otimes\mu,\nu)$. We thus regard $\Psi$ as a map from $\mathcal V^l_\gk{\nu\choose\lambda~\mu}$ to $\Hom_\gk(\lambda\otimes\mu,\nu)$.

The following well known proposition (cf. \cite{TK88} proposition 2.1, \cite{FZ92} theorem 3.2.3) allows us  to study the intertwining operators of affine VOAs using Lie-algebraic methods. To make this paper self-contained, we present a proof here following the argument in \cite{TK88}.
\begin{pp}\label{lb48}
	The map $\Psi:\mathcal V^l_\gk{\nu\choose\lambda~\mu}\rightarrow\Hom_\gk(\lambda\otimes\mu,\nu)$ is injective.	As a consequence, 	the fusion rule $N^\nu_{\lambda\mu}\equiv\dim \mathcal V^l_\gk{\nu\choose\lambda~\mu}$ of $V^l_\gk$ is a finite number, and satisfies
	\begin{align*}
	N^\nu_{\lambda\mu}\leq\dim\Hom_\gk(\lambda\otimes\mu,\nu).
	\end{align*}
\end{pp}	
\begin{proof}
	Suppose $\mathcal Y_\alpha\in\mathcal V^l_\gk{\nu\choose\lambda~\mu}$ and $\Psi\mathcal Y_\alpha=0$. We show that for any $s\in\mathbb R$,
	\begin{align}
	\bk{\mathcal Y_\alpha(w^{(\lambda)},s)w^{(\mu)}|w^{(\nu)}}=0\label{eq37}
	\end{align}
	for any $w^{(\lambda)}\in L_\gk(\lambda,l),w^{(\mu)}\in L_\gk(\mu,l),w^{(\nu)}\in L_\gk(\nu,l)$. By Jacobi identity \eqref{eq24}, it suffices to prove this when $w^{(\lambda)}\in L_\gk(\lambda)$. Let
	
	\begin{align*}
	k(w^{(\mu)},w^{(\nu)})=\Delta_{w^{(\mu)}}-\Delta_\mu+\Delta_{w^{(\nu)}}-\Delta_\nu\in\mathbb Z_{\geq0},
	\end{align*}
	where $\Delta_{w^{(\mu)}},\Delta_{w^{(\nu)}}$ are the conformal weights of $w^{(\mu)},w^{(\nu)}$ respectively. We prove \eqref{eq37} by induction on $k(w^{(\mu)},w^{(\nu)})$.
	
	If $k(w^{(\mu)},w^{(\nu)})=0$, then $w^{(\mu)},w^{(\nu)}$ are also in the lowest energy subspaces. By \eqref{eq4}, $\mathcal Y_\alpha(w^{(\lambda)},s)w^{(\mu)}$ is a homogeneous vector with conformal weight $\Delta_\lambda+\Delta_\mu-s-1=\Delta_\alpha+\Delta_\nu-s-1$. So clearly \eqref{eq37} holds when $s\neq \Delta_\alpha-1$. Since $\Psi\mathcal Y_\alpha=0$, \eqref{eq37} is also true when $s=\Delta_\alpha-1$. So we've proved \eqref{eq37} for any $s$ when $k(w^{(\mu)},w^{(\nu)})=0$.
	
	Now choose an arbitrary $N\in\mathbb Z_{>0}$, and assume that for any $s$, \eqref{eq37} holds whenever $k(w^{(\mu)},w^{(\nu)})=0,1,2,\dots,N-1$. We show that this is also true when $k(w^{(\mu)},w^{(\nu)})=N$. Since $N>0$, either $w^{(\mu)}$ or $w^{(\nu)}$ must have non-lowest conformal weight (energy). We first assume that $w^{(\mu)}\notin L_\gk(\mu)$. Then $w^{(\mu)}$ can be written as a linear combination of vectors of the form $X(-n)w^{(\mu)}_0$, where $X\in\gk,n\in\mathbb Z_{>0}$, and $w^{(\mu)}_0\in L_\gk(\mu,l)$ is homogeneous. So $\Delta_{w^{(\mu)}_0}=\Delta_{w^{(\mu)}}-n$. By Jacobi identity \eqref{eq14},
	\begin{align*}
	&\bk{\mathcal Y_\alpha(w^{(\lambda)},s)X(-n)w^{(\mu)}_0|w^{(\nu)} }\\
	=&\bk {\mathcal Y_\alpha(w^{(\lambda)},s)w^{(\mu)}_0|X^*(n)w^{(\nu)} }-\sum_{l\in\mathbb Z_{\geq0}}{-n\choose l}\bk{\mathcal Y_\alpha(X(l)w^{(\lambda)},-n+s-l)w^{(\mu)}_0|w^{(\nu)} },
	\end{align*}
	Since $w^{(\lambda)}$ has lowest energy, the right hand side of the above equation becomes
	\begin{align*}
	\bk {\mathcal Y_\alpha(w^{(\lambda)},s)w^{(\mu)}_0|X^*(n)w^{(\nu)} }-\bk{\mathcal Y_\alpha(Xw^{(\lambda)},-n+s)w^{(\mu)}_0|w^{(\nu)} }.
	\end{align*}
	Since $k(w^{(\mu)}_0,X^*(n)w^{(\nu)} )<N$ and $k(w^{(\mu)}_0,w^{(\nu)} )<N$, by induction, the above expression is $0$. So we've proved \eqref{eq37} for $k(w^{(\mu)},w^{(\nu)})=N$ when $w^{(\mu)}\notin L_\gk(\mu)$. The case $w^{(\nu)}\notin L_\gk(\nu)$ is treated in a similar way.
\end{proof}

Thus the vector space $\mc V^l_\gk{\nu\choose\lambda~\mu}$ can be identified with its image under the map $\Psi$. For most  examples  considered in this paper, $\Psi$ is in fact surjective. But this is not true in general.   We refer the reader to \cite{FZ92} theorem 3.2.3 for a general description of the image of $\Psi$.

\begin{thm}[Compression principle]\label{lb10}
	Let $V$ be a unitary VOA of CFT type, and $\gk$ a unitary  simple Lie subalgebra of $V(1)$. Let $l$ be the level of $\gk$ in $V$. Let $W_i,W_j,W_k$ be unitary representations of $V$. Assume that $w_\lambda$ (resp. $w_\mu$, $w_\nu$) is a (non-zero) homogeneous highest weight vector of $\gh\curvearrowright W_i$ (resp. $\gh\curvearrowright W_j$, $\gh\curvearrowright W_k$) with highest weight $\lambda$ (resp. $\mu$, $\nu$), and choose isometric $\gk$-module homomorphisms $\varphi_\lambda:L_\gk(\lambda)\rightarrow W_i$ (resp. $\varphi_\mu:L_\gk(\mu)\rightarrow W_j$, $\varphi_\nu:L_\gk(\nu)\rightarrow W_k$) with image $\gk w_\lambda$ (resp. $\gk w_\mu$, $\gk w_\nu$). Let $\varphi_\nu^\dagger:W_k\rightarrow L_\gk(\nu)$ be the formal adjoint of $\varphi_\nu$, i.e., $\varphi_\nu^\dagger$ satisfies
	\begin{align*}
	\bk{\varphi_\nu u^{(\nu)}|w^{(k)}}=\bk{u^{(\nu)}|\varphi_\nu^\dagger w^{(k)} }\qquad(\forall u^{(\nu)}\in L_\gk(\nu), w^{(k)}\in W_k).
	\end{align*}
	If there exist $s\in\mathbb R$ and  a type $k\choose i~j$ intertwining operator $\mathfrak Y$ of $V$, such that the map
	\begin{gather}
	\varphi_\nu^\dagger\mathfrak Y(\varphi_\lambda,s)\varphi_\mu:L_\gk(\lambda)\otimes L_\gk(\mu)\rightarrow L_\gk(\nu),\nonumber\\
	u^{(\lambda)}\otimes u^{(\mu)}\mapsto \varphi_\nu^\dagger\mathfrak Y(\varphi_\lambda u^{(\lambda)},s)\varphi_\mu u^{(\mu)}
	\end{gather}
	is non-zero, then $\varphi_\nu^\dagger\mathfrak Y(\varphi_\lambda,s)\varphi_\mu\in \Hom_\gk(\lambda\otimes\mu,\nu)$, and there exists a type $\nu\choose\lambda~\mu$ intertwining operator $\mathcal Y$ of $V^l_\gk$ such that
	\begin{align}
	\Psi\mathcal Y=\varphi_\nu^\dagger\mathfrak Y(\varphi_\lambda,s)\varphi_\mu.
	\end{align}
	If, moreover, $u^{(\lambda)}\in L_\gk(\lambda)$, and $\mathfrak Y(\varphi_\lambda u^{(\lambda)},x)$ satisfies $r$-th order energy bounds for some $r\geq0$, then  $\mathcal Y(u^{(\lambda)},x)$ also satisfies $r$-th order energy bounds.
\end{thm}

\begin{proof}
	By theorem \ref{lb8}, we can extend $\varphi_\lambda$ (resp. $\varphi_\mu$, $\varphi_\nu$) to an isometric $V^l_\gk$-module homomorphism $\phi_\lambda:L_\gk(\lambda,l)\rightarrow W_i$ (resp. $\phi_\mu:L_\gk(\mu,l)\rightarrow W_j$, $\phi_\nu:L_\gk(\nu,l)\rightarrow W_k$) with image $\gh w_\lambda$ (resp. $\gh w_\mu$, $\gh w_\nu$). Define
	\begin{gather*}
	\mathcal Y_0:L_\gk(\lambda,l)\otimes L_\gk(\mu,l)\rightarrow L_\gk(\nu,l)\{x\},\\
	v^{(\lambda)}\otimes v^{(\mu)}\mapsto \mathcal Y_0(v^{(\lambda)},x)v^{(\mu)} :=\phi_\nu^\dagger\mathfrak Y(\phi_\lambda v^{(\lambda)},x)\phi_\mu v^{(\mu)}.
	\end{gather*}
	Then $\mathcal Y_0$ clearly satisfies the Jacobi identity \eqref{eq5} with the vertex operator of $V^l_\gk$. Note that $\gh w_\lambda,\gh w_\mu,\gh w_\nu$ are graded subspaces of $W_i,W_j,W_k$ respectively. Let $\{L_n \}$ and $\{L_n^\gk \}$ be the Virasoro operators of $V$ and $V^l_\gk$ respectively. Then these subspaces are $L_0$ invariant. By the translation property \eqref{eq6} for $\mathfrak Y$, for any $v^{(\lambda)}\in L_\gk(\lambda,l)$,
	\begin{align*}
	\phi_\nu^\dagger L_0\phi_\nu\cdot \mathcal Y_0(v^{(\lambda)},x)-\mathcal Y_0(v^{(\lambda)},x)\cdot\phi_\mu^\dagger L_0\phi_\mu=\mathcal Y_0(\phi_\lambda^\dagger L_0\phi_\lambda v^{(\lambda)},x)+x\frac d{dx}\mathcal Y_0(v^{(\lambda)},x).
	\end{align*} 
	By lemma \ref{lb9}, the action of $\phi_\lambda^\dagger L_0\phi_\lambda-L^\gk_0$ (resp. $\phi_\mu^\dagger L_0\phi_\mu-L^\gk_0$, $\phi_\nu^\dagger L_0\phi_\nu-L^\gk_0$) on $L_\gk(\lambda,l)$ (resp. $L_\gk(\mu,l)$, $L_\gk(\nu,l)$) is a real constant. So we can easily find $t\in\mathbb R$ such that the map
	\begin{gather*}
	\mathcal Y:L_\gk(\lambda,l)\otimes L_\gk(\mu,l)\rightarrow L_\gk(\nu,l)\{x\},\\
	v^{(\lambda)}\otimes v^{(\mu)}\mapsto \mathcal Y(v^{(\lambda)},x)v^{(\mu)} := x^t\mathcal Y_0(v^{(\lambda)},x)v^{(\mu)}
	\end{gather*}
	satisfies the translation property
	\begin{align*}
	[L^\gk_0,\mathcal Y(v^{(\lambda)},x)]=\mathcal Y(L^\gk_0v^{(\lambda)},x)+x\frac d{dx}\mathcal Y(v^{(\lambda)},x).
	\end{align*}
	So $\mathcal Y$ is a type $\nu\choose\lambda~\mu$ intertwining operator of $V^l_\gk$. One can easily check that $\Psi\mathcal Y=\mathfrak Y$, and that for any homogeneous $v^{(\lambda)}\in L_\gk(\lambda,l)$, $\mathcal Y(v^{(\lambda)},x)$ satisfies $r$-th order energy bounds if $\mathfrak Y(\phi_\lambda v^{(\lambda)},x)$ does.
\end{proof}

In this article, compression principle is more often used in the following simpler form.
\begin{co}\label{lb13}
	Let $V,\gk,l,W_i,W_j,W_k,w_\lambda,w_\mu.w_\nu$ be as in theorem \ref{lb10}. Identify $\gh w_\lambda,\gh w_\mu, \gh w_\nu$ with $L_\gk(\lambda,l),L_\gk(\mu,l),L_\gk(\nu,l)$ respectively. If 
	\begin{align*}
	\dim\Hom_\gk(\lambda\otimes\mu,\nu)\leq1,
	\end{align*}
	and there exists a type $k\choose i~j$ intertwining operator $\mathfrak Y$ of $V$,  such that
	\begin{align*}
	\bk{\mathfrak Y(\gk w_\lambda,x)\gk w_\mu|\gk w_\nu}\neq0,
	\end{align*}
	then 
	\begin{align*}
	N^\nu_{\lambda\mu}=\dim\Hom_\gk(\lambda\otimes\mu,\nu)=1,
	\end{align*}
	where $N^\nu_{\lambda\mu}:=\dim\mathcal V^l_\gk{\nu\choose\lambda~\mu}$ is a fusion rule of $V^l_\gk$. Moreover, let $u^{(\lambda)}\in \gk w_\lambda$, and suppose that  $\mathfrak Y(u^{(\lambda)},x)$ satisfies $r$-th order energy bounds for some $r\geq0$. Then for any type $\nu\choose\lambda~\mu$ intertwining operator $\mathcal Y$ of $V^l_\gk$, $\mathcal Y(u^{(\lambda)},x)$ satisfies $r$-th order energy bounds.
\end{co}

As an application of compression principle, we now show that when the level $l$  is large, the fusion rules of $V^l_\gk$ equal the corresponding tensor product rules of $\gk$, and  the energy bounds conditions for general intertwining operators can be deduced from those for creation operators. To begin with, we choose $\lambda,\mu,\nu\in P_+(\gk)$ admissible at level $l$. As usual, we let $\theta$ be the highest root of $\gk$.

\begin{pp}\label{lb29}
	Let $a=(\lambda|\theta)$. Assume that  $(\mu|\theta)\leq l-a$ or $(\nu|\theta)\leq l-a$. Then $$N^\nu_{\lambda\mu}=\dim\Hom_\gk(\lambda\otimes\mu,\nu),$$ i.e., the map $\Psi:\mathcal V^l_\gk{\nu\choose\lambda~\mu}\rightarrow \Hom_\gk(\lambda\otimes\mu,\nu)$ is bijective. Let  $\mathcal Y_{\kappa(\lambda)}\in\mathcal V^a_\gk{\lambda\choose\lambda~0}$ be the creation operator of $L_\gk(\lambda,a)$. Choose  $u^{(\lambda)}\in L_\gk(\lambda)$. Assume, moreover, that there exists $r\geq0$, such that $\mathcal Y_{\kappa(\lambda)}(u^{(\lambda)},x)$ satisfies $r$-th order energy bounds. Then for any $\mathcal Y_\alpha\in\mathcal V^l_\gk{\nu\choose\lambda~\mu}$, $\mathcal Y_\alpha(u^{(\lambda)},x)$ also satisfies $r$-th order energy bounds.
\end{pp}

\begin{proof}
	Let $b=l-a$. Suppose $(\mu|\theta)\leq b$. Choose nonzero $T\in\Hom_\gk(\lambda\otimes\mu,\nu)$. Since $L_\gk(\nu)$ is irreducible, $TT^*=c\cdot\id_{L_\gk(\nu)}$ for some $c>0$. Assume, without loss of generality, that $c=1$. So $T^*:L_\gk(\nu)\rightarrow L_\gk(\lambda)\otimes L_\gk(\mu)$ is an isometry. Let $V=V^a_\gk\otimes V^b_\gk$, and choose unitary $V$-modules
	\begin{align*}
	W_i=L_\gk(\lambda,a)\otimes L_\gk(0,b),\qquad W_j=L_\gk(0,a)\otimes L_\gk(\mu,b),\qquad W_k=L_\gk(\lambda,a)\otimes L_\gk(\mu,b).
	\end{align*}
	Let $Y_\mu$ be the vertex operator of $V^b_\gk$ on $L_\gk(\mu,b)$. So $Y_\mu\in\mathcal V^b_\gk{\mu\choose 0~\mu}$. Now it is clear that 
	$$\mathfrak Y:=\mathcal Y_{\kappa(\lambda)}\otimes Y_\mu$$
	is a type $k\choose i~j$ intertwining operator of $V$. Let $v_\lambda,v_\mu,v_\nu$ be the highest weight vectors of $L_\gk(\lambda),L_\gk(\mu),L_\gk(\nu)$ of unit length respectively. Choose an isometric $\gk$-module homomorphism $\varphi_\lambda:L_\gk(\lambda)\rightarrow W_i$ (resp. $\varphi_\mu:L_\gk(\mu)\rightarrow W_j$) with image $L_\gk(\lambda)\otimes\Omega$ (resp. $\Omega\otimes L_\gk(\mu)$) satisfying $\varphi_\lambda u^{(\lambda)}=u^{(\lambda)}\otimes\Omega$ (resp. $\varphi_\mu u^{(\mu)}=\Omega\otimes u^{(\mu)}$) for any $u^{(\lambda)}\in L_\gk(\lambda)$ (resp. $u^{(\mu)}\in L_\gk(\mu)$). Note that $L_\gk(\lambda)\otimes L_\gk(\mu)$ is the lowest energy subspace of $W_k$. We extend $T^*:L_\gk(\nu)\rightarrow L_\gk(\lambda)\otimes L_\gk(\mu)$ to $\varphi_\nu:L_\gk(\nu)\rightarrow W_k$. So $\varphi_\nu$ is also an isometric $\gk$-module homomorphism, and the image of $\varphi_\nu$ has the lowest conformal weight (energy). It follows that $\varphi_\nu v_\nu$ is a highest weight vector of $\gh\curvearrowright W_k$ with highest weight $\nu$. Now consider the $\gk$-module homomorphism
	\begin{gather*}
	\varphi_\nu^\dagger\mathfrak Y(\varphi_\lambda,-1)\varphi_\mu:L_\gk(\lambda)\otimes L_\gk(\mu)\rightarrow L_\gk(\nu).
	\end{gather*}
	For any $u^{(\lambda)}\in L_\gk(\lambda),u^{(\mu)}L_\gk(\mu)$, we compute
	\begin{align*}
	\varphi_\nu^\dagger\mathfrak Y(\varphi_\lambda u^{(\lambda)},-1)\varphi_\mu u^{(\mu)}=&\varphi_\nu^\dagger\mathfrak Y(u^{(\lambda)}\otimes\Omega,-1)(\Omega\otimes u^{(\mu)}) =\varphi_\nu^\dagger (\mathcal Y_{\kappa(\lambda)}(u^{(\lambda)},-1)\otimes \id)(\Omega\otimes u^{(\mu)})\\
	=&\varphi^\dagger_\nu(u^{(\lambda)}\otimes u^{(\mu)})=T(u^{(\lambda)}\otimes u^{(\mu)}).
	\end{align*}
	Hence $\varphi_\nu^\dagger\mathfrak Y(\varphi_\lambda ,-1)\varphi_\mu=T$. By theorem \ref{lb10}, there exists $\mathcal Y_\alpha\in\mathcal V^l_\gk{\nu\choose\lambda~\mu}$ satisfying $\Psi\mathcal Y_\alpha=T$, and if for a vector $u^{(\lambda)}\in L_\gk(\lambda)$, $\mathcal Y_{\kappa(\lambda)}(u^{(\lambda)},x)\otimes\id=\mathfrak Y(u^{(\lambda)}\otimes\Omega,x)$ satisfies $r$-th order energy bounds, then the same property holds for $\mathcal Y_\alpha(u^{(\lambda)},x)$.  
	
	Now suppose that $(\nu|\theta)\leq l-a$. Since $(\overline\lambda|\theta)=(\lambda|\theta)=a$, we have $N^\mu_{\overline\lambda\nu}=\dim\Hom_\gk(\overline\lambda\otimes\nu,\mu)$. Therefore
	\begin{align*}
	N^\nu_{\lambda\mu}=N^\mu_{\overline\lambda\nu}=\dim\Hom_\gk(\overline\lambda\otimes\nu,\mu)=\dim\Hom_\gk(\lambda\otimes\mu,\nu).
	\end{align*}
	Since the conjugate intertwining operator $\mathcal Y_{\overline{\kappa(\lambda)}}(C_\lambda u^{(\lambda)},x)$ of $\mathcal Y_{\kappa(\lambda)}(u^{(\lambda)},x)$ clearly satisfies $r$-th order energy bounds, by what we've proved above, $\mathcal Y_{\alpha^*}(C_\lambda u^{(\lambda)},x)$ satisfies $r$-th order energy bounds for any $\mathcal Y_\alpha\in\mathcal V^l_\gk{\nu\choose\lambda~\mu}$. By proposition \ref{lb14}, $\mathcal Y_\alpha(u^{(\lambda)},x)$ also satisfies $r$-th order energy bounds.
\end{proof}

The situation becomes more complicated when $l$ is not so large, say, when $(\mu|\theta)\leq l\leq (\lambda|\theta)+(\mu|\theta)$. In this case, the following technical lemma will be helpful. The reader might temporarily skip this lemma, and return to it later.

\begin{lm}\label{lb24}
	Let $l\in\mathbb Z_{\geq0}$. Choose $\lambda,\mu,\nu,\rho,\mu_1,\nu_1\in P_+(\gk)$ admissible at level $l$, and choose highest-weight vectors $v_{\mu_1}\in L_\gk(\mu_1),v_{\nu_1}\in L_\gk(\nu_1)$. Set $a=\max\{(\lambda|\theta),(\mu_1|\theta),(\nu_1|\theta) \}$.
	Assume that the weight spaces of $L_\gk(\lambda)$ have dimensions at most $1$, that 
	\begin{gather}
	a\leq l,\qquad (\rho|\theta)\leq l-a,
	\end{gather}
	that $\dim\mathcal V^a_\gk{\nu_1\choose\lambda~\mu_1}=1$, and that there exists $r\geq0$, such that for any $\mathcal Y_\sigma\in\mathcal V^a_\gk{\nu_1\choose\lambda~\mu_1}$ and $u^{(\lambda)}\in L_\gk(\lambda)$, $\mathcal Y_\sigma(u^{(\lambda)},x)$ satisfies $r$-th order energy bounds. If one of the  three conditions stated below holds, then $N^\nu_{\lambda\mu}:=\dim\mathcal V^l_\gk{\nu\choose\lambda~\mu}=1$, and for any $\mathcal Y_\alpha\in\mathcal V^a_\gk{\nu\choose\lambda~\mu}$ and $u^{(\lambda)}\in L_\gk(\lambda)$, $\mathcal Y_\alpha(u^{(\lambda)},x)$ satisfies $r$-th order energy bounds. The three conditions are:\\
	(a) $\mu=\mu_1+\rho$ and $\nu=\nu_1+\rho$.\\
	(b)  $\mu=\mu_1+\rho$, $L_\gk(\nu)$ is equivalent to a submodule of $L_\gk(\nu_1)\otimes L_\gk(\rho)$,  the weight spaces of $L_\gk(\nu_1)$ have dimensions at most $1$, and 
	\begin{align}
	T\big(L_\gk(\lambda)[\nu-\mu]\otimes v_{\mu_1}\big)\neq0\label{eq44}
	\end{align}
	for any non-zero  $T\in\Hom_\gk(\lambda\otimes\mu_1,\nu_1)$.\\
	(c)  $\nu=\nu_1+\rho$, $L_\gk(\mu)$ is equivalent to a submodule of $L_\gk(\mu_1)\otimes L_\gk(\rho)$,  the weight spaces of $L_\gk(\mu_1)$ have dimensions at most $1$, and for any non-zero  $T\in\Hom_\gk(\lambda\otimes\mu_1,\nu_1)$, there exist $u^{(\lambda)}_{\nu-\mu}\in L_\gk(\lambda)[\nu-\mu],u^{(\mu_1)}\in L_\gk(\mu_1)$, such that
	\begin{align}
	\langle T(u^{(\lambda)}_{\nu-\mu}\otimes u^{(\mu_1)})|v_{\nu_1}\rangle\neq0. \label{eq45}
	\end{align}
\end{lm}
\begin{proof}
	Let $b=l-a$. Then $\rho$ is admissible at level $b$. Let $V=V^a_\gk\otimes V^b_\gk$, and choose unitary $V$-modules
	\begin{align*}
	W_i=L_\gk(\lambda,a)\otimes L_\gk(0,b),\qquad W_j=L_\gk(\mu_1,a)\otimes L_\gk(\rho,b),\qquad W_k=L_\gk(\nu_1,a)\otimes L_\gk(\rho,b).
	\end{align*}
Choose any non-zero $\mathcal Y_\sigma\in\mathcal V^a_\gk{\nu_1\choose\lambda~\mu_1}$.	Then
	$$\mathfrak Y=\mathcal Y_\sigma\otimes Y_\rho$$
	is a type $k\choose i~j$ intertwining operator of $V$. Consider the diagonal Lie subalgebra $\gk\subset V(1)=\gk\oplus\gk$. Then $\gk$ has level $l=a+b$ in $V$.
	
	Let $v_\lambda$ and $v_\rho$ be the highest weight vectors of $L_\gk(\lambda)$ and $L_\gk(\rho)$ respectively. Let $w_\lambda=v_\lambda\otimes\Omega$. Then $w_\lambda$ is a level $l$ highest weight vector of $\gh\curvearrowright W$ with highest weight $\lambda$. Since $\gk w_\lambda=L_\gk(\lambda)\otimes\Omega$, $\mathfrak Y(w^{(i)},x)$ satisfies $r$-th order energy bounds for any $w^{(i)}\in\gk w_\lambda$. In the following, we shall construct, for cases (a) and (b), highest weight vectors $w_\mu,w_\nu$ of the level $l$ $\gh$-modules $W_j,W_k$ with highest weights $\mu,\nu$ respectively, and show that
	\begin{align}
	\bk{\mathfrak Y(\gk w_\lambda,x)\gk w_\mu|\gk w_\nu}\neq0.\label{eq43}
	\end{align}
	Then corollary \ref{lb13} will imply the statement  claimed in this proposition. Case (c) will follow from case (b).

	(a) Let $w_\mu=v_{\mu_1}\otimes v_\rho$ and $w_\nu=v_{\nu_1}\otimes v_\rho$. Then these two vectors are highest weight vectors with weights $\mu$ and $\nu$ respectively. Note that by propositions \ref{lb47} and \ref{lb48},
	\begin{align*}
	1=\dim\mathcal V^a_\gk{\nu_1\choose\lambda~\mu_1}\leq\dim\Hom_\gk(\lambda\otimes\mu_1,\nu_1)\leq\dim L_\gk(\lambda)[\nu_1-\mu_1]=\dim L_\gk(\lambda)[\nu-\mu]\leq 1.
	\end{align*}
	So all the intermediate terms in the  above inequality equal $1$. We choose a non-zero vector $u^{(\lambda)}_{\nu-\mu}\in L_\gk(\lambda)[\nu-\mu]$. Then $u^{(\lambda)}_{\nu-\mu}\otimes\Omega\in\gk w_\lambda$. Let $\Delta_\sigma$ be the conformal weight of $\mathcal Y_\sigma$, and set $s=\Delta_\sigma-1$.  We compute
	\begin{align*}
	&\bk{\mathfrak Y(u^{(\lambda)}_{\nu-\mu}\otimes\Omega,s)w_\mu|w_\nu}\\
	=&\bk{\mathfrak Y(u^{(\lambda)}_{\nu-\mu}\otimes\Omega,s)(v_{\mu_1}\otimes v_\rho)|v_{\nu_1}\otimes v_\rho}=\lVert v_\rho\lVert^2 \bk{\mathcal Y_\sigma(u^{(\lambda)}_{\nu-\mu},s)v_{\mu_1}|v_{\nu_1} },
	\end{align*}
	which must be non-zero since the map $\Gamma\circ\Psi:\mathcal V^a_\gk{\nu_1\choose\lambda~\mu_1}\rightarrow  L_\gk(\lambda)[\nu_1-\mu_1]^*$ is injective. Thus \eqref{eq43} follows.
	
	(b) Condition \eqref{eq44} implies that $L_\gk(\lambda)[\nu-\mu]$ has positive dimension, which must be one. We again let $u^{(\lambda)}_{\nu-\mu}$ be a non-zero vector of $L_\gk(\lambda)[\nu-\mu]$. So $u^{(\lambda)}_{\nu-\mu}\otimes\Omega$ is inside $\gk w_\lambda$. Let $w_\mu=v_{\mu_1}\otimes v_\rho$ be a $\mu$-highest weight vector of $\gh\curvearrowright W_j$. We now define $w_\nu$. Since the non-trivial weight spaces of $L_\gk(\nu_1)$ have dimension $1$, $\dim\Hom_\gk(\nu_1\otimes\rho,\nu)=1$. Thus, up to scalar multiplication, there is a unique $\nu$-highest weight vector of $\gk\curvearrowright L_\gk(\nu_1)\otimes L_\gk(\rho)$. We let $w_\nu$ be such a vector. Identify $L_\gk(\nu_1)\otimes L_\gk(\rho)$ as the lowest energy (conformal weight) subspace of $W_k$. Then $w_\nu$ is a $\nu$-highest weight vector of $\gh\curvearrowright W_k$.

Since $\mathcal Y_\sigma\in\mathcal V^a_\gk{\nu_1\choose\lambda~\mu_1}$ is non-zero, $T=\Psi(\mathcal Y_\sigma)$ is a non-zero element in $\Hom_\gk(\lambda\otimes\mu_1,\nu_1)$.  Let again $s=\Delta_\sigma-1$.  Then the vector
	\begin{align}
	u^{(\nu_1)}_{\nu-\rho}:=\mathcal Y_\sigma(u^{(\lambda)}_{\nu-\mu},s)v_{\mu_1}
	\end{align}
	is non-zero by condition \eqref{eq44}. Clearly $u^{(\nu_1)}_{\nu-\rho}$ is a weight vector of $L_\gk(\nu_1)$ with weight $\nu-\mu+\mu_1=\nu-\rho$. So $u^{(\nu_1)}_{\nu-\rho}\in L_\gk(\nu_1)[\nu-\rho]$. Since the map $\Gamma:\Hom_\gk(\nu_1\otimes\rho,\nu)\rightarrow L_\gk(\nu_1)[\nu-\rho]^*$ is injective, we must have $\bk{u^{(\nu_1)}_{\nu-\rho}\otimes v_\rho|w_\nu}\neq0$. We now compute
	\begin{align*}
	&\bk{\mathfrak Y(u^{(\lambda)}_{\nu-\mu}\otimes\Omega,s)w_\mu|w_\nu}=\bk{\mathfrak Y(u^{(\lambda)}_{\nu-\mu}\otimes\Omega,s)(v_{\mu_1}\otimes v_\rho)|w_\nu}\\
	=&\bk{\mathcal Y_\sigma(u^{(\lambda)}_{\nu-\mu},s)v_{\mu_1}\otimes v_\rho|w_\nu}=\bk{u^{(\nu_1)}_{\nu-\rho}\otimes v_\rho|w_\nu}\neq0.
	\end{align*}
	Thus \eqref{eq43} is proved.
	
	(c) Choose an arbitrary non-zero $T\in\Hom_\gk(\lambda\otimes\mu_1,\nu_1)$. Then $\Ad T\in\Hom_\gk(\overline\lambda\otimes\nu_1,\mu_1)$, and the condition that \eqref{eq45} holds for some $u^{(\lambda)}_{\nu-\mu}\in L_\gk(\lambda),u^{(\mu_1)}\in L_\gk(\mu_1)$ is  equivalent to 
	\begin{align}
	\Ad T\big(L_\gk(\overline\lambda)[\mu-\nu]\otimes v_{\nu_1}\big)\neq0.
	\end{align}
	Here we use the fact that $C_\lambda\cdot L_\gk(\lambda)[\eta]=L_\gk(\overline\lambda)[-\eta]$ for any $\eta\in\mathfrak h^*$, which is easy to check using the definition of dual representations, and we let in particular $\eta$ be $\nu-\mu$. Now, $\dim\mathcal V^a_\gk{\mu_1\choose\overline\lambda~\nu_1}=\dim\mathcal V^a_\gk{\nu_1\choose\lambda~\mu_1}=1$. By proposition \ref{lb14}, $\mathcal Y_{\sigma^*}(C_\lambda u^{(\lambda)},x)$ satisfies $r$-th order energy bounds . Thus case (b) implies that $N^\mu_{\overline\lambda\nu}=1$, and for any $\mathcal Y_\alpha\in\mathcal V^a_\gk{\nu\choose\lambda~\mu}$ and $u^{(\lambda)}\in L_\gk(\lambda)$, $\mathcal Y_{\alpha^*}(C_\lambda u^{(\lambda)},x)$ satisfies $r$-th order energy bounds. Therefore, $N^\nu_{\lambda\mu}=N^\mu_{\overline\lambda\nu}=1$ and, by proposition \ref{lb14} again, $\mathcal Y_\alpha(u^{(\lambda)},x)$ satisfies $r$-th order energy bounds.
\end{proof}

\section{Type $B$}

\subsection{$\so_{2n+1}\subset\so_{2n+2}$}\label{lb21}

Let $n=2,3,4,\dots$. $\so_{2n+2}$ is the unitary Lie algebra of skew-symmetric linear operators on $\mathbb C^{2n+2}$. Likewise, elements in $\so_{2n+1}$ are skew-symmetric linear operators on $\mathbb C^{2n+1}$.  Let $e_1,e_2,\dots,e_{2n+2}$ be the standard orthonormal basis of $\mathbb C^{2n+2}$. Regard $\mathbb C^{2n+1}$ as a subspace of $\mathbb C^{2n+2}$ spanned by $e_1,\dots,e_{2n+1}$. For any $X\in\so_{2n+1}$, we extend it to a linear operator  on $\mathbb C^{2n+2}$ by setting $Xe_{2n+2}=0$. In this way, we embed $\so_{2n+1}$ as a unitary Lie subalgebra of $\so_{2n+2}$. 

For any $i,j=1,2,\dots,2n+2$, we choose $E_{i,j}\in\End(\mathbb C^{2n+2})$ to satisfy $E_{i,j}e_k=\delta_{j,k}e_i$ ($\forall k=1,2,\dots,2n+2$), then for any $X=\sum c_{ij}E_{i,j}\in\End(\mathbb C^{2n+2})$, $X\in\so_{2n+2}$ if and only if $c_{ij}=-c_{ji}$ for all $i,j=1,\dots,2n+2$. If $X\in\so_{2n+2}$, then $X\in\so_{2n+1}$ if and only if $c_{k,2n+2}=c_{2n+2,k}=0$ for all $k=1,\dots 2n+2$.

Write $\gk=\so_{2n+1}$ and $\pk=\so_{2n+2}$ for simplicity. In this section, we collect some basic facts about these two Lie algebras. First we choose Cartan subalgebras. See, for instance, \cite{FH91} for more details. For any $i=1,\dots,n+1$, we let
\begin{align*}
S_i=-\mathbf iE_{2i+1,2i+2}+\mathbf i E_{2i+2,2i+1}.
\end{align*}
Here $\mathbf i$ is the imaginary number $\sqrt{-1}$, not to be confused with the index $i$. Let  $\tk=\Span_\mathbb C\{S_1,\dots,S_{n+1} \}$, and $\hk=\Span_\mathbb C\{S_1,\dots,S_n\}$. Then $\tk$ is a Cartan subalgebra of $\pk$, and $\hk$ is a Cartan subalgebra of $\gk$, which is a unitary Lie subalgebra of $\tk$ at the same time. Let $\{\theta_1,\dots,\theta_{n+1} \}$ be the dual basis of $\{S_1,\dots,S_{n+1}\}$, i.e., for any $i,j=1,\dots,n+1$, $\theta_i\in\tk^*$ satisfies $\bk{\theta_i,S_j}=\delta_{ij}$. Let $\vartheta_i,\dots,\vartheta_n$ be respectively the restrictions of $\theta_1,\dots,\theta_n$ on $\hk$. Then $\{\vartheta_i,\dots,\vartheta_n \}$ is a basis of $\hk^*$, which is the dual basis of $\{S_1,\dots,S_n \}\subset \hk$. Clearly the restriction of $\theta_{n+1}$ on $\hk$ is $0$.

We also introduce the  relation $\succsim$ between real numbers to simplify  our discussion. Given $a,b\in\mathbb R$, we say  $b\succsim a$ if $b-a\in\{0,1,2,\dots \}$.\\

\subsubsection*{Facts about $\pk=\so_{2n+2}$}
\noindent
Simple roots: $\theta_1-\theta_2,\theta_2-\theta_3,\dots,\theta_{n-1}-\theta_n,\theta_n\pm\theta_{n+1}$.\\
Roots: $\pm\theta_i\pm\theta_j$ ($i\neq j$).\\
$\pk_+$: matrices whose non-zero entries are above the diagonal $2\times2$ blocks $S_1,\dots,S_{n+1}$.\\
Fundamental weights:
\begin{gather*}
\theta_1+\cdots+\theta_i\qquad(i=1,2,\dots,n-1),\\
\varsigma_+=\frac 1 2(\theta_1+\cdots+\theta_n+\theta_{n+1}),\\
\varsigma_-=\frac 1 2(\theta_1+\cdots+\theta_n-\theta_{n+1}).
\end{gather*}
Dominant integral weights: $\lambda=f_1\theta_1+\cdots+f_{n+1}\theta_{n+1}$, where
\begin{align}
&f_1\succsim\cdots\succsim f_n\succsim |f_{n+1}|\succsim 0\qquad (\text{single-valued}),\label{eq50}\\
\textrm{or }&f_1\succsim\cdots\succsim f_n\succsim |f_{n+1}|\succsim \frac 1 2\qquad(\text{double-valued}).\label{eq51}
\end{align}
Highest root: $\theta_1+\theta_2$.\\
The normalized invariant inner product satisfies $(S_i|S_j)=(\theta_i|\theta_j)=\delta_{i,j}$.\\
Weights of $L_\pk(\theta_1)$: $\pm\theta_1,\dots,\pm\theta_n,\pm\theta_{n+1}$.\\
Weights of $L_\pk(\varsigma_\pm)$: 
\begin{gather*}
\frac {(-1)^{\sigma_1}\theta_1+\cdots+(-1)^{\sigma_{n+1}}\theta_{n+1}}2\qquad (\sigma_1,\dots,\sigma_{n+1}\in\mathbb Z),\\
\sigma_1+\cdots+\sigma_{n+1}\in \left\{ \begin{array}{ll}
2\mathbb Z & \textrm{for $\varsigma_+$}\\
2\mathbb Z+1 & \textrm{for $\varsigma_-$.}\\
\end{array} \right.
\end{gather*}
The weight spaces of $L_\pk(\theta_1),L_\pk(\varsigma_+)$, and $L_\pk(\varsigma_-)$ have dimensions at most $1$.

\subsubsection*{Facts about $\gk=\so_{2n+1}$}
\noindent
Simple roots: $\vartheta_1-\vartheta_2,\vartheta_2-\vartheta_3,\dots,\vartheta_{n-1}-\vartheta_n,\vartheta_n$.\\
Roots: $\pm\vartheta_i\pm\vartheta_j$ ($i\neq j$), $\pm\vartheta_i$.\\
$\gk_+=\pk_+\cap\gk$.\\
Fundamental weights:
\begin{gather*}
\vartheta_1+\cdots+\vartheta_i\qquad(i=1,2,\dots,n-1),\\
\varsigma=\frac 1 2(\vartheta_1+\cdots+\vartheta_n).
\end{gather*}
Dominant integral weights: $\lambda=f_1\vartheta_1+\cdots+f_n\vartheta_n$, where
\begin{align}
&f_1\succsim\cdots\succsim f_n\succsim 0\qquad (\text{single-valued}),\label{eq48}\\
\textrm{or }&f_1\succsim\cdots\succsim f_n\succsim \frac 1 2\qquad(\text{double-valued}).\label{eq49}
\end{align}
Highest root: $\vartheta_1+\vartheta_2$.\\
The normalized invariant inner product satisfies $(S_i|S_j)=(\vartheta_i|\vartheta_j)=\delta_{i,j}$.\\
Weights of $L_\gk(\vartheta_1)$: $\pm\vartheta_1,\dots,\pm\vartheta_n,0$.\\
Weights of $L_\gk(\varsigma)$: 
\begin{gather}
\frac {(-1)^{\sigma_1}\vartheta_1+\cdots+(-1)^{\sigma_n}\vartheta_n}2\qquad (\sigma_1,\dots,\sigma_n\in\mathbb Z).\label{eq47}
\end{gather}
The weight spaces of $L_\gk(\vartheta_1)$ and $L_\gk(\varsigma)$ have dimensions at most $1$.\\

It is clear that the Dynkin index of $\gk\subset\pk$ is $1$.

\subsection{Main result}

The dominant integral weights of $\pk$ admissible at level $1$ are $\theta_1,\varsigma_+,\varsigma_-$. Choose $l=1,2,\dots$. The following results were essentially proved in \cite{TL04}. The case $l=1$ should be proved using Frenkel-Kac construction and lattice VOAs, see corollary \ref{lb41}. The higher level cases can be reduced to the level $1$ case using compression principle, see \cite{TL04} chapter VI sections 2 and 3.

\begin{thm}\label{lb16}
Let $\pk=\mathfrak{so}_{2n}$ ($n\geq3$), and choose $\lambda,\mu,\nu\in P_+(\pk)$ admissible at level $l$.

(a) If $\lambda=\theta_1$ or $\varsigma_\pm$, then the maps $\Psi:\mathcal V^l_\pk{\nu\choose\lambda~\mu}\rightarrow \Hom_\pk(\lambda\otimes\mu,\nu)$ and $\Gamma:\Hom_\pk(\lambda\otimes\mu,\nu)\rightarrow L_\pk(\lambda)[\nu-\mu]^*$ are bijective.

(b) The $V^l_\pk$-modules $L_\pk(\varsigma_+,l)$ and $L_\pk(\varsigma_-,l)$ are generating (see definition \ref{lb42}).

(c) If $\lambda=\varsigma_\pm$, then any $\mathcal Y_\alpha\in\mathcal V^l_\pk{\nu\choose\lambda~\mu}$ is energy-bounded.

(d) If $\lambda=\theta_1$, then for any $\mathcal Y_\alpha\in\mathcal V^l_\pk{\nu\choose\lambda~\mu}$ and $u^{(\lambda)}\in L_\pk(\lambda)$, $\mathcal Y_\alpha(u^{(\lambda)},x)$ satisfies $0$-th order energy bounds. In particular, $\mathcal Y_\alpha$ is energy-bounded.
\end{thm}

Now we assume that $\lambda,\mu,\nu$ are dominant integral weights of $\gk$ admissible at level $l$, and that $\lambda=\vartheta_1$ or $\varsigma$. Then the weight spaces of $L_\gk(\lambda)$ have dimensions at most $1$. We shall calculate the fusion rule $N^\nu_{\lambda\mu}=\dim\mathcal V^l_\gk{\nu\choose\lambda~\mu}$, and establish the energy bounds condition for type $\nu\choose\lambda~\mu$ intertwining operators of $V^l_\gk$. Since $N^\nu_{\lambda\mu}\leq\dim L_\gk(\lambda)[\nu-\mu]\leq1$, in order to know when $N^\nu_{\lambda\mu}$ equals $1$, we can assume that $\nu-\mu$ is a weight of $L_\gk(\lambda)$.

\subsubsection*{Case $\lambda=\varsigma$}

First we assume that $\mu$ is single-valued. Then, since $\nu-\mu$ is a weight of $L_\gk(\lambda)$, which can be written as \eqref{eq47}, $\mu$ must be double-valued. Write
\begin{gather*}
\nu-\mu=\frac 12((-1)^{\sigma_1}\vartheta_1+\cdots+(-1)^{\sigma_n}\vartheta_n),\\
\mu=f_1\vartheta_1+\cdots+f_n\vartheta_n,\\
\nu=g_1\vartheta_1+\cdots+g_n\vartheta_n.
\end{gather*}
Then $f_1,\dots,f_n$ satisfy \eqref{eq48}, $g_1,\dots,g_n$ satisfy \eqref{eq49}, and $g_i=f_i+(-1)^{\sigma_i}/2$ ($i=1,\dots,n$). Let
\begin{gather*}
\wtd\lambda=\left\{ \begin{array}{ll}
\varsigma_+ & (\textrm{if }\sigma_1+\cdots+\sigma_n\textrm{ is even})\\
\varsigma_- & (\textrm{if }\sigma_1+\cdots+\sigma_n\textrm{ is odd})\\
\end{array} \right.\\
\wtd\mu=f_1\theta_1+\cdots+f_n\theta_n,\\
\wtd\nu=g_1\theta_1+\cdots+g_n\theta_n+\frac {\theta_{n+1}}2 .
\end{gather*}
Then by \eqref{eq50} and \eqref{eq51}, $\wtd\mu$ and $\wtd\nu$ are dominant integral weights of $\pk$, which are clearly admissible at level $l$, and
\begin{gather*}
\wtd\nu-\wtd\mu=\frac 12((-1)^{\sigma_1}\theta_1+\cdots+(-1)^{\sigma_n}\theta_n+\theta_{n+1})
\end{gather*}
is a weight of $L_\pk(\wtd\lambda)$. So by theorem \ref{lb16}-(a), $\dim \mathcal V^l_\pk{\wtd\nu\choose\wtd\lambda~\wtd\mu}=1$. 

Choose a non-zero $\mathfrak Y\in \mathcal V^l_\pk{\wtd\nu\choose\wtd\lambda~\wtd\mu}$. Then $\mathfrak Y$ is energy-bounded by theorem \ref{lb16}-(c). Let $w_\lambda,w_\mu$ and $w_\nu$ be respectively highest weight vectors of $\ph\curvearrowright L_\pk(\wtd\lambda,l),L_\pk(\wtd\mu,l),L_\pk(\wtd\nu,l)$. Then they are also highest weight vectors of the action of $\gh$ on these vector spaces with weights $\lambda,\mu,\nu$ respectively. If we can show that
\begin{align}
\bk{\mathfrak Y(\gk w_\lambda,x)\gk w_\mu|\gk w_\nu}\neq0,\label{eq52}
\end{align}
then by corollary \ref{lb13}, $\dim\mathcal V^l_\gk{\nu\choose\lambda\mu}=1$, and for any $\mathcal Y\in\mathcal V^l_\gk{\nu\choose\lambda~\mu}$ and $ u^{(\lambda)}\in L_\gk(\lambda)$, $\mathcal Y(u^{(\lambda)},x)$ is energy-bounded. Therefore, by theorems \ref{lb17} and \ref{lb18}, $\mathcal Y$ is energy bounded.

To prove \eqref{eq52}, we first choose a non-zero $u^{(\wtd\lambda)}\in L_\pk(\wtd\lambda)[\wtd\nu-\wtd\mu]$. Since the map $\Gamma\Psi:\mathcal V^l_\pk{\wtd\nu\choose\wtd\lambda~\wtd\mu}\rightarrow L_\pk(\wtd\lambda)[\wtd\nu-\wtd\mu]^*$ is injective, we can choose a suitable  $s\in\mathbb R$ such that $\bk{\mathfrak Y(u^{(\wtd\lambda)},s) w_\mu| w_\nu}\neq0$. Since we know all the possible weights of  $\gk w_\lambda\simeq L_\gk(\varsigma)$, and since each weigh space has dimension no greater than $1$, we can easily compute that $\gk w_\lambda$ has dimension $2^n$. But $L_\pk(\wtd\lambda)$ also has dimension $2^n$, which can be computed in a similar way. Therefore $\gk w_\lambda=L_\pk(\wtd\lambda)$, and hence $u^{(\wtd\lambda)}\in\gk w_\lambda$. Thus \eqref{eq52} is proved.

The case $\mu$ is double-valued (and hence $\nu$ is single-valued) can be treated in a similar way. Finally, using the knowledge of fusion rules, one can easily check that $L_\gk(\varsigma,l)$ is a generating $V^l_\gk$-module.

\subsubsection*{Case $\lambda=\vartheta_1$}

As we have seen, our analysis of the case $\lambda=\varsigma$ relies on the fact that the action $\pk\curvearrowright L_\pk(\varsigma_\pm)$, when restricted to $\gk$, is also irreducible and equivalent to $L_\gk(\varsigma)$. This is achieved by a calculation of dimensions. However, we don't have a similar result when $\lambda=\vartheta_1$. Let $v_{\theta_1}$ be a highest weight vector of $\pk\curvearrowright L_\pk(\theta_1)$. Then $v_{\theta_1}$ is  a weight-$\vartheta_1$ highest weight vector of $\gk\curvearrowright L_\pk(\theta_1)$, and hence $\gk v_{\theta_1}$ is equivalent to $L_\gk(\vartheta_1)$. We have $\dim\gk v_{\theta_1}=\dim L_\gk(\vartheta_1)=2n+1$, and $\dim L_\pk(\theta_1)=2n+2$. So $\gk v_{\theta_1}$ is a proper subspace of $L_\pk(\theta_1)$. We still have the following:

\begin{lm}\label{lb19}
Identify $\gk v_{\theta_1}$ with $L_\gk(\vartheta_1)$.	

(a) For any $i=1,2,\dots,n$, we have $L_\gk(\vartheta_1)[\vartheta_i]=L_\pk(\theta_1)[\theta_i]$ and $L_\gk(\vartheta_1)[-\vartheta_i]=L_\pk(\theta_1)[-\theta_i]$.

(b) For any non-zero $u_0\in L_\gk(\vartheta_1)[0]$, there exist non-zero vectors $u_+\in L_\pk(\theta_1)[\theta_{n+1}]$ and $u_-\in L_\pk(\theta_1)[-\theta_{n+1}]$, such that $u_0=u_++u_-$.
\end{lm}

\begin{proof}
It is clear that for any $i=1,\dots,n$, $L_\pk(\theta_1)[\theta_i]$ is the weight-$\vartheta_i$ subspace of $\gk\curvearrowright L_\pk(\theta_1)$,  $L_\pk(\theta_1)[-\theta_i]$ is the weight-$(-\vartheta_i)$ subspace of $\gk\curvearrowright L_\pk(\theta_1)$ and  $L_\pk(\theta_1)[\theta_{n+1}]\oplus L_\pk(\theta_1)[-\theta_{n+1}]$ is the weight $0$ subspace of $\gk\curvearrowright L_\pk(\theta_1)$. Thus
\begin{gather}
L_\gk(\vartheta_1)[\vartheta_i]\subset L_\pk(\theta_1)[\theta_i],\label{eq53}\\
L_\gk(\vartheta_1)[-\vartheta_i]\subset L_\pk(\theta_1)[-\theta_i],\label{eq54}\\
L_\gk(\vartheta_1)[0]\subset L_\pk(\theta_1)[\theta_{n+1}]\oplus L_\pk(\theta_1)[-\theta_{n+1}].
\end{gather}
Since both sides of \eqref{eq53} and \eqref{eq54} have dimension $1$, we actually have $L_\gk(\vartheta_1)[\vartheta_i]=L_\pk(\theta_1)[\theta_i]$ and $L_\gk(\vartheta_1)[-\vartheta_i]=L_\pk(\theta_1)[-\theta_i]$.

Now identify $L_\pk(\theta_1)$ with the standard representation $\so_{2n+2}\curvearrowright\mathbb C^{2n+2}$, and let $e_1,\dots,e_{2n+2}$ be the standard orthonormal basis. Then it is clear that $L_\gk(\vartheta_1)[0]=\mathbb C\cdot e_{2n+1}$. Clearly $e_{2n+1}$ is not a weight vector of $\pk\curvearrowright L_\pk(\theta)$. So
$$L_\gk(\vartheta_1)[0]\nsubset L_\pk(\theta_1)[\theta_{n+1}]\cup L_\pk(\theta_1)[-\theta_{n+1}].$$
This proves (b).
\end{proof}

We are now ready to compute fusion rules and prove the energy bounds condition. Write $\mu=\sum_{j=1}^nf_j\vartheta_j$ and $\nu=\sum_{j=1}^ng_j\vartheta_j$ as usual. Let $\wtd\lambda=\theta_1$. In the following, we shall either show that $\dim\Hom_\gk(\lambda\otimes\mu,\nu)=0$ (which indicates $\dim\mathcal V^l_\gk{\nu\choose\lambda~\mu}=0$), or choose $\wtd\mu,\wtd\nu\in P_+(\pk)$ admissible at level $l$, such that the restrictions of these weights to $\mathfrak h^*$ are $\mu$ and $ \nu$, and that $\wtd\nu-\wtd\mu$ is a weight of $L_\pk(\lambda)$. Then we choose   highest weight vectors $w_\lambda,w_\mu,w_\nu$ of $\ph\curvearrowright L_\pk(\wtd\lambda,l),L_\pk(\wtd\mu,l),L_\pk(\wtd\nu,l)$ respectively, which are also highest weight vectors of $\gh\curvearrowright L_\pk(\wtd\lambda,l),L_\pk(\wtd\mu,l),L_\pk(\wtd\nu,l)$ with highest weights $\lambda,\mu,\nu$ respectively. Choose a non-zero weight-$(\nu-\mu)$ vector $u$ of $\gk\curvearrowright \gk w_\lambda$. Choose a non-zero $\mathfrak Y\in\mathcal V^l_\pk{\wtd\nu\choose\wtd\lambda~\wtd\nu}$. We shall show that  $\bk{\mathfrak Y(u,x) w_\mu| w_\nu}\neq0$. This will imply, by corollary \ref{lb13} and theorem \ref{lb16}-(d), that $\dim\mathcal V^l_\gk{\nu\choose\lambda~\mu}=1$, and that for any $\mathcal Y\in\mathcal V^l_\gk{\nu\choose\lambda~\mu}$ and $u^{(\lambda)}\in L_\gk(\lambda)$, $\mathcal Y(u^{(\lambda)},x)$ satisfies $0$-th order energy-bounds.\\

Subcase 1: $\nu-\mu$ equals $\vartheta_i$ or $-\vartheta_i$ ($i=1,\dots,n$), and $\mu$ is single-valued. Then, since $\nu-\mu$ is a weight of $L_\gk(\vartheta_1)$, $\nu$ must also be single-valued.  Set
 $$\wtd\mu=\sum^n_jf_j\theta_j,\qquad\wtd\nu=\sum^n_jg_j\theta_j.$$ 
Then $\wtd\nu-\wtd\mu$ equals $\theta_i$ or $-\theta_i$, which is a weight of $L_\pk(\theta_1)$. If $u\in(\gk w_\lambda)[\nu-\mu]$ is non-zero, then by lemma \ref{lb19}-(a), $u\in L_\pk(\wtd\lambda)[\wtd\nu-\wtd\mu]$. Therefore, as $\Gamma\Psi$ is injective, $\bk{\mathfrak Y(u,x) w_\mu| w_\nu}\neq0$.

Subcase 2: $\nu-\mu$ equals $\vartheta_i$ or $-\vartheta_i$ ($i=1,\dots,n$), and $\mu$ is double-valued. This subcase can be treated in a similar way as we treat subcase 1, except that we let $$\wtd\mu=\sum^n_jf_j\theta_j+\frac 1 2\theta_{n+1},\qquad\wtd\nu=\sum^n_jg_j\theta_j+\frac 1 2\theta_{n+1}.$$

Subcase 3: $\mu=\nu$ are single-valued, and $f_n=g_n\geq1$. Set $$\wtd\mu=\sum^n_jf_j\theta_j,\qquad\wtd\nu=\sum^n_jg_j\theta_j+\theta_{n+1}~~(=\wtd\mu+\theta_{n+1}).$$
Choose a non-zero $u\in(\gk w_\lambda)[\nu-\mu]$. Then by lemma \ref{lb19}-(b), there exist non-zero  $u_+\in L_\pk(\theta_1)[\theta_{n+1}]$ and $u_-\in L_\pk(\theta_1)[-\theta_{n+1}]$ satisfying $u_0=u_++u_-$. By injectivity of $\Gamma\Psi$, $\bk{\mathfrak Y(u_+,s) w_\mu| w_\nu}\neq0$ for some $s\in\mathbb R$. The vector $\mathfrak Y(u_-,s) w_\mu$ is a weight-$(\wtd\mu-\theta_{n+1})$ vector of $\pk\curvearrowright L_\pk(\wtd\nu,l)$. So we must have $\bk{\mathfrak Y(u_-,s) w_\mu| w_\nu}=0$. Therefore $\bk{\mathfrak Y(u,s) w_\mu| w_\nu}\neq0$.

Subcase 4: $\mu=\nu$ are double-valued. We can prove this by choosing
$$\wtd\mu=\sum^n_jf_j\theta_j-\frac 1 2\theta_{n+1},\qquad\wtd\nu=\sum^n_jg_j\theta_j+\frac 1 2\theta_{n+1}.$$
and using a similar argument as in subcase 3.

Subcase 5: $\mu=\nu$ are single-valued, and $f_n=g_n=0$. We know that $\alpha=\vartheta_n$ is a simple root of $\gk$. Then $n_{\mu,\alpha}:=2(\mu|\vartheta_n)/(\vartheta_n|\vartheta_n)=0$, and hence
\begin{align*}
\dim L_\gk(\lambda)[\nu-\mu+(n_{\mu,\alpha}+1)\alpha]=\dim L_\gk(\vartheta_1)[\vartheta_n]=1.
\end{align*}
By corollary \ref{lb20}, $\dim\Hom_\gk(\lambda\otimes\mu,\nu)=0$. Therefore $\dim\mathcal V^l_\gk{\nu\choose\lambda~\mu}=0$.

Now that we know the fusion rules, we can easily show that  $L_\gk(\varsigma,l)$ is generating. Thus we've proved the main result of this chapter.

\begin{thm}\label{lb44}
	Let $\gk=\mathfrak{so}_{2n+1}$ ($n\geq2$), and choose $\lambda,\mu,\nu\in P_+(\gk)$  admissible at level $l$.
	
	(a) If $\lambda=\vartheta_1$ or $\varsigma$, then  $\Psi:\mathcal V^l_\gk{\nu\choose\lambda~\mu}\rightarrow \Hom_\gk(\lambda\otimes\mu,\nu)$ is bijective. 
	
	(b) If $\lambda=\varsigma$, then  $\Gamma:\Hom_\gk(\lambda\otimes\mu,\nu)\rightarrow L_\gk(\lambda)[\nu-\mu]^*$ is bijective. If $\lambda=\vartheta_1$, then  $\Gamma$ is bijective if and only if $(\mu|\vartheta_n)>0$; otherwise $\dim\Hom_\gk(\lambda\otimes\mu,\nu)=0$.
	
	(c) $L_\gk(\varsigma,l)$ is a generating $V^l_\gk$-module.
	
	(d) If $\lambda=\varsigma$, then any $\mathcal Y_\alpha\in\mathcal V^l_\gk{\nu\choose\lambda~\mu}$ is energy-bounded.
	
	(e) If $\lambda=\vartheta_1$, then for any $\mathcal Y_\alpha\in\mathcal V^l_\gk{\nu\choose\lambda~\mu}$ and $u^{(\lambda)}\in L_\gk(\lambda)$, $\mathcal Y_\alpha(u^{(\lambda)},x)$ satisfies $0$-th order energy bounds. In particular, $\mathcal Y_\alpha$ is energy-bounded.
\end{thm}

\section{Type $C$}

\subsection{$\smp_{2n}\subset \so_{4n}$}

Choose $n=2,3,4,\dots$. Let $I_n$ be the $n\times n$ identity matrix, and let $J_{2n}=\left(\begin{array}{cc}
0&I_n\\
-I_n&0
\end{array}\right)$ The complex symplectic Lie algebra $\smp_{2n}$ is the unitary Lie subalgebra of $\gl_{2n}=\End(\mathbb C^{2n})$ whose elements $X\in\gl_{2n}$ satisfy $J_{2n}X+X^\tr J_{2n}=0$. Write $X=\left(\begin{array}{cc}
A&B\\
C&D
\end{array}\right)$, where $A,B,C,D$ are $n\times n$ matrices. Then $X\in\smp_{2n}$ if and only if $D=-A^\tr$, and $B$ and $C$ are symmetric. 

We first choose a Cartan subalgebra of $\gk:=\smp_{2n}$. As in the last section, we let $e_1,\dots,e_{2n}$ be the standard orthonormal basis of $\mathbb C^{2n}$. For any $i,j=1,\dots,2n$, we define $E_{i,j}\in\End(\mathbb C^{2n})$ to satisfy  $E_{i,j}e_k=\delta_{j,k}e_i$ ($\forall k=1,\dots,2n$). For any $i=1,\dots,n$, we let
\begin{align}
T_i=E_{i,i}-E_{n+i,n+i}.\label{eq55}
\end{align}
A cartan subalgebra can be chosen to be $\hk=\Span_{\mathbb C}\{T_1,\dots,T_n \}$. Choose a basis $\{\vartheta_1,\dots,\vartheta_n \}$ of $\hk^*$ to be the dual basis of $\{T_1,\dots,T_n \}$.

As in the last chapter, given two real numbers $a,b$, we say $b\succsim a$ if $b-a=0,1,2,\dots$.

\subsubsection*{Facts about $\gk=\smp_{2n}$}
\noindent
Simple roots: $\vartheta_1-\vartheta_2,\vartheta_2-\vartheta_3,\dots,\vartheta_{n-1}-\vartheta_n,2\vartheta_n$.\\
Roots: $\pm\vartheta_i\pm\vartheta_j$ ($i\neq j$), $\pm2\vartheta_i$.\\
Positive roots: $\vartheta_i\pm\vartheta_j$ ($i<j$), $2\vartheta_i$.\\ 
If $\alpha$ is a positive root, then the root space
\begin{gather*}
\gk_\alpha=\left\{
\begin{array}{ll}
\mathbb C(E_{i,n+j}+E_{j,n+i}) & (\textrm{if }\alpha=\vartheta_i+\vartheta_j,~~1\leq i,j\leq n)\\
\mathbb C(E_{i,j}-E_{n+j,n+i}) & (\textrm{if }\alpha=\vartheta_i-\vartheta_j,~~1\leq i<j\leq n)
\end{array}
\right.
\end{gather*}
Fundamental weights: $\vartheta_1+\cdots+\vartheta_i\qquad(i=1,2,\dots,n)$.\\
Dominant integral weights: 
\begin{gather}
\lambda=f_1\vartheta_1+\cdots+f_n\vartheta_n\qquad(f_1\succsim\cdots\succsim f_n\succsim 0).
\end{gather}
Highest root: $2\vartheta_1$.\\
The normalized invariant inner product satisfies $(T_i|T_j)=2\delta_{i,j}$, $(\vartheta_i|\vartheta_j)=\frac 12\delta_{i,j}$.\\
Weights of $L_\gk(\vartheta_1)$: $\pm\vartheta_1,\dots,\pm\vartheta_n$.\\
The weight spaces of $L_\gk(\vartheta_1)$ has dimensions at most $1$.

\subsubsection*{Embedding $\smp_{2n}\subset \so_{4n}$}

Note that $L_\gk(\vartheta_1)$ is equivalent to the standard representation $\gk\curvearrowright \mathbb C^{2n}$. We now consider the representation $L_\gk(\vartheta_1)\oplus L_\gk(\overline{\vartheta_1})$. Equivalently, we consider the embedding
\begin{gather}
\smp_{2n}\hookrightarrow \End(\mathbb C^{4n}),\qquad X\mapsto \left(\begin{array}{cc}
X&0\\
0&-X^\tr
\end{array}\right).\label{eq81}
\end{gather}
We now regard $\gk$ as a unitary Lie subalgebra of $\End(\mathbb C^{4n})$ in this way. Then a $4n\times 4n$ matrix $X$ is an element of $\gk$ if and only if
\begin{align*}
X=\left(
\begin{array}{cccc}
A&B&0&0\\
C&-A^\tr&0&0\\
0&0&-A^\tr&-C\\
0&0&-B&A
\end{array}
\right),
\end{align*}
where $A,B,C,D$ are $n\times n$ matrices, and $B,C$ are symmetric. Then \eqref{eq55} should be replaced by
\begin{align}
T_i=E_{i,i}-E_{n+i,n+i}-E_{2n+i,2n+i}+E_{3n+i,3n+i}\qquad(1\leq i\leq n).
\end{align}
For a positive root $\alpha$ of $\gk$, the root space
\begin{gather*}
\gk_\alpha=\left\{
\begin{array}{ll}
\mathbb C(E_{i,n+j}+E_{j,n+i}-E_{3n+i,2n+j}-E_{3n+j,2n+i}) & (\textrm{if }\alpha=\vartheta_i+\vartheta_j,~1\leq i,j\leq n)\\
\mathbb C(E_{i,j}-E_{n+j,n+i}-E_{2n+j,2n+i}+E_{3n+i,3n+j}) & (\textrm{if }\alpha=\vartheta_i-\vartheta_j,~1\leq i<j\leq n)
\end{array}
\right..
\end{gather*}

Set $K_{4n}=\left(\begin{array}{cc}
0&I_{2n}\\
I_{2n}&0
\end{array}\right)$, where $I_{2n}$ is the identity $2n\times 2n$ matrix. Let $\pk$ be the unitary Lie subalgebra of $\End(\mathbb C^{4n})$ whose elements $Y\in\End(\mathbb C^{4n})$ satisfy $K_{4n}Y+Y^\tr K_{4n}=0$. Then $\pk$ is $*$-isomorphic to $\so_{4n}$. Indeed, if we choose $U=\frac 1{\sqrt2}\left(\begin{array}{cc}
I_{2n}&I_{2n}\\
\mathbf{i} I_{2n}& -\mathbf{i} I_{2n}
\end{array}\right)$, then $U^\tr U=K_{4n}$. The Lie algebra isomorphism can be defined as
\begin{gather*}
\pk\rightarrow\so_{4n},\qquad Y\mapsto UYU^{-1}.
\end{gather*}
Since $U$ is unitary, i.e., $U^*=U^{-1}$, this isomorphism preserves the $*$-structures of $\pk$ and $\so_{4n}$, both defined by the adjoint of operators.

A $4n\times 4n$ matrix $Y$ is an element in $\pk$ if and only if  $Y=\left(\begin{array}{cc}
P&Q\\
R&-P^\tr
\end{array}\right)$, where $P,Q,R$ are $2n\times 2n$ matrices, and $Q,R$ are skew-symmetric. Let $e_1,\dots,e_{4n}$ be the standard orthonormal basis of $\mathbb C^{4n}$ as usual, and define the $4n\times 4n$ matrix $E_{i,j}$ ($i\leq i,j\leq 4n$) as usual. For any $1\leq i\leq 2n$ we set 
$$S_i= E_{i,i}-E_{2n+i,2n+i}.$$
Then $\tk=\Span_{\mathbb C} \{S_1,\dots,S_{2n} \}$ is a Cartan subalgebra of $\pk$, and $\hk\subset\tk$. We let $\theta_1,\dots,\theta_{2n}$ be the dual basis of $S_1,\dots,S_{2n}$. Then for any $1\leq i\leq n$, the restrictions of $\theta_i,\theta_{n+i}\in\tk^*$ to $\hk$ are $\vartheta_i,-\vartheta_i$ respectively.

The facts about $\pk$ described in section \ref{lb21} also hold for $\theta_1,\dots,\theta_{2n}$. However, if we choose simple roots of $\pk$ as in section \ref{lb21}, then $\gk_+$ will no longer be a subset of $\pk_+$. So let us choose another set of simple roots instead, and consider relevant facts below. This new set of simples roots is chosen with respect to the order $1,2,\dots,n,2n,2n-1,2n-2,\dots,n+2,n+1$, rather than the standard  $1,2,\dots,2n$ considered in section \ref{lb21}.\\

\noindent
Simple roots: $\theta_1-\theta_2,\theta_2-\theta_3,\dots,\theta_{n-1}-\theta_n,\theta_{n}-\theta_{2n},\theta_{2n}-\theta_{2n-1},\theta_{2n-1}-\theta_{2n-2},\dots,\theta_{n+2}-\theta_{n+1},\theta_{n+2}+\theta_{n+1}$.\\
Roots: $\pm\theta_i\pm\theta_j$ ($i\neq j$).\\
Positive roots: $\theta_i+\theta_j$ for all $1\leq i\neq j\leq 2n$; $\theta_i-\theta_j,-\theta_{n+i}+\theta_{n+j}$ for all $1\leq i<j\leq n$;  $\theta_i-\theta_{n+j}$ for all $1\leq i,j\leq n$.\\
If $\alpha$ is a positive root of $\pk$, then the root space
\begin{gather}
\gk_\alpha=\left\{
\begin{array}{ll}
\mathbb C(E_{i,2n+j}-E_{j,2n+i}) & (\textrm{if }\alpha=\theta_i+\theta_j,~~1\leq i\neq j\leq 2n)\\
\mathbb C(E_{i,j}-E_{2n+j,2n+i}) & (\textrm{if }\alpha=\theta_i-\theta_j,~~1\leq i<j\leq n)\\
\mathbb C(E_{n+j,n+i}-E_{3n+i,3n+j}) & (\textrm{if }\alpha=-\theta_{n+i}+\theta_{n+j},~~1\leq i<j\leq n)\\
\mathbb C(E_{i,n+j}-E_{3n+j,2n+i}) & (\textrm{if }\alpha=\theta_i-\theta_{n+j},~~1\leq i,j\leq n)
\end{array}
\right..\label{eq61}
\end{gather}
Non-zero dominant integral weights admissible at level $1$:
\begin{gather*}
\theta_1,\qquad
\varsigma_\pm=\frac 1 2\bigg(\pm\theta_{n+1}+\sum_{i\neq n+1}\theta_i\bigg).
\end{gather*}
Highest root: $\theta_1+\theta_2$.\\
The normalized invariant inner product satisfies $(S_i|S_j)=(\theta_i|\theta_j)=\delta_{i,j}$.\\
Weights of $L_\pk(\theta_1)$: $\pm\theta_1,\dots,\pm\theta_{2n}$.\\
Weights of $L_\pk(\varsigma_\pm)$: 
\begin{gather*}
\frac {(-1)^{\sigma_1}\theta_1+\cdots+(-1)^{\sigma_{2n}}\theta_{2n}}2\qquad (\sigma_1,\dots,\sigma_{2n}\in\mathbb Z),\\
\sigma_1+\cdots+\sigma_{2n}\in \left\{ \begin{array}{ll}
2\mathbb Z & \textrm{for $\varsigma_+$}\\
2\mathbb Z+1 & \textrm{for $\varsigma_-$.}\\
\end{array} \right.
\end{gather*}
The weight spaces of $L_\pk(\theta_1),L_\pk(\varsigma_+)$, and $L_\pk(\varsigma_-)$ have dimensions at most $1$.\\

Using this information, one can check that $\gk\subset\pk$ has Dynkin index $1$, and that
\begin{align}
\gk_+\subset&\pk_+\cap \bigg(\bigoplus^\perp_{1\leq i\neq j\leq 2n}\pk_{\theta_i+\theta_j}\bigg)^\perp\nonumber\\
=&\bigg(\bigoplus^\perp_{1\leq i<j\leq 2n}\pk_{\theta_i-\theta_j}\bigg)\oplus^\perp \bigg(\bigoplus^\perp_{1\leq i<j\leq 2n}\pk_{-\theta_{n+i}+\theta_{n+j}}\bigg)\oplus^\perp \bigg(\bigoplus^\perp_{1\leq i,j\leq 2n}\pk_{\theta_i-\theta_{n+j}}\bigg).\label{eq56}
\end{align}

\subsection{Main result}

In this section, we show, for any $\mu,\nu\in P_+(\gk)$ admissible at level $l$, that the fusion rule $N^\nu_{\vartheta_1\mu}$ equals $1$ when $\nu-\mu$ is a weight of $L_\gk(\vartheta_1)$, and equals $0$ otherwise. This will imply that $L_\gk(\vartheta_1,l)$ is a generating $V^l_\gk$-module. We also prove the $0$-th order energy bounds condition for any $\mathcal Y\in\mathcal V^l_\gk{\nu\choose\vartheta_1~\mu}$. We first do this for level $1$.
\subsubsection*{Level $1$}
Dominant integral weights of $\gk$ admissible at level $1$ are $0,\vartheta_1,\vartheta_1+\vartheta_2,\dots,\vartheta_1+\vartheta_2+\cdots+\vartheta_n$. Now assume $\mu,\nu$ are among these weights. Since $\Gamma\Psi:\mathcal V^1_\gk{\nu\choose\vartheta_1~\mu}\rightarrow L_\gk(\vartheta_1)[\nu-\mu]^*$ is injective, we can assume that $\nu-\mu$ is a weight of $L_\gk(\vartheta_1)$. So there exists $k=1,\dots,n$, such that $\nu-\mu=\pm\vartheta_k$. 

Let us first assume that $\nu-\mu=\vartheta_k$. Then we must have
\begin{gather*}
\mu=\vartheta_1+\dots+\vartheta_{k-1},\qquad \nu=\vartheta_1+\cdots+\vartheta_k
\end{gather*}
(set $\mu=0$ when $k=1$). 
Set
\begin{gather*}
\varsigma_\mu=\varsigma_+,\qquad\varsigma_\nu=\varsigma_-\qquad(\textrm{if $k$ is odd}),\\
\varsigma_\mu=\varsigma_-,\qquad\varsigma_\nu=\varsigma_+\qquad(\textrm{if $k$ is even}).
\end{gather*}
We now choose highest weight vectors of $\gk$ in $L_\pk(\theta_1,1),L_\pk(\varsigma_\mu,1),L_\pk(\varsigma_\nu,1)$ respectively. Identify $\pk\curvearrowright L_\pk(\theta_1)$ with the standard representation $\pk\curvearrowright \mathbb C^{4n}$, and choose $w_{\vartheta_1}=e_{3n+1}$. One can easily check that  $w_{\vartheta_1}$ is a $\vartheta_1$-highest weight vector of the level $1$ representation $\gh\curvearrowright L_\pk(\theta_1,1)$.

Let $w_\mu$ (resp. $w_\nu$) be a non-zero weight vector $L_\pk(\varsigma_\mu)$ (resp. $L_\pk(\varsigma_\nu)$) with weight
\begin{gather*}
\bigg(\sum_{i=1}^n\theta_i-\sum_{i=n+1}^{n+k-1}\theta_i+\sum_{i=n+k}^{2n}\theta_i\bigg)/2\\
(\textrm{resp. }\bigg(\sum_{i=1}^n\theta_i-\sum_{i=n+1}^{n+k}\theta_i+\sum_{i=n+k+1}^{2n}\theta_i\bigg)/2).
\end{gather*}
 Now, using \eqref{eq56}, we can check that $w_\mu$ and $w_\nu$ are highest weight vectors of $\gh\curvearrowright L_\pk(\varsigma_\mu,1),L_\pk(\varsigma_\nu,1)$ with highest weights $\mu$ and $\nu$ respectively. Suppose that we can prove lemma \ref{lb22}, then, since  $e_{3n+k}\in L_\pk(\theta_1)[-\theta_{n+k}]$, by the injectivity of $\Psi$, for any non-zero $\mathfrak Y\in\mathcal V^1_\pk{\varsigma_\nu\choose\theta_1~\varsigma_\mu}$ (which exists due to theorem \ref{lb16}-(a)), we have $\bk{\mathfrak Y(e_{3n+k},x)w_\mu|w_\nu }\neq0$. Since $e_{3n+k}\in 0\oplus\mathbb C^{2n}=\gk e_{3n+1}=\gk w_{\vartheta_1}$, we actually have $\bk{\mathfrak Y(\gk w_{\vartheta_1},x)w_\mu|w_\nu }\neq0$. Therefore, by corollary \ref{lb13} and theorem \ref{lb16}-(d), $\dim\mathcal V^1_\gk{\nu\choose\vartheta_1~\mu}=1$, and for any $\mathcal Y\in\mathcal V^1_\gk{\nu\choose\vartheta_1~\mu}$ and $u^{(\vartheta_1)}\in L_\gk(\vartheta_1)$, $\mathcal Y(u^{(\vartheta_1)},x)$ satisfies $0$-th order energy bounds.

\begin{lm}\label{lb22}
For any non-zero $T\in\Hom_\pk(\theta_1\otimes\varsigma_\mu,\varsigma_\nu)$, inequality  $\bk{T(e_{3n+k}\otimes w_\mu)|w_\nu }\neq 0$ holds.
\end{lm}

\begin{proof}
We write $w_{k-1}=w_\mu,w_k=w_\nu$, and prove
\begin{align}
\bk{T(e_{3n+k}\otimes w_{k-1})|w_k }\neq 0\label{eq57}
\end{align}
by induction on $k$. For any positive root $\alpha$ of $\pk$, we choose  $E_\alpha,F_\alpha,H_\alpha$ as in \eqref{eq11}. Now for any $i=1,\dots,n-1$ we choose $\alpha_i=\theta_{n+i}+\theta_{n+i+1}$. Then, by \eqref{eq61},
\begin{gather}
E_{\alpha_i}\in\mathbb C(E_{n+i,3n+i+1}-E_{n+i+1,3n+i}),\label{eq62}\\
F_{\alpha_i}\in\mathbb C(E_{3n+i+1,n+i}-E_{3n+i,n+i+1}).\label{eq63}
\end{gather}

Given a general $k$, we first show that inequality
\begin{align}
\bk{T(e_{3n+k-2}\otimes w_{k-3})|w_{k-2}}\neq 0\label{eq58}
\end{align}
implies \eqref{eq57}. Assume \eqref{eq58} is true. Then
\begin{align}
\bk{T(e_{3n+k}\otimes F_{\alpha_{k-2}}w_{k-3})|F_{\alpha_{k-1}}w_{k-2}}=&\bk{T(E_{\alpha_{k-1}}e_{3n+k}\otimes F_{\alpha_{k-2}}w_{k-3})|w_{k-2}}\nonumber\\
&+\bk{T(e_{3n+k}\otimes E_{\alpha_{k-1}} F_{\alpha_{k-2}}w_{k-3})|w_{k-2}}.\label{eq59}
\end{align}
However, the weight of $E_{\alpha_{k-1}} F_{\alpha_{k-2}}w_{k-3}$ is $\cdots+\frac 3 2\theta_{n+k}+\cdots$, which is not a weight of $L_\pk(\varsigma_\pm)$. Therefore, the second summand in equation \eqref{eq59} vanishes. Thus \eqref{eq59} equals
\begin{align}
\bk{T(E_{\alpha_{k-1}}e_{3n+k}\otimes F_{\alpha_{k-2}}w_{k-3})|w_{k-2}}=&-\bk{T(F_{\alpha_{k-2}}E_{\alpha_{k-1}}e_{3n+k}\otimes w_{k-3})|w_{k-2}}\nonumber\\
&+\bk{T(E_{\alpha_{k-1}}e_{3n+k}\otimes w_{k-3})|E_{\alpha_{k-2}}w_{k-2}}.\label{eq60}
\end{align}
Again, the weight of $E_{\alpha_{k-2}}w_{k-2}$ is $\cdots+\frac 3 2\theta_{n+k-1}+\cdots$. So the second summand of \eqref{eq60} equals $0$. By \eqref{eq62} and \eqref{eq63}, there exists $c\in\mathbb C\setminus\{0 \}$ such that $F_{\alpha_{k-2}}E_{\alpha_{k-1}}e_{3n+k}=c\cdot e_{3n+k-2}$. Therefore \eqref{eq59} equals the left hand side of \eqref{eq58} multiplied by $-c$. So \eqref{eq59} is non-zero. Since $F_{\alpha_{k-2}}w_{k-3}$ and $w_{k-1}$ have the same weight, they must be proportional. Similarly, $F_{\alpha_{k-1}}w_{k-2}$ and $w_k$ are proportional. Thus inequality \eqref{eq57} is proved.

Now, from what we have proved above, in order to finish the proof of this lemma, it suffices to prove \eqref{eq57} when $k=1,2$. If $k=1$, then $w_{k-1}$ and $w_k$ are highest weight vectors of $L_\pk(\varsigma_\mu)=L_\pk(\varsigma_+)$ and $L_\pk(\varsigma_\nu)=L_\pk(\varsigma_-)$ respectively. Then \eqref{eq57} follows from the injectivity of $\Gamma:\Hom_\pk(\theta_1\otimes\varsigma_+,\varsigma_-)\rightarrow L_\pk(\theta_1)[-\theta_{n+1}]^*$. If $k=2$, then $w_1$ is highest weight vector but $w_2$ is not. Choose a highest weight vector $w_0$ of $L_\pk(\varsigma_\nu)=L_\pk(\varsigma_+)$, Then, since $E_{\alpha_1}w_1=0$, 
\begin{align}
\bk{T(e_{3n+2}\otimes w_1)|F_{\alpha_1}w_0 }=\bk{T(E_{\alpha_1}e_{3n+2}\otimes w_1)|w_0 }.\label{eq64}
\end{align}
By \eqref{eq62}, the right hand side of  \eqref{eq64} equals $\bk{T(e_{n+1}\otimes w_1)|w_0 }$ multiplied by a non-zero complex number. Since the difference of the weights of $w_0$ and $w_1$ is $\varsigma_+-\varsigma_-=\theta_{n+1}$, which is the same as the weight of $e_{n+1}$, from the injectivity of $\Gamma:\Hom_\pk(\theta_1\otimes\varsigma_-,\varsigma_+)\rightarrow L_\pk(\theta_1)[\theta_{n+1}]^*$, we find that $\bk{T(e_{n+1}\otimes w_1)|w_0 }$ is non-zero. So \eqref{eq64} is also non-zero. Since $F_{\alpha_1}w_0$ and $w_2$ have the same weight, they are proportional. Thus \eqref{eq57} is proved when $k=2$.
\end{proof}

The case $\nu-\mu=-\vartheta_k$ can be treated in a similar way. Alternatively, we first show $\dim\Hom_\gk(\vartheta_1\otimes\vartheta_1,0)=1$ using corollary \ref{lb20}, which implies $\overline{\vartheta_1}=\vartheta_1$ by corollary \ref{lb23}. Then we use the result of case $\nu-\mu=-\vartheta_k$ to prove this case by applying theorem \ref{lb14}.

\subsubsection*{Level $l$}

We assume $\nu-\mu=\vartheta_k$ ($k=1,\dots,n$). The case $\nu-\mu=-\vartheta_k$ can be treated either in a similar way, or by applying theorem \ref{lb14}. Write $\mu=\sum_{i=1}^n f_i\vartheta_i,\nu=\sum_{i=1}^n g_i\vartheta_i$. Then $g_k=f_k+1$, and $f_i=g_i$ when $i\neq k$. Let $a_n=f_n,b_n=g_n$, and let $a_i=f_i-f_{i+1},b_i=g_i-g_{i+1}$ for any $i=1,\dots,n-1$.  Consider the fundamental weights $\lambda_1=\vartheta_1,\lambda_2=\vartheta_1+\vartheta_2,\dots,\lambda_n=\vartheta_1+\vartheta_2+\cdots+\vartheta_n$. Then $\mu=\sum^n_{i=1}a_i\lambda_i,\nu=\sum_{i=1}^n b_i\lambda_i$.

If $2\leq k\leq n-1$, then $b_{k-1}=a_{k-1}-1,b_k=a_k+1$, and $b_i=a_i$ for any $i\neq k-1,k$. Let
\begin{gather*}
\rho=(a_{k-1}-1)\lambda_{k-1}+\sum_{i\neq k-1}a_i\lambda_i,\qquad \mu_1=\lambda_{k-1},\qquad \nu_1=\lambda_k.
\end{gather*}
Then $\rho,\mu_1,\nu_1\in P_+(\gk)$,  $\mu_1,\nu_1$ are admissible at level $1$, $\rho$ is admissible at level $l-1$, and $\mu=\rho+\mu_1,\nu=\rho+\nu_1$. Thus, by the result of level $1$ and lemma \ref{lb24}-(a), $N^\nu_{\vartheta_1\mu}=1$, and for any $\mathcal Y\in\mathcal V^l_\gk{\nu\choose\vartheta_1~\mu}$ and $u^{(\vartheta_1)}\in L_\gk(\vartheta_1)$, $\mathcal Y(u^{(\vartheta_1)},x)$ satisfies $0$-th order energy bounds. 

Now assume that $k=1$. Then $b_1=a_1+1$, and $b_i=a_i$ for any $i\neq k-1$. By setting $\rho=\mu,\mu_1=0,\nu_1=\lambda_1=\vartheta_1$, we again prove the desired result using lemma \ref{lb24}-(a). We summarize the result of this chapter as follows:

\begin{thm}\label{lb43}
Let $\gk=\smp_{2n}$ ($n\geq2$), and choose $\mu,\nu\in P_+(\gk)$ admissible at level $l$.

(a) The maps $\Psi:\mathcal V^l_\gk{\nu\choose\vartheta_1~\mu}\rightarrow \Hom_\gk(\vartheta_1\otimes\mu,\nu)$ and $\Gamma:\Hom_\gk(\vartheta_1\otimes\mu,\nu)\rightarrow L_\gk(\vartheta_1)[\nu-\mu]^*$ are bijective.

(b) $L_\gk(\vartheta_1)$ is a generating $V^l_\gk$-module.

(c) For any $\mathcal Y\in\mathcal V^l_\gk{\nu\choose\vartheta_1~\mu}$ and $u^{(\vartheta_1)}\in L_\gk(\vartheta_1)$, $\mathcal Y(u^{(\vartheta_1)},x)$ satisfies $0$-th order energy bounds. In particular, $\mathcal Y$ is energy-bounded.
\end{thm}

\begin{rem}
Our arguments in this chapter also work pretty well for affine $\mathfrak{sl}_n$ VOAs. Indeed, one has a similar embedding $\mathfrak{sl}_n\subset\mathfrak{so}_{2n}$ with Dynkin index $1$ defined by \eqref{eq81}. Then one can show, similarly, that any  irreducible intertwining operator of $V^1_{\mathfrak {sl}_n}$ whose charge space restricts to the standard representation $\mathbb C^n=L_{\mathfrak{sl}_n}(\vartheta_1)$ of $\mathfrak{sl}_n$ is a compression of an irreducible intertwining operator of $V^1_{\mathfrak{so}_{2n}}$ whose charge space is $L_{\mathfrak{so}_{2n}}(\theta_1,1)$, and whose source and target spaces are $L_{\mathfrak{so}_{2n}}(\varsigma_\pm,1)$. The higher level case can similarly be reduced to the level $1$ one using lemma \ref{lb24}-(a).

We remark that the method used in \cite{Was98} for type $A$ also works for our type $C$. This method also uses essentially the index $1$ embedding $\mathfrak{sl}_n\subset\mathfrak{so}_{2n}$, but it uses free fermions instead of the intertwining operators of affine $V^1_{\mathfrak{so}_{2n}}$. In fact, if we let $F_{2n}$ be the $2n$-dimensional free Fermion unitary vertex super algebra, whose even part is a unitary VOA, then its weight $1$ subspace $F_{2n}(1)$, as a unitary Lie algebra, is equivalent to $\mathfrak{so}_{2n}$, and the level of $\mathfrak{so}_{2n}$ in (the even part of) $F_{2n}$ is $1$. Therefore, the level of $\mathfrak{sl}_n$ in $F_{2n}$ is also one. One can then realize the intertwining operators of $V^1_{\mathfrak{sl}_n}$ whose charge spaces are $L_{\mathfrak{sl}_n}(\vartheta_1,1)$ as compressions of the vertex operator of $F_{2n}$, and thus proves the results for $V^1_{\mathfrak{sl}_n}$. This method can easily be adapted to $\mathfrak{sp}_{2n}\subset\mathfrak{so}_{4n}\subset F_{4n}$ to prove the results for $V^1_{\mathfrak{sp}_{2n}}$. One thus sees that type $A$ and type $C$ are similar in many aspects. However, the lattice VOA method works only for type $A$, but (clearly) not for type $C$.
\end{rem}

For the reader's convenience, we list the result for type $A$ as follows.
\begin{thm}[cf.\cite{Was98}]\label{lb45}
	Let $\gk=\mathfrak{sl}_n$ ($n\geq2$), and choose $\mu,\nu\in P_+(\gk)$ admissible at level $l$. Let $\vartheta_1$ be the integral dominant weight of $\gk$ so that $L_\gk(\vartheta_1)$ is the standard representation $\mathbb C^n$ of $\mathfrak{sl}_n$. Then (a), (b), (c) of theorem \ref{lb43} hold verbatim.
\end{thm}

\section{Type $G_2$}

\subsection{$\gk_2\subset\ek_8$}

\subsubsection*{Facts about $\gk_2$}

The root system of $\gk_2$ is shown in figure \ref{fig1}. Among these roots, we choose $\alpha_1,\alpha_2$ to be a set of simple roots. Then $\vartheta_1$ and $\vartheta_2$ are the fundamental weights, and $\vartheta_2$ is the highest root of $\gk_2$. Let  $(\cdot|\cdot)$ be the normalized invariant inner product of $\gk_2$. Then clearly $(\vartheta_1|\vartheta_2)=1$, and $(\vartheta_2|\vartheta_2)=2$.

\begin{figure}[h]
	\centering
	\includegraphics[width=0.3\linewidth]{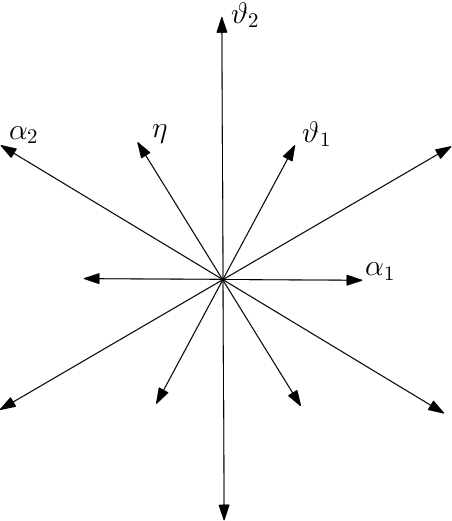}
	\caption[Figure]{. $G_2$ root system}
	\label{fig1}
\end{figure}

It is well known that $L_{\gk_2}(\vartheta_1)$ has dimension $7$,  its weights are $\pm\vartheta_1,\pm\eta,\pm\alpha_1,0$, and for any weight the weight space has dimension $1$. (Here $\eta=\alpha_2-\alpha_1$. See figure \ref{fig1}.) We now determine, for any integral dominant weights $\mu,\nu$ of $\gk_2$, the tensor product rule $\dim\Hom_{\gk_2}(\vartheta_1\otimes\mu,\nu)$. If this number is non-zero, then $\nu-\mu$ must be one of the 7 weights of $L_{\gk_2}(\vartheta_1)$. Write $\nu=n_1\vartheta_1+n_2\vartheta_2$ with $n_1,n_2\in\mathbb Z_{\geq0}$. Then, using corollary \ref{lb20}, one can show that $\dim\Hom_{\gk_2}(\vartheta_1\otimes\mu,\nu)=1$ if  $\nu-\mu\in\{\pm\vartheta_1,\pm\eta,\pm\alpha_1\}$, or if $\mu=\nu$ and $n_1>0$; otherwise $\dim\Hom_{\gk_2}(\vartheta_1\otimes\mu,\nu)=0$. This result can be summarized by the tensor graph \ref{fig2}.
\begin{figure}[h]
	\centering
	\includegraphics[width=0.25\linewidth]{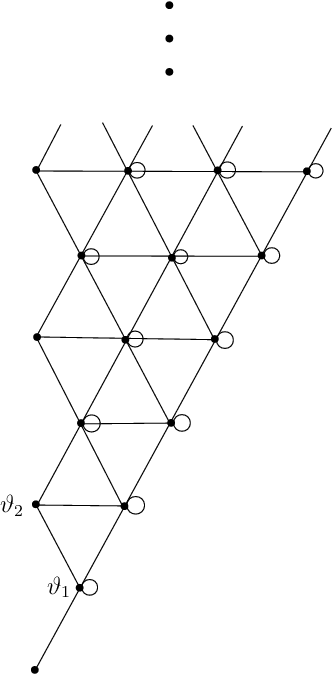}
	\caption[Figure]{. $L_\gk(\vartheta_1)$ tensor graph}
	\label{fig2}
\end{figure}

Graph \ref{fig2} is read as follows. Dots represent dominant integral weights of $\gk_2$. If two dots represent $\mu,\nu\in P_+(\gk_2)$, then a line segment connecting these two points means $\dim\Hom_{\gk_2}(\vartheta_1\otimes\mu,\nu)=\dim\Hom_{\gk_2}(\vartheta_1\otimes\nu,\mu)=1$. The equality of these two dimensions stems from the fact that $\vartheta_1=\overline{\vartheta_1}$, which is a consequence of $\dim\Hom_{\gk_2}(\vartheta_1\otimes\vartheta_1,0)=1$ and proposition \ref{lb23}. When two dots coincide, a line segment becomes a loop.

\subsubsection*{Embedding $\gk_2\oplus\fk_4\subset\ek_8$}

In \cite{Dyn52}, E.B.Dynkin gave a systematic treatment of maximal Lie subalgebras of a simple Lie algebra. We will use the example  $\gk_2\oplus\fk_4\subset\ek_8$ to prove the energy bounds conditions of intertwining operators of $\gk_2$. Let us first review the construction of the subalgebra $\gk_2\oplus\fk_4$ given in \cite{Dyn52} table 35.
\begin{figure}[h]
	\centering
	\includegraphics[width=0.5\linewidth]{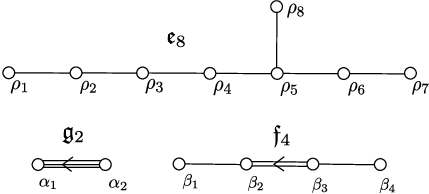}
	\caption[Figure]{.}
	\label{fig3}
\end{figure}

The simple roots $\rho_1,\dots,\rho_8$ of $\ek_8$ and $\beta_1,\dots,\beta_4$ of $\fk_4$ are chosen as in figure \ref{fig3}. We choose  raising operators $P_1\in(\ek_8)_{\rho_1},\dots,P_8\in(\ek_8)_{\rho_8}$ of unit length. Now, for any $i_1,i_2,\dots,i_n\in\{1,\dots,8 \}$, we write $P_{i_1i_2\cdots i_n}=[P_{i_1},[P_{i_2},\cdots,[P_{i_{n-1}},P_{i_n}]]]$. Set
\begin{gather}
A_1=P_{234568543}-P_{234567854}+P_{234567856},\qquad A_2=P_1,\label{eq66}\\
B_1=P_3+P_7,\qquad B_2=P_4+P_6,\qquad B_3=P_5,\qquad B_4=P_8.\label{eq67}
\end{gather}

\begin{thm}[\cite{Dyn52} theorem 13.1 and table 35]\label{lb31}
The unitary Lie subalgebra of $\ek_8$  generated by $A_1,A_2,B_1,\dots,B_4$ is equivalent to $\gk_2\oplus\fk_4$. It is a maximal unitary Lie subalgebra of $\ek_8$.  Moreover, $A_1\in(\gk_2)_{\alpha_1},A_2\in(\gk_2)_{\alpha_2},B_1\in(\fk_4)_{\beta_1},\dots,B_4\in(\fk_4)_{\beta_4}$, and these elements are non-zero.
\end{thm}
Note that a \emph{unitary} Lie subalgebra generated by some elements is the Lie subalgebra generated by these elements together with their adjoints.

\subsubsection*{Irreducible decomposition of $\gk_2\curvearrowright\ek_8$}
Since $\gk_2$ is a unitary Lie subalgebra of $\ek_8$, the adjoint representation $\gk_2\curvearrowright\ek_8$ (which is the restriction of the adjoint representation $\ek_8\curvearrowright\ek_8$ to $\gk_2$) is a unitary representation of $\gk_2$. We now study the irreducible decomposition of this representation.

\begin{lm}\label{lb26}
The embedding $\gk_2\subset\ek_8$ has Dynkin index $1$
\end{lm}
\begin{proof}
Let $(\cdot|\cdot)$ be the normalized invariant inner product of $\ek_8$. Recall that $(P_1|P_1)=1$. Let $H_1=[P_1,P_1^*]$. As $(\rho_1|\rho_1)=2$, by \eqref{eq11}, $P_1,P_1^*,H_1$ satisfy the unitary $\mathfrak{sl}_2$-relation. Since $A_1=P_1$ and $H_1=[A_1,A_1^*]$, by \eqref{eq11} and \eqref{eq65}, there exists $c\in\mathbb C$ of unit length, such that $cA_1= \sqrt 2 A_1/(\lVert A_1 \lVert \lVert \alpha_2 \lVert)=\sqrt2 A_1/\lVert \alpha_2 \lVert$. Compare the norms of both sides, we get $\lVert \alpha_2 \lVert=\sqrt2$, i.e., the longest roots of $\gk_2$ has norm $\sqrt2$ under the normalized invariant inner product of $\ek_8$. Hence $(\cdot|\cdot)$ restricts to the normalized invariant inner product of $\gk_2$. 
\end{proof}

\begin{lm}\label{lb27}
Let $\pk$ be a  unitary simple Lie algebra, and $\gk$ a unitary simple Lie subalgebra of $\pk$. Assume that $\gk\subset\pk$ has Dynkin index $k$. Consider the adjoint action of $\gk$ on the orthogonal complement $\gk^\perp=\pk\ominus^\perp\gk$. If $v_\lambda\in \gk^\perp$ is a highest weight vector of $\gk\curvearrowright\gk^\perp$ with highest weight $\lambda$, then $\lambda$ is admissible at level $k$. 
\end{lm}	

\begin{proof}
The PER $\ph\curvearrowright L_\pk(0,1)$ restricts to a level $k$ PER $\gh\curvearrowright L_\pk(0,1)$. Let $\Omega$ be the vacuum vector of $L_\pk(0,1)$. If we can show that $v_\lambda(-1)\Omega$ is a weight $\lambda$ highest weight vector of this PER, then $\gh\curvearrowright \gh v_\lambda(-1)\Omega$ is unitarily equivalent to $\gh\curvearrowright L_\gk(\lambda,k)$, and hence $\lambda$ is admissible at level $k$.

Choose any $X\in\gk$ and $n=1,2,\dots$. If $n>1$, then the conformal weight of $X(n)v_\lambda(-1)\Omega$ is negative, which immediately shows $X(n)v_\lambda(-1)\Omega=0$. In the case that $n=1$, we compute
\begin{align*}
X(1)v_\lambda(-1)\Omega=[X(1),v_\lambda(-1)]\Omega=[X,v_\lambda](0)\Omega+(X|v_\lambda^*)\Omega=(v_\lambda|X^*)\Omega.
\end{align*}
Since $v_\lambda\in\gk^\perp$, the right hand side of the above equation vanishes. Therefore $X(1)v_\lambda(-1)\Omega=0$.

Now let us assume $X\in\gk_+$. Then $[X,v_\lambda]=\Ad(X)v_\lambda=0$. Therefore
\begin{align*}
X(0)v_\lambda(-1)\Omega=[X(0),v_\lambda(-1)]\Omega=[X,v_\lambda](-1)\Omega=0.
\end{align*}
Finally, if $H\in\hk$ where $\hk$ is the Cartan subalgebra of $\gk$, then $[H,v_\lambda]=\bk{\lambda,H}v_\lambda$. Thus
\begin{align*}
H(0)v_\lambda(-1)\Omega=[H,v_\lambda](-1)\Omega=\bk{\lambda,H}v_\lambda(-1)\Omega.
\end{align*}
Our proof is finished.
\end{proof}

Now we are ready to study the irreducible decomposition of $\gk_2\curvearrowright\ek_8$. In the following, we let $(\cdot|\cdot)$ be the normalized invariant inner product of $\ek_8$, which restricts to the normalized invariant inner product of $\gk_2$. Consider the orthogonal decomposition of unitary $\gk_2$-module:
\begin{align}
\ek=\gk_2\oplus\fk_4\oplus(\gk_2\oplus\fk_4)^\perp.
\end{align}
The adjoint action $\gk_2\curvearrowright\gk_2$ is unitarily equivalent to $\gk_2\curvearrowright L_{\gk_2}(\vartheta_2)$. Since $[\gk_2,\fk_4]=0$, the action $\gk_2\curvearrowright\fk_4$ is trivial, i.e., it is an orthogonal direct sum of $\dim\fk_4=52$ pieces of irreducible trivial $\gk_2$-modules. 

It remains to study $\gk_2\curvearrowright(\gk_2\oplus\fk_4)^\perp$. We first show that this representation has no trivial irreducible submodules. Let $\kk=\{X\in\ek_8:[\gk_2,X]=0 \}$. Then
\begin{gather*}
[\gk_2,[\kk,\kk]]=[[\gk_2,\kk],\kk]+[\kk,[\gk_2,\kk]]=0,\\
[\gk_2,\kk^*]=[\kk,\gk_2^*]^*=[\kk,\gk_2]^*=0.
\end{gather*} 
which shows that $[\kk,\kk]\subset\kk$ and $\kk^*\subset\kk$, i.e., $\kk$ is a unitary Lie subalgebra of $\ek_8$. Moreover, since $\gk_2$ is simple,
\begin{align*}
(\gk_2|\kk)=([\gk_2,\gk_2]|\kk)=(\gk_2|[\gk_2^*,\kk])=(\gk_2|[\gk_2,\kk])=0.
\end{align*}
Thus $\kk$ is orthogonal to $\gk_2$. Hence $\ek_8$ contains a unitary Lie subalgebra $\gk_2\oplus\kk$. Obviously $\fk_4\subset\kk$. Since $\gk_2\oplus\fk_4$ is a maximal unitary Lie subalgebra of $\ek_8$, $\fk_4$ must equal $\kk$. Therefore, for any $v\in(\gk_2\oplus\fk_4)^\perp$, if $\gk_2$ acts trivially on $v$ (which is equivalent to saying that $[\gk_2,v]=0$), then $v\in\kk=\fk_4$, which forces $v$ to be $0$. We conclude that $\gk_2\curvearrowright(\gk_2\oplus\fk_4)^\perp$ has no irreducible submodule equivalents to $L_{\gk_2}(0)=\mathbb C$.

Now, by lemmas \ref{lb26} and \ref{lb27}, any irreducible submodule of $\gk_2\curvearrowright (\gk_2\oplus\fk_4)^\perp$ must be a non-trivial irreducible $\gk_2$-module whose highest weight is admissible at level $1$. The only choice of such weight is $\vartheta_1$. Therefore $(\gk_2\oplus\fk_4)^\perp$ is equivalent to  a direct sum of $L_{\gk_2}(\vartheta_1)$. Using the dimensions of $\gk_2,\fk_4,\ek_8$, and $L_{\gk_2}(\vartheta_1)$, one checks that the multiplicity of $L_{\gk_2}(\vartheta_1)$ in $(\gk_2\oplus\fk_4)^\perp$ is $(248-14-52)/7=26$. We conclude the following:
\begin{thm}\label{lb28}
The adjoint representation $\gk_2\curvearrowright\ek_8$ has the irreducible decomposition
\begin{align}
\ek_8\simeq 52\mathbb C\oplus 26L_{\gk_2}(\vartheta_1)\oplus L_{\gk_2}(\vartheta_2).
\end{align}
Moreover, we have the following identifications:
\begin{align}
52\mathbb C=\fk_4,\qquad 26L_{\gk_2}(\vartheta_1)=(\gk_2\oplus\fk_4)^\perp,\qquad L_{\gk_2}(\vartheta_2)=\gk_2.
\end{align}
\end{thm}	

\subsection{Main result}
In this section we set $\gk=\gk_2$ and $\pk=\ek_8$.

We first prove the energy bounds condition for the creation operator of $L_\gk(\vartheta_1,1)$. Choose an arbitrary $\mathcal Y\in\mathcal V^1_\gk{\vartheta_1\choose \vartheta_1~0}$. We show that for any $u^{(\vartheta_1)}\in L_\gk(\vartheta_1)\subset L_\gk(\vartheta_1,1)$, $\mathcal Y(u^{(\vartheta_1)},x)$ satisfies $1$-st order energy bounds. Let $Y_\pk$ be the vertex operator of $V^1_\pk$ representing the vacuum module. Then $Y_\pk\in\mathcal V^1_\pk{0\choose 0~0}$. Let $\Omega$ be the vacuum vector of $V^1_\pk$. Then $\Omega$ is a $0$-highest weight vector of $\gh\curvearrowright V^1_\pk$. Now, by theorem \ref{lb28}, we can choose a $\vartheta_1$-highest weight vector $v_{\vartheta_1}$ of $\gk\curvearrowright (\gk_2\oplus\fk_4)^\perp$. Using the argument in the proof of lemma \ref{lb27}, one can show that $v_{\vartheta_1}(-1)\Omega$ is a $\vartheta_1$-highest weight vector of $\gh\curvearrowright V^1_\pk$.  Now the creation property $\lim_{x\rightarrow 0}Y_\pk(v_{\vartheta_1}(-1)\Omega,x)\Omega=v_{\vartheta_1}(-1)\Omega$ gives $\bk{Y_\pk(v_{\vartheta_1}(-1)\Omega,x)\Omega|v_{\vartheta_1}(-1)\Omega}\neq0$. By theorem \ref{lb18}, for any $u\in \gk v_{\vartheta_1}(-1)\Omega\subset V^1_\pk(1)$, $Y_\pk(u,x)$ satisfies $1$-st order energy bounds. Thus we can prove the desired results using corollary \ref{lb13}.

We now fix an arbitrary $l=1,2,\dots$. For any $\mu,\nu\in P_+(\gk)$ admissible at level $l$, we shall show that $\dim\mathcal V^l_\gk{\nu\choose\vartheta_1~\mu}=\dim\Hom_\gk(\vartheta_1\otimes\mu,\nu)$, and that for any $\mathcal Y\in\mathcal V^1_\gk{\nu\choose\vartheta_1~\mu}$ and $u^{(\vartheta_1)}\in L_\gk(\vartheta_1)$, $\mathcal Y(u^{(\vartheta_1)},x)$ satisfies $1$-st order energy bounds. Note that when $(\mu|\vartheta_2)<l$ or $(\nu|\vartheta_2)<l$, this result follows from proposition \ref{lb29} and the energy bounds condition for level $1$ creation operators. Therefore we may well assume that $(\mu|\vartheta_2)=(\nu|\vartheta_2)=l$. This means that on the tensor graph \ref{fig2} $\mu$ and $\nu$ has the same height $l$ (if we assume that the height of $\vartheta_1$ is 1).

Now we assume that $\dim\Hom_\gk(\vartheta_1\otimes\mu,\nu)=1$. Then on the tensor graph, $\mu$ and $\nu$ are connected by either a horizontal line segment or a loop. (In the later case $\mu=\nu$.) Using the tensor graph, it is not hard to find $a\in\mathbb Z_{>0}$ and $\rho,\mu_1,\nu_1\in P_+(\gk)$, such that $a\leq l$, $\rho$ is admissible at level $l-a$, $\mu_1,\nu_1$ are admissible at level $a$, $\mu=\rho+\mu_1$, $\nu=\rho+\nu_1$, and one of the following three cases holds: (a) $a=1,\mu_1=\nu_1=\vartheta_1$; (b) $a=2,\mu_1=2\vartheta_1,\nu_1=\vartheta_2$; (c) $a=2,\mu_1=\vartheta_2,\nu_1=2\vartheta_1$. Therefore, to prove the result for intertwining operators of general types, it suffices, by lemma \ref{lb24}-(a), to prove this result in the following three special  cases:\\
(I) $l=1,\mu=\nu=\vartheta_1$.\\
(II) $l=2,\mu=2\vartheta_1,\nu=\vartheta_2$.\\ 
(III) $l=2,\mu=\vartheta_2,\nu=2\vartheta_1$.\\
Since $\vartheta_1=\overline{\vartheta_1}$, case (III) will follow from case (II) and proposition \ref{lb14}. Thus we only need to prove cases (I) and (II).

\subsubsection*{Case (I)}

To prove this case we need the following lemma.
\begin{lm}\label{lb30}
There exsit $\vartheta_1$-highest weight vectors $v_{\vartheta_1}^1,v_{\vartheta_1}^2,v_{\vartheta_1}^3\in (\gk_2\oplus\fk_4)^\perp$ satisfying $([\gk_2v_{\vartheta_1}^1,\gk_2v_{\vartheta_1}^2]|\gk_2v_{\vartheta_1}^3)\neq0$.
\end{lm}
\begin{proof}
Let $M=(\gk_2\oplus\fk_4)^\perp$. Recall that the raising operators $P_1,\cdots,P_8$ of $\ek_8$ have unit length. Choose $X=P_2,Y=P_4-P_6\in\ek_8$. Then $[X,Y]=P_{24}-P_{26}$. By comparing them with \eqref{eq66} and \eqref{eq67}, one can easily check that $X,Y,[X,Y]\in M$. This proves $[M,M]\cap M\neq0$, and hence $([M,M]|M)\neq0$. By theorem \ref{lb28}, the representation $\gk_2\curvearrowright M$ has the orthogonal irreducible decomposition $M=\bigoplus_{i=1}^{26}M_i$, where each $M_i$ is unitarily equivalent to $L_{\gk_2}(\vartheta_1)$. So there exist $i,j,k=1,\dots,26$ such that $([M_i,M_j]|M_k)$. Choose highest weight vectors $v_{\vartheta_1}^1,v_{\vartheta_1}^2,v_{\vartheta_1}^3$ of $M_i,M_j,M_k$ respectively. Then $([\gk_2v_{\vartheta_1}^1,\gk_2v_{\vartheta_1}^2]|\gk_2v_{\vartheta_1}^3)\neq0$.
\end{proof}

Let again $Y_\pk$ be the vertex operator for the vacuum module of $V^l_\pk$. Choose $v_{\vartheta_1}^1,v_{\vartheta_1}^2,v_{\vartheta_1}^3\in (\gk_2\oplus\fk_4)^\perp$ as in lemma \ref{lb30}. From the proof of lemma \ref{lb27}, $v_{\vartheta_1}^1(-1)\Omega,v_{\vartheta_1}^2(-1)\Omega,v_{\vartheta_1}^3(-1)\Omega$ are level $1$ $\vartheta_1$-highest weight vectors of $\gh\curvearrowright V^1_\pk$. Since $\gk v_{\vartheta_1}^1(-1)\Omega\subset V^1_\pk(1)$, by theorem \ref{lb18},  $Y_\pk(u,x)$ satisfies $1$-st order energy bounds for any $u\in \gk v_{\vartheta_1}^1(-1)\Omega$. By lemma \ref{lb30}, we can choose $u_1\in \gk v_{\vartheta_1}^1,u_2\in \gk v_{\vartheta_1}^2,u_3\in \gk v_{\vartheta_1}^3$ satisfying $([u_1,u_2]|u_3)\neq0$. Hence
\begin{align*}
&\bk{ Y(u_1(-1)\Omega,0)u_2(-1)\Omega|u_3(-1)\Omega }=\bk{u_1(0)u_2(-1)\Omega|u_3(-1)\Omega}\\
=&\bk{[u_1(0),u_2(-1)]\Omega|u_3(-1)\Omega}=\bk{[u_1,u_2](-1)\Omega|u_3(-1)\Omega}=\bk{u_3(-1)^*[u_1,u_2](-1)\Omega|\Omega}\\
=&\bk{u_3^*(1)[u_1,u_2](-1)\Omega|\Omega}=\bk{[u_3^*,[u_1,u_2]](0)\Omega|\Omega}+(u_3^*|[u_1,u_2]^*)\bk{\id\Omega|\Omega}\\
=&([u_1,u_2]|u_3)\neq0.
\end{align*}
We can therefore use corollary \ref{lb13} to prove the desired result.

\subsubsection*{Case (II)}

To prove this case we also need a lemma.
\begin{lm}\label{lb32}
Choose a non-zero $T\in\Hom_{\gk_2}(\vartheta_1\otimes\vartheta_1,\vartheta_1)$. Let $u_1$ be a non-zero element in $L_{\gk_2}(\vartheta_1)[-\alpha_1]$, and let $u_2\in L_{\gk_2}(\vartheta_1)[\vartheta_1]$ be a highest weight vector of $L_{\gk_2}(\vartheta_1)$. Then $T(u_1\otimes u_2)\neq0$.
\end{lm}
\begin{proof}
Our notation is as in theorem \ref{lb31}, where $A_1$ is a non-zero element in $(\gk_2)_{\alpha_1}$. Then $A_1u_2=0$. By analyzing the $\mathfrak{sl}_2$-module  $L_{\gk_2}(\vartheta_1)[\alpha_1]\oplus L_{\gk_2}(\vartheta_1)[0]\oplus L_{\gk_2}(\vartheta_1)[-\alpha_1]$, it is  easy to see that $A_1u_1\neq0$. Therefore
\begin{align}\label{eq68}
\bk{T(u_1\otimes u_2)|A_1^*u_2}=-\bk{T(A_1u_1\otimes u_2|u_2)}.
\end{align}
Now $A_1u_1$ is a non-zero element in $L_{\gk_2}(\vartheta_1)[0]$. Hence, by the injectivity of $\Gamma:\Hom_{\gk_2}(\vartheta_1\otimes\vartheta_1,\vartheta_1)\rightarrow L_{\gk_2}(\vartheta_1)[0]$, the right hand side of equation \eqref{eq68} is non-zero. Therefore $T(u_1\otimes u_2)\neq0$.
\end{proof}

Now for case (II), we let $a=1,\rho=\vartheta_1,\mu_1=\nu_1=\vartheta_1$. Then $\mu=\rho+\mu_1,L_\gk(\nu)\subset L_\gk(\rho)\otimes L_\gk(\nu_1)$. By lemma \ref{lb32}, condition \eqref{eq44} is satisfied. By case (I), $\dim\mathcal V^a_\gk{\nu_1\choose\vartheta_1~\mu_1}=\dim\mathcal V^1_\gk{\vartheta_1\choose\vartheta_1~\vartheta_1}=1$, and for any $\mathcal Y\in \mathcal V^a_\gk{\nu_1\choose\vartheta_1~\mu_1}$ and $u\in L_\gk(\vartheta_1)$, $\mathcal Y(u,x)$ satisfies $1$-st order energy bounds. Therefore case (II) can be proved using lemma \ref{lb24}-(b). 

Now, since we know the fusion rules of  $L_\gk(\vartheta_1)$ with another irreducible $V^l_\gk$-module, we can  deduce that $L_\gk(\vartheta_1,l)$ is generating. The main result of this chapter is thus proved.

\begin{thm}\label{lb46}
	Let $\gk=\gk_2$, and choose $\mu,\nu\in P_+(\gk)$ admissible at level $l$.
	
	(a) The map $\Psi:\mathcal V^l_\gk{\nu\choose\vartheta_1~\mu}\rightarrow \Hom_\gk(\vartheta_1\otimes\mu,\nu)$ is bijective. The dimension of these two isomorphic spaces can be read from figure \ref{fig2}.
	
	(b) $L_\gk(\vartheta_1)$ is a generating $V^l_\gk$-module.
	
	(c) For any $\mathcal Y\in\mathcal V^l_\gk{\nu\choose\vartheta_1~\mu}$ and $u^{(\vartheta_1)}\in L_\gk(\vartheta_1)$, $\mathcal Y(u^{(\vartheta_1)},x)$ satisfies $1$-st order energy bounds. In particular, $\mathcal Y$ is energy-bounded.
\end{thm}

\section{Concluding remarks}\label{lb38}

Assume that $V$ is a unitary rational VOA whose representation category exists a modular tensor category in the sense of \cite{Hua08}. We've defined, in \cite{Gui19,Gui17}, a sesquilinear form $\Lambda$ on each dual vector space of intertwining operators of $V$, which is automatically non-degenerate. To obtain a unitary tensor product theory, one needs to prove the positivity of $\Lambda$. In those two papers, we gave two criteria on the positivity of $\Lambda$, both concerning the energy-bounds condition for intertwining operators and vertex operators. See conditions A and B in \cite{Gui17} section 5.3. 

Now let $V$ be a unitary affine VOA $V^l_\gk$. If  there exists a generating set $\mathcal F$ of irreducible $V^l_\gk$-modules, such that any (irreducible) intertwining operator of $V^l_\gk$ whose charge space lies   inside $\mathcal F$ is energy-bounded, then condition A is satisfied, and hence, by \cite{Gui17} theorem 7.8, $\Lambda$ is positive, and  the modular tensor category is unitary. Now by the results of \cite{Was98} and \cite{TL04}, and by the main results we've obtained in this paper (see theorems \ref{lb16}, \ref{lb44}, \ref{lb43}, \ref{lb45}, \ref{lb46}), $V^l_\gk$ satisfies such requirement when $\gk$ is of type $A,B,C,D,G_2$. Therefore we have

\begin{thm}\label{lb49}
Let $\gk$ be a (unitary) complex simple Lie algebra of type $A$, $B$, $C$, $D$, or $G_2$. Choose $l=0,1,2,\dots$. Then on any dual vector space of intertwining operators of $V^l_\gk$, the sesquilinear form $\Lambda$ defined in \cite{Gui17} section 6.2 is positive-definite. Moreover $\Lambda$ induces a unitary structure on the modular tensor category of $V^l_\gk$.
\end{thm}

The remaining unsolved cases are type $E$ and $F_4$. Note that the type $E$ level $1$ cases are known due to corollary \ref{lb41}. The methods we have used for $G_2$ in this article, in particular, the embedding $\gk_2+\fk_4\subset\ek_8$, might also be helpful in the $F_4$ level $1$ case. However, the analysis might be more complicated due to the increase of the dimension. For type $E_6,E_7,F_4$, it is also possible, though much more complicated, to reduce the higher level cases to the level 1 ones using compression principle. This method, however, does not work for $E_8$, as the representation theory of type $E_8$ level $1$ is trivial. The smallest non-trivial situation for $E_8$ is when the level equals $2$, in which case there will be fusion rules greater than $1$. This phenomenon  tremendously increases the difficulty of analysis, so one needs to look for other methods to work with this case.

\begin{appendices}
	
	\section{Heisenberg and lattice VOAs}\label{lb34}
	
	In this appendix chapter we discuss the energy-bounds conditions for the intertwining operators of Heisenberg VOAs and even lattice VOAs.
	
	\subsubsection*{Unitary Heisenberg VOAs}
	
	Let $\hk$ be a unitary abelian Lie algebra of dimension $n$. In particular, an inner product $(\cdot|\cdot)$ and an anti-unitary involution $*$ is chosen. We let $\hk_{\mathbb R}=\{X\in\hk:X^*=-X \}$. The affine Lie algebra $\hh$ is defined in much the same way as we define $\gh$ when $\gk$ is simple. The irreducible PER theory of $\hh$ is also similar to $\gh$, except that for any $\lambda\in (i\hk_{\mathbb R})^*$ and $l>0$, there exists a unique (up to equivalence) unitary $\hh$-module $L_\hk(\lambda,l)$, which is also  isomorphic to $L_\hk(\lambda,1)$. So we can always assume that the level $l=1$.  Similar to unitary affine VOAs, there exists a unitary VOA structure $V^1_\hk$ on $L_\hk(0,1)$ of CFT type, and equation \eqref{eq25} gives the conformal vector of $V^1_\hk$ by setting $h^\vee=0$ (and of course, $l=1$). $V^1_\hk$ is called the \textbf{unitary Heisenberg VOA} associated associated to $\hk$. Any unitary irreducible $V^1_\hk$-module  can be restricted to the unitary $\hh$-module $L_\hk(\lambda,1)$ for some $\lambda\in (\ih)^*$. The conformal weight of $L_\hk(\lambda,1)$ is $(\lambda|\lambda)/2$. Conversely, for any $\lambda\in (\ih)^*$, the irreducible PER $L_\hk(\lambda,1)$ of $\hh$ can be extended to a unitary $V^1_\hk$-module.  See \cite{LL04} for more details. Unfortunately, there are uncountably many irreducible $V^1_\hk$-modules. So $V^1_\hk$ is not a rational VOA.
	
	The inner product on $\hk$ identifies the dual space $\hk^*$ of $\hk$ as the conjugate of $\hk$, which in turn identifies the $\mathbb R$-vector space $\ih$ with its (real) dual space $(\ih)^*$. We assume this identification in the following. One of the most important problems in the theory of Heisenberg VOAs and lattice VOAs was to determine the intertwining operators of $V^1_\hk$, as well as their braid and fusion relations. This was mainly achieved in \cite{DL93}. See also \cite{TZ11} for a brief exposition of this theory. The result can be summarized as below. For any $\alpha,\beta,\gamma\in\ih$, if $\gamma=\alpha+\beta$ then the fusion rule $N^\gamma_{\alpha\beta}=\dim\mathcal V^1_\hk{\gamma\choose\alpha~\beta}=1$. Otherwise $N^\gamma_{\alpha\beta}=0$. Intertwining operators of type $\alpha+\beta\choose\alpha~\beta$ can be described as follows. First, we fix for any $\lambda\in\ih$ a unital highest weight vector $v_\lambda$ in $L_\hk(\lambda,1)$. We choose $v_0$ to be the vacuum vector $\Omega\in V^1_\hk=L_\hk(0,1)$. Then for any $\mu\in\ih$, one can easily define a unitary map $c_\lambda:L_\hk(\mu,1)\rightarrow L_\hk(\lambda+\mu,1)$ such that
	\begin{align}
	c_\lambda\cdot X_1(-n_1)\cdots X_k(-n_k)v_\mu=X_1(-n_1)\cdots X_k(-n_k)v_{\lambda+\mu}
	\end{align}
	for any $X_1,\dots,X_k\in\hk,n_1,\dots,n_k\in\mathbb Z_{>0}$. It follows that $[X(n),c_\lambda]=\delta_{n,0}(X|\lambda)c_\lambda$ for any $X\in\hk$ and $n\in\mathbb Z$. We also have $v_\lambda=c_\lambda\Omega$. Let $x^{\lambda(0)}$ be an operator-valued formal series of $x$, so that $x^{\lambda(0)}= x^{(\lambda|\mu)}\cdot\id$ when acting on $L_\hk(\mu,1)$. We also defined operator-valued formal power series  
	\begin{align}
	E^\pm(\lambda,x)=\exp(\mp\sum_{n\in\mathbb Z_{>0}}\frac{\lambda(\pm n)}{n}x^{\mp n})
	\end{align}
acting on each $L_\hk(\mu,1)$. Then  a standard (non-zero) type $\alpha+\beta\choose\alpha~\beta$ intertwining operator $\mathcal Y^{\alpha+\beta}_{\alpha,\beta}$ can be defined to be
\begin{align}
\mathcal Y^{\alpha+\beta}_{\alpha,\beta}(c_\alpha u,x)=c_\alpha E^-(\alpha,x)Y(u,x)E^+(\alpha,x)x^{\alpha(0)},
\end{align}
where $u\in L_\hk(0,1)=V^1_\hk$, and $Y(u,x)$ is the vertex operator of $V^1_\hk$. To simplify our notations, we let $\mathcal Y_\alpha(c_\alpha u,x)$ act on any $L_\hk(\mu,1)$ (where $\mu\in\ih$) to be $\mathcal Y^{\alpha+\mu}_{\alpha,\mu}(c_\alpha u,x)$. Now we can let $x$ be a complex variable $z$ in $\mathbb C^\times=\mathbb C\setminus\{0\}$. Then $\mathcal Y_\alpha (c_\alpha u,z)$ is a multivalued function of $z$ depending on its argument $\arg z$.  The following fusion and braid relations are due to \cite{DL93} theorem 5.1. 
\begin{thm}\label{lb35}
Choose  $\alpha,\beta,\gamma\in\ih,u,v\in V^1_\hk,z_1,z_2\in\mathbb C^\times=\mathbb C\setminus\{0\}$. Then the following fusion and braid relations hold when acting on $L_\hk(\gamma,1)$. 

(a) If $0<|z_1-z_2|<|z_2|<|z_1|$ and $\arg z_1=\arg z_2=\arg(z_1-z_2)$, then
\begin{align}
\mathcal Y_\alpha(c_\alpha u,z_1)\mathcal Y_\beta(c_\beta v,z_2)=\mathcal Y_{\alpha+\beta}\big(\mathcal Y_\alpha(c_\alpha u,z_1-z_2)c_\beta v,z_2 \big).\label{eq69}
\end{align}

(b) If $|z_1|=|z_2|=1$ and $\arg z_2<\arg z_1<\arg z_2+2\pi$, then
\begin{align}
\mathcal Y_\alpha(c_\alpha u,z_1)\mathcal Y_\beta(c_\beta v,z_2)=e^{i\pi(\alpha|\beta)}\mathcal Y_\beta(c_\beta v,z_2)\mathcal Y_\alpha(c_\alpha u,z_1),\label{eq70}
\end{align}
where the left hand side is defined by first replacing $z_1$ with $rz_1$ where $r>1$, then taking the product of the two intertwining operators, and finally letting $r\rightarrow 1$; the right hand side is defined similarly, except that we replace $z_2$ with $rz_2$.
\end{thm}

We now  discuss adjoint intertwining operators. First, let $\Theta$ be the PCT operator of $V^1_\hk$. For any $\alpha\in\ih$ we define an anti-unitary map $-: L_\hk(\alpha,1)\rightarrow L_\hk(-\alpha,1)$ satisfying $\overline {c_\alpha u }=c_{-\alpha}\Theta u$ for any $u\in V^1_\hk$. So more explicitly, for any $X_1,\dots,X_k\in \hk,n_1,\dots,n_k\in\mathbb Z$ we have
\begin{align}
\overline {X_1(n_1)\cdots X_k(n_k)v_\alpha}=(-1)^kX_1^*(n_1)\cdots X_k^*(n_k)v_{-\alpha}.
\end{align}	
It is easy to check that $-$ induces a unitary $V^1_\hk$-module equivalence between $L_\hk(-\alpha,1)$ and the contragredient module $L_\hk(\overline\alpha,1)$  of $L_\hk(\alpha,1)$. Now one can easily compute that the adjoint intertwining operator $(\mathcal Y^{\alpha+\beta}_{\alpha,\beta})^\dagger$ of $\mathcal Y^{\alpha+\beta}_{\alpha,\beta}$ is\footnote{An easy way to see this  relation is to verify it for the vector $v_\alpha=c_\alpha\Omega$. This is sufficient since $(\mathcal Y^{\alpha+\beta}_{\alpha,\beta})^\dagger$ must be a type $\beta\choose-\alpha~\alpha+\beta$ intertwining operator, which is therefore proportional to $\mathcal Y^\beta_{-\alpha,\alpha+\beta}$.}
\begin{align}
(\mathcal Y^{\alpha+\beta}_{\alpha,\beta})^\dagger= e^{i\pi\cdot\frac {(\alpha|\alpha)}2}\mathcal Y^\beta_{-\alpha,\alpha+\beta}.\label{eq76}
\end{align}

Theorem \ref{lb18} clearly holds with $\gk$ replaced by $\hk$, i.e., for any $\lambda\in \ih$, $L_\hk(\lambda,1)$ is energy-bounded, and for any $X\in\hk$, $Y_\lambda(X(-1)\Omega,x)$ satisfies $1$-st order energy bounds.  The energy bounds condition of intertwining operators was established in \cite{TL04}.

\begin{thm}\label{lb37}
Choose any $\alpha,\beta\in\ih$.

(a) If $(\alpha|\alpha)\leq 1$, then $\mathcal Y_\alpha(c_\alpha\Omega,x)$ satisfies $0$-th order energy bounds.

(b) If $(\alpha|\alpha)\leq 2$, then $\mathcal Y_\alpha(c_\alpha\Omega,x)$ satisfies $1$-st order energy bounds.

(c) For general $\alpha$ and any $u\in V^1_\hk$, $\mathcal Y_\alpha(c_\alpha u,x)$ is energy-bounded.
\end{thm}
The proof of this theorem can be sketched as follows. First, if $(\alpha|\alpha)=1$, the operator $Y_\alpha(x):=E^-(\alpha,x)E^+(\alpha,x)$ and its adjoint satisfy free-fermion like anti-commuting relation: Write $Y_\alpha(x)=\sum_{n\in\mathbb Z}Y_\alpha(n)x^{-n}$, then $$Y_\alpha(n)Y_\alpha(m)^\dagger+Y_\alpha(m-1)^\dagger Y_\alpha(n-1)=\delta_{n,m}.$$ From this we can deduce the $0$-th order energy bounds. See \cite{TL04} chapter VI proposition 1.2.1 for more details.\footnote{Another way to see this is to note that the lattice vertex superalgebra for the integral lattice $\mathbb Z$ is the $1$-dimensional free fermion vertex superalgebra. Likewise, when $(\alpha|\alpha)=2$, one has the $1$-st order energy bounds condition since the lattice VOA for $\sqrt 2\mathbb Z$ is equivalent to the affine VOA $V^1_{\mathfrak{sl}_2}$.} (Note that it doesn't matter which $L_\hk(\mu,1)$ the operator $Y_\alpha(x)$ is acting on, since $Y_\alpha(x)$ commutes with the unitary operator $c_\beta$ for any $\beta\in\ih$.) If $(\alpha|\alpha)=2$, then a similar argument shows that $Y_\alpha(x)$ and its adjoint satisfy an affine-Lie-algebra like relation. So the $1$-st order energy bounds can be established using a similar argument as in theorem \ref{lb18}. Now suppose that $(\alpha|\alpha)<1$ or $<2$. Identify $\hk$ with $\hk\oplus 0$ in $\hk^1=\hk\oplus \hk$. Now choose $\beta\in 0\oplus\ih$ so that $(\beta|\beta)+(\alpha|\alpha)=1$ or $2$, and let $\gamma=\alpha+\beta\in i\hk^1_{\mathbb R}$. Then $(\gamma|\gamma)=1$ or $2$, and $Y_\gamma(x)=Y_\alpha(x)\otimes Y_\beta(x)$. Then the $0$-th order or $1$-st order energy bounds condition of $Y_\alpha(x)$ follows from that of $Y_\gamma(x)$. Thus (a) and (b) are proved.\footnote{This tensor product trick is also due to \cite{TL04} proposition VI.1.2.1.} Note that conversely, if we know the energy-bounds condition of $Y_\alpha(x)$ and $Y_\beta(x)$, then we can also easily know the energy bounds condition of $Y_\gamma(x)$. Thus by induction and (a), one can easily prove (c) when $u=\Omega$, and hence for any $u$ due to proposition \ref{lb17}. Hence this theorem is proved.

\subsubsection*{Even lattice VOAs}	

In some sense, the relation between lattice VOAs and Heisenberg VOAs is similar to that between simple Lie algebras and their Cartan subalgebras. Let us choose a non-degenerate even lattice $\Lambda$ in $\ih$.  Here ``non-degenerate" means that the rank of $\Lambda$ equals the dimension of $\ih$, and ``even" means that $(\alpha|\alpha)\in 2\mathbb Z$ for any $\alpha\in \Lambda$. In particular $(\Lambda|\Lambda)\subset\mathbb Z$. The reason we require evenness is to ensure that the conformal weights of $V^1_\hk$-modules are integers.

We now extend the action of $V^1_\hk$ on  $V_\Lambda=\bigoplus_{\alpha\in\Lambda} L_\hk(\alpha,1)$ to a VOA structure. The vacuum vector and the conformal vector of $V_\Lambda$ equals those of $L_\hk(0,1)=V^1_\hk$. For any $\alpha\in\Lambda$ and $u\in V^1_\hk$, we define the vertex operator $Y^\Lambda(c_\alpha u,x)$ of $V_\Lambda$ to be 
\begin{align}
Y^\Lambda(c_\alpha u,x)w^{(\mu)}=\epsilon(\alpha,\mu)\mathcal Y_\alpha(c_\alpha u,x)w^{(\mu)}\label{eq77}
\end{align}
for any $w^{(\mu)}\in L_\hk(\mu,1)$ where $\mu\in\Lambda$. Here $\epsilon(\alpha,\mu)\in S^1\subset\mathbb C$ will be determined shortly. $Y^\Lambda$ satisfies translation property since each $\mathcal Y_\alpha$ does. So in order for $V_\Lambda$ to be a VOA, the sufficient and necessary condition is that $Y^\Lambda$  satisfies creation property and Jacobi identity. Creation property is equivalent to
\begin{align}
\epsilon(\alpha,0)=1\qquad(\forall\alpha\in\Lambda).\label{eq74}
\end{align}
The Jacobi identity on $Y^\Lambda$ is equivalent to the following associativity and commutativity properties
\begin{gather}
Y^\Lambda(c_\alpha u,z_1)Y^\Lambda(c_\beta v,z_2)= Y^\Lambda\big(Y^\Lambda(c_\alpha u,z_1-z_2)c_\beta v,z_2 \big),\label{eq71}\\
Y^\Lambda(c_\alpha u,z_1)Y^\Lambda(c_\beta v,z_2)=Y^\Lambda(c_\beta v,z_2)Y^\Lambda(c_\alpha u,z_1),\label{eq72}
\end{gather}
which are understood in a similar way as in theorem \ref{lb35}. By the fusion relation \eqref{eq69}, equation \eqref{eq71} is equivalent to 
\begin{align}
\epsilon(\alpha,\beta+\gamma)\epsilon(\beta,\gamma)=\epsilon(\alpha,\beta)\epsilon(\alpha+\beta,\gamma)\qquad(\forall\alpha,\beta,\gamma\in\Lambda),\label{eq73}
\end{align}
i.e., $\epsilon$ is a $S^1$-valued cocycle on $\Lambda$. If we assume \eqref{eq73}, then, by the braid relation \ref{eq70}, the commutativity \eqref{eq72} is equivalent to that
\begin{align}
\epsilon(\alpha,\beta)=(-1)^{(\alpha|\beta)}\epsilon (\beta,\alpha).\label{eq75}
\end{align}
The existence of $\epsilon$ satisfying \eqref{eq74}, \eqref{eq73}, and \eqref{eq75} is guaranteed by the following well known proposition (cf. \cite{FLM89} proposition 5.2.3 or \cite{Kac98} lemma 5.5):
\begin{pp}\label{lb36}
If $\Lambda\simeq \mathbb Z^n$, $(A,\cdot)$ is an abelian group, and  $\omega:\Lambda\times\Lambda\rightarrow A$ satisfies $\omega(\alpha+\beta,\gamma)=\omega(\alpha,\gamma)\omega(\beta,\gamma)$ and $\omega(\alpha,\beta)=\omega(\beta,\alpha)^{-1}$ for any $\alpha,\beta,\gamma$, then there is a $A$-valued cocycle $\epsilon$ on $\Lambda$ such that $\epsilon(\alpha,\beta)=\omega(\alpha,\beta)\epsilon(\beta,\alpha)$, i.e., the commutator of $\epsilon$ is $\omega$. Moreover, if $b:\Lambda\times\Lambda\rightarrow A$ is a coboundary, i.e., $b(\alpha,\beta)=f(\alpha)f(\beta)f(\alpha+\beta)^{-1}$ for some function $f:\Lambda\rightarrow A$, then $\epsilon'=\epsilon b$ is also a cocycle with commutant $\omega$.
\end{pp}

Hence, once such $\epsilon$ is chosen, we have a VOA $V_\Lambda$ with vertex operator $Y^\Lambda$. Define an anti-unitary map $\Theta^\Lambda:V_\Lambda\rightarrow V_\Lambda$ such that for any $\alpha\in\Lambda,u\in V^1_\hk$,
\begin{align}
\Theta^\Lambda c_\alpha u=(-1)^{\frac {(\alpha|\alpha)}2}\epsilon(-\alpha,\alpha)^{-1}\overline{c_\alpha u}=(-1)^{\frac {(\alpha|\alpha)}2}\epsilon(-\alpha,\alpha)^{-1}c_{-\alpha} \Theta u.
\end{align}
Then using \eqref{eq74}, \eqref{eq73}, and \eqref{eq75}, one can check that $\Theta^\Lambda$ is an anti-automorphism. Using relation \eqref{eq76}, one arrives at the following (see also \cite{DL14} theorem 4.12):
\begin{thm} 
$(V_\Lambda,Y^\Lambda)$ is a unitary VOA of CFT type with PCT operator $\Theta^\Lambda$. 
\end{thm}
Note that if we choose $f,b,\epsilon'$ as in proposition \ref{lb36}, and assume that $f(0)=1$, then the unitary VOA $V_\Lambda$ defined by $\epsilon$ is clearly equivalent to the one defined by $\epsilon'$, with the unitary equivalence map defined by $f$.

Let $\Lambda^\circ$ be the dual lattice of $\Lambda$, i.e., the lattice in $\ih$ generated by the dual basis of a basis of $\Lambda$. So $\Lambda\subset\Lambda^\circ$, and we have quotient map $[\cdot]:\Lambda^\circ\rightarrow \Lambda^\circ/\Lambda$, $\lambda\mapsto [\lambda]$. Choose a cocycle $\epsilon$ on $\Lo$ whose  commutator $\omega$ satisfies $\omega(\alpha,\beta)=(-1)^{(\alpha|\beta)}$ for any $\alpha,\beta\in\Lambda$.  As above, we let $W_{\Lo}=\bigoplus_{\lambda\in\Lo}L_\hk(\lambda,1)$, and let $V^1_\Lambda$ act on $W_{\Lo}$ with vertex operator $Y^\Lambda$ defined by \eqref{eq77} for any $\mu\in\Lo$. Then, using a similar argument as above, one can show that $V_\Lo$ is a unitary $V_\Lambda$-module. For any $\lambda\in\Lo$, $W_{[\lambda]}=\bigoplus_{\alpha\in\Lambda}L_\hk(\lambda+\alpha,1)$ is clearly an irreducible unitary $V_\Lambda$-submodule of $W_\Lo$, and any irreducible unitary $V_\Lambda$-module is of this form. (See also \cite{DL14} theorem 4.12.)

We now describe intertwining operators of $V_\Lambda$. If $\lambda_0,\mu_0,\nu_0\in\Lo$, then any type $[\nu_0]\choose[\lambda_0]~[\nu_0]$ intertwining operator of $V_\Lambda$ is clearly also an intertwining operator of $V^1_\hk$. From this we can easily see that there are not  type $[\nu_0]\choose[\lambda_0]~[\mu_0]$ intertwining operators when $\nu_0-\lambda_0-\mu_0\notin\Lambda_0$. Now we look at type $[\lambda_0+\mu_0]\choose[\lambda_0]~[\mu_0]$ intertwining operators. If $\mathfrak Y\in\mathcal V{[\lambda_0+\mu_0]\choose[\lambda_0]~[\mu_0]}$, then, since it is an intertwining operator of $V^1_\hk$, there exists $\kappa:(\lambda_0+\Lambda)\times(\mu_0+\Lambda)\rightarrow \mathbb C$ satisfying that for any $\lambda\in\lambda_0+\Lambda,\mu\in\mu_0+\Lambda,w^{(\mu)}\in L_\hk(\mu,1),v\in V^1_\hk$,
\begin{align}
\mathfrak Y(c_\lambda v,x)w^{(\mu)}=\kappa(\lambda,\mu)\mathcal Y_\lambda(c_\alpha v,x)w^{(\mu)}.
\end{align}
$\mathfrak Y$ should satisfy Jacobi identity, which says that for any $\alpha\in\Lambda,u\in V^1_\hk$,
\begin{gather}
Y^\Lambda(c_\alpha u,z_1)\mathfrak Y(c_\lambda v,z_2)= \mathfrak Y\big(Y^\Lambda(c_\alpha u,z_1-z_2)c_\lambda v,z_2 \big),\label{eq78}\\
Y^\Lambda(c_\alpha u,z_1)\mathfrak Y(c_\lambda v,z_2)=\mathfrak Y(c_\lambda v,z_2)Y^\Lambda(c_\alpha u,z_1).\label{eq79}
\end{gather}
By the fusion and braid relations \eqref{eq69}, \eqref{eq70}, these two relations are equivalent to the  condition that for any $\alpha\in\Lambda,\lambda\in\lambda_0+\Lambda,\mu\in\mu_0+\Lambda$,
\begin{gather}
\epsilon(\alpha,\lambda+\mu)\kappa(\lambda,\mu)=\epsilon(\alpha,\lambda)\kappa(\alpha+\lambda,\mu)=(-1)^{(\alpha|\lambda)}\kappa(\lambda,\alpha+\mu)\epsilon(\alpha,\mu).\label{eq80}
\end{gather}
From this it is easy to see that the value of $\kappa$ is uniquely determined by $\kappa(\lambda_0,\mu_0)$. Hence the dimension of $\mathcal V{[\lambda_0+\mu_0]\choose[\lambda_0]~[\mu_0]}$ is at most $1$. Conversely, if we define, for any $\lambda\in\lambda_0+\Lambda,\mu\in\mu_0+\Lambda$,
\begin{align}
\kappa(\lambda,\mu)=\epsilon(\lambda,\mu)\omega(\mu-\mu_0,\lambda)e^{i\pi(\mu-\mu_0|\lambda)}
\end{align}
(cf. \cite{TL04} chapter V equation (5.3.1)), then $\kappa$ satisfies condition \eqref{eq80}. Hence $\mathfrak Y$ satisfies Jacobi identity. Since $\mathfrak Y$ clearly satisfies translation property as it is an intertwining operator of $V^1_\hk$, $\mathfrak Y$ is a non-zero type type $[\lambda_0+\mu_0]\choose[\lambda_0]~[\mu_0]$ intertwining operator of $V_\Lambda$. We arrive at the following theorem proved in \cite{DL93} proposition 12.2:
\begin{thm}\label{lb40}
Choose $\lambda_0,\mu_0,\nu_0\in\Lo$. If $\nu_0-\lambda_0-\mu_0\in\Lambda$, then the fusion rule $N^{[\nu_0]}_{[\lambda_0][\mu_0]}=1$, and a non-zero  type $[\nu_0]\choose[\lambda_0]~[\mu_0]$ intertwining operator of $V_\Lambda$ can be defined such that for any $\lambda\in\lambda_0+\Lambda,\mu\in\mu_0+\Lambda,w^{(\mu)}\in L_\hk(\mu,1),v\in V^1_\hk$,
\begin{align}
\mathfrak Y(c_\lambda v,x)w^{(\mu)}=\epsilon(\lambda,\mu)\omega(\mu-\mu_0,\lambda)e^{i\pi(\mu-\mu_0|\lambda)}\mathcal Y_\lambda(c_\lambda v,x)w^{(\mu)}.
\end{align}
If $\nu_0-\lambda_0-\mu_0\notin\Lambda$, then  $N^{[\nu_0]}_{[\lambda_0][\mu_0]}=0$.
\end{thm}
This theorem, together with theorem \ref{lb37}, immediately implies
\begin{thm}
Let $\lambda_0\in\Lo$, and $\mathfrak Y$ an intertwining operator of $V_\Lambda$ with charge space $W_{[\lambda_0]}$, then $\mathfrak Y$ is energy-bounded. Moreover, for any $\lambda\in\lambda_0+\Lambda$,  $\mathfrak Y(c_\lambda \Omega,x)$ satisfies $0$-th order energy-bounds when $(\lambda|\lambda)\leq 1$, and satisfies $1$-st order energy-bounds when $(\lambda|\lambda)\leq 2$.
\end{thm}

\begin{co}[cf.\cite{TL04}]\label{lb41}
Let $\gk$ be a complex unitary simple Lie algebra of type $A,D,$ or $E$. The inner product on $\gk$ as well as on its Cartan subalgebra $\hk$ is chosen such that the length of the roots of $\gk$ are $\sqrt 2$.  Let $\Lambda$ be the root lattice of $\gk$, and choose dominant integral weights $\lambda,\mu,\nu$  of $\gk$ admissible at level $1$. If $\nu-\lambda-\mu\in\Lambda$, then $\dim\mathcal V^1_\gk{\nu\choose\lambda~\mu}=1$, and any type $\nu\choose\lambda~\mu$ intertwining operator of $V^1_\gk$ is energy bounded. Moreover, let $d$ be the distance between $0$ and $\lambda+\Lambda$. Choose any $\mathcal Y\in\mathcal V^1_\gk{\nu\choose\lambda~\mu},u^{(\lambda)}\in L_\gk(\lambda)$. If $d^2\leq 1$ then $\mathcal Y(u^{(\lambda)},x)$ satisfies $0$-th order energy bounds. If $d^2\leq 2$ then $\mathcal Y(u^{(\lambda)},x)$ satisfies $1$-th order energy bounds. 

If, however, $\nu-\lambda-\mu\notin\Lambda$, then $\dim\mathcal V^1_\gk{\nu\choose\lambda~\mu}=0$.
\end{co}

\end{appendices}


\begin{thebibliography}{99}
\small	


	
	
\bibitem [BS90]{BS90}
Buchholz, D. and Schulz-Mirbach, H., 1990. Haag duality in conformal quantum field theory. Reviews in Mathematical Physics, 2(01), pp.105-125.
	

	
	\bibitem [CKLW18]{CKLW18}
	Carpi, S., Kawahigashi, Y., Longo, R. and Weiner, M., 2018. From vertex operator algebras to conformal nets and back. Memoirs of the American Mathematical
	Society (Vol. 254, No. 1213).
	
\bibitem[DL93]{DL93}
Dong, C. and Lepowsky, J., 1993. Generalized vertex algebras and relative vertex operators (Vol. 112). Springer Science \& Business Media.
	
	\bibitem [DL14]{DL14}
	Dong, C. and Lin, X., 2014. Unitary vertex operator algebras. Journal of algebra, 397, pp.252-277.
	
\bibitem[DM06]{DM06}
Dong, C. and Mason, G., 2006. Integrability of $C_2$-cofinite vertex operator algebras. International Mathematics Research Notices, 2006.
	
	\bibitem [Dyn52]{Dyn52}
	Dynkin, E.B., 1952. Semisimple subalgebras of semisimple Lie algebras. Selected Papers of E.B.Dynkin with Commentary, AMS, 2000.
	
\bibitem[FH91]{FH91}
Fulton, W. and Harris, J., 1991. Representation theory: a first course (Vol. 129). Springer Science \& Business Media.

	
		\bibitem [FHL93]{FHL93}
	Frenkel, I., Huang, Y.Z. and Lepowsky, J., 1993. On axiomatic approaches to vertex operator algebras and modules (Vol. 494). American Mathematical Soc..
	
\bibitem [FK80]{FK80}
Frenkel, I.B. and Kac, V.G., 1980. Basic representations of affine Lie algebras and dual resonance models. Inventiones mathematicae, 62(1), pp.23-66.

\bibitem[FLM89]{FLM89}
Frenkel, I., Lepowsky, J. and Meurman, A., 1989. Vertex operator algebras and the Monster (Vol. 134). Academic press.
	
	\bibitem[FZ92]{FZ92}
	Frenkel, I.B. and Zhu, Y., 1992. Vertex operator algebras associated to representations of affine and Virasoro algebras. Duke Mathematical Journal, 66(1), pp.123-168.

\bibitem[GKO86]{GKO86}	
Goddard, P., Kent, A. and Olive, D., 1986. Unitary representations of the Virasoro and super-Virasoro algebras. Communications in Mathematical Physics, 103(1), pp.105-119.	

\bibitem[GW84]{GW84}
Goodman, R. and Wallach, N.R., 1984. Structure and unitary cocycle representations of loop groups and the group of diffeomorphisms of the circle. energy, 3, p.3.



\bibitem [Gui19]{Gui19}
	Gui, B., 2019. Unitarity of the modular tensor categories associated to unitary vertex operator algebras, I. Commun. Math. Phys. 366, 333-396.
	
	
	\bibitem [Gui17]{Gui17}
	Gui, B., 2017. Unitarity of the modular tensor categories associated to unitary vertex operator algebras, II. arXiv:1712.04931
	
\bibitem[Gui18]{Gui18}
Gui, B., 2018. Categorical extensions of conformal nets. arXiv preprint arXiv:1812.04470.
	
\bibitem[HK07]{HK07}
Huang, Y.Z. and Kong, L., 2007. Full field algebras. Communications in mathematical physics, 272(2), pp.345-396.
	
	\bibitem [Hua95]{Hua95}
	Huang, Y.Z., 1995. A theory of tensor products for module categories for a vertex operator algebra, IV. Journal of Pure and Applied Algebra, 100(1-3), pp.173-216.
	
\bibitem[Hua08]{Hua08}
Huang, Y.Z., 2008. Rigidity and modularity of vertex tensor categories. Communications in contemporary mathematics, 10(supp01), pp.871-911.

\bibitem[Jon83]{Jon83}
Jones, V.F., 1983. Index for subfactors. Inventiones mathematicae, 72(1), pp.1-25.

\bibitem[KL04]{KL04}
Kawahigashi, Y. and Longo, R., 2004. Classification of local conformal nets. Case $c<1$. Annals of mathematics, pp.493-522.

\bibitem[KLM01]{KLM01}
Kawahigashi, Y., Longo, R. and M\"uger, M., 2001. Multi-Interval Subfactors and Modularity of Representations in Conformal Field Theory. Communications in Mathematical Physics, 219(3), pp.631-669.

\bibitem[Kac94]{Kac94}
Kac, V.G., 1994. Infinite-dimensional Lie algebras (Vol. 44). Cambridge university press.

\bibitem[Kac98]{Kac98}
Kac, V.G., 1998. Vertex algebras for beginners (No. 10), 2nd ed. American Mathematical Soc..

\bibitem[LL04]{LL04}
Lepowsky, J. and Li, H., 2004. Introduction to vertex operator algebras and their representations (Vol. 227). Springer Science \& Business Media.

\bibitem[Loke94]{Loke94}
Loke, T.M., 1994. Operator algebras and conformal field theory of the discrete series representations of Diff ($S^1$) (Doctoral dissertation, University of Cambridge).

\bibitem[Lon89]{Lon89}
Longo, R., 1989. Index of subfactors and statistics of quantum fields. I. Communications in Mathematical Physics, 126(2), pp.217-247.

\bibitem[MS88]{MS88}
Moore, G. and Seiberg, N., 1988. Polynomial equations for rational conformal field theories. Physics Letters B, 212(4), pp.451-460.

\bibitem[TK88]{TK88}
Tsuchiya, A. and Kanie, Y., 1988. Vertex operators in conformal field theory on $P^1$ and monodromy representations of braid group. Conformal Field Theory and Solvable Lattice Models, Advanced Studies in Pure Math, 16, pp.297-372.

\bibitem[TL04]{TL04}
Toledano-Laredo, V., 2004. Fusion of positive energy representations of lspin (2n). arXiv preprint math/0409044.

\bibitem[TZ11]{TZ11}
Tuite, M.P. and Zuevsky, A., 2011. A generalized vertex operator algebra for Heisenberg intertwiners. arXiv preprint arXiv:1106.6149.

\bibitem[Tur94]{Tur94}
Turaev, V.G., 1994. Quantum invariants of knots and 3-manifolds (Vol. 18). Walter de
Gruyter GmbH \& Co KG.

\bibitem[Was90]{Was90}
Wassermann, A., 1990. Subfactors arising from positive energy representations of some infinite dimensional groups. unpublished notes.

\bibitem[Was98]{Was98}
Wassermann, A., 1998. Operator algebras and conformal field theory III. Fusion of positive energy representations of LSU (N) using bounded operators. Inventiones mathematicae, 133(3), pp.467-538.
	
	\bibitem[Was10]{Was10}
	Wassermann, A., 2010. Kac-Moody and Virasoro algebras. arXiv preprint arXiv:1004.1287.
	
\bibitem[Xu00]{Xu00}
Xu, F., 2000. Jones-Wassermann subfactors for disconnected intervals. Communications in Contemporary Mathematics, 2(03), pp.307-347.

	
\end{thebibliography}
\end{document}